\documentclass[a4paper,11pt]{article}
\usepackage[a4paper,left=2.5cm,right=2.5cm,top=3.cm,bottom=3.cm]{geometry}

\usepackage{amsmath,amsfonts}
\usepackage{algorithm}
\usepackage{algpseudocode}
\usepackage{array}
\usepackage[caption=false,font=normalsize,labelfont=sf,textfont=sf]{subfig}
\usepackage{textcomp}
\usepackage{stfloats}
\usepackage{url}
\usepackage{verbatim}
\usepackage{graphicx}
\usepackage{hyperref}

\newcommand{\R}{{\mathbb R}}

 \def\F{\mathbf{F}}
 \def\E{\mathbf{E}}

 \def\prox{\mathrm{prox}}

\def\argmin{\mathop{\mathrm{arg min}}}

\usepackage{mathrsfs}
\usepackage{xcolor}
\usepackage{enumerate}

\providecommand\argmin{\mathop{\mathrm{argmin}}}

\usepackage{amsthm}
\newtheorem{lemma}{Lemma}
\newtheorem{remark}{Remark}

\usepackage{authblk}

\begin{document}
\title{Deep Unfolding Network for Nonlinear Multi-Frequency\\ Electrical Impedance Tomography}
\author[1]{G.S. Alberti}
\author[2]{D. Lazzaro}
\author[2]{S. Morigi}
\author[2]{L. Ratti}
\author[1]{M. Santacesaria}
\affil[1]{\small MaLGa Center, Department of Mathematics, University of Genova, Via Dodecaneso 35, 16146 Genova, Italy}
\affil[2]{\small University of Bologna, Department of Mathematics, Piazza di Porta S. Donato 5, 40126 Bologna, Italy}
\date{ }


\maketitle

\begin{abstract}
Multi-frequency Electrical Impedance Tomography (mfEIT) represents a promising biomedical imaging modality that enables the estimation of tissue conductivities across a range of frequencies. Addressing this challenge, we present a novel variational network, a model-based learning paradigm that strategically merges the advantages and interpretability of classical iterative reconstruction with the power of deep learning. This approach integrates graph neural networks (GNNs) within the iterative Proximal Regularized Gauss Newton (PRGN) framework. By unrolling the PRGN algorithm, where each iteration corresponds to a network layer, we leverage the physical insights of nonlinear model fitting alongside the GNN's capacity to capture inter-frequency correlations. Notably, the GNN architecture preserves the irregular triangular mesh structure used in the solution of the nonlinear forward model, enabling accurate reconstruction of overlapping tissue fraction concentrations.\end{abstract}


\section{Introduction}
\label{sec:intro}

Electrical Impedance Tomography (EIT) is a non-invasive medical imaging technique that utilizes a series of electrodes placed on the body’s surface to inject low-amplitude alternating currents at a given frequency, and measure the resulting voltages. From these voltage differences, solving the inverse EIT problem enables the reconstruction of images that map the internal distribution of electrical impedance within the body \cite{adler-boyle-2017}. This impedance distribution reflects variations in tissue properties, such as conductivity and permittivity, which can differ between healthy and pathological tissues.
EIT is particularly valued for its portability, lack of ionizing radiation, and ability to provide real-time functional imaging, making it suitable for applications like monitoring lung ventilation \cite{frerichs2000}, detecting brain abnormalities \cite{TIDSWELL2001283}, or assessing breast tissue \cite{VCherepenin_2001},  and biological tissue imaging \cite{DEAN2008165}.

However, reconstructing the internal impedance from surface voltage measurements is a complex inverse problem \cite{cheney-isaacson-newell-1999,borcea-2002}. It is inherently nonlinear and severely ill-posed, meaning it is highly sensitive to modeling inaccuracies, measurement noise, and even slight errors in electrode placement. This also requires significant computational resources.  

\subsection{State of the art on multi-frequency EIT}

The bioimpedance characteristics of biological tissues, comprising both conductivity and permittivity components, exhibit frequency-dependent variations that correlate with tissue type and underlying physiological or pathological states \cite{DEAN2008165,widlak-scherzer-2012}. 
 
Multi-frequency EIT (mfEIT)  is an advanced extension of EIT that leverages the frequency-dependent electrical properties of biological tissues. 
This frequency-dependent behavior, known as the conductivity spectrum, enables mfEIT to differentiate tissues more effectively and detect subtle biological changes, providing richer diagnostic information.
By injecting currents at multiple frequencies and measuring the resulting voltages through electrodes placed on the body's surface, solving the mfEIT inverse problem allows for the reconstruction of detailed images of the internal impedance distribution. 

The availability of multi-frequency data allows for the use of Multiple Measurement Vector (MMV) approaches, which perform simultaneous reconstructions across all frequencies through multi-task optimization, generating multiple conductivity images concurrently. These methods capitalize on inter-frequency correlations by identifying shared structural features, enabling more accurate characterization of frequency-dependent tissue electrical properties.

These methods can be compared with Single Measurement Vector (SMV) approaches, 
which decompose mfEIT reconstruction into independent single-frequency problems, generating distinct images for each frequency. Although this strategy enables the application of well-established single-frame algorithms \cite{TanyuNingHauptmannJinMaass+2025+437+470, Lazzaro2024OracleNet}, it neglects inter-frequency correlations in the measurement data, potentially compromising the characterization of complex tissue properties.

Our motivating application comes from tissue engineering, and in particular from the study of osteogenic differentiation in cell cultures. In this process, the deposition of calcium and phosphate ions at a cellular scale leads to the formation of mineral deposits, which macroscopically alter the conductivity properties of the tissue. 

We therefore assume the fractional model paradigm \cite{Malone2014}: by averaging at a microscopic scale, we assume that several elementary tissue types coexist at each point of the domain of interest, and the relative concentration of each tissue is represented by space-dependent fractions.
Under the assumption that the frequency-dependent electrical conductivity spectrum (or admittivity spectrum) of each elementary tissue is empirically characterized in advance, or theoretically modeled, this leads to a paradigm shift in the considered inverse problem. 
Indeed, instead of trying to recover an arbitrary conductivity value for every point in the domain and every frequency, the fraction model targets the reconstruction of the fraction of each tissue type. At each domain point, such unknown fractions must satisfy non-negativity and unit sum properties. Finally, to guarantee both an accurate and efficient reconstruction, we assume that the microscopic relative concentrations of each tissue (hence, the unknown fractions) are piecewise constant on a predefined partition of the domain.

\subsection{Contributions}

In the context of the MMV-based methods which follow the fraction-model paradigm, our contribution is twofold.

We first propose a variational method, named Fractional Proximal Regularized Gauss-Newton (FR-PRGN), in which the Entropic Mirror Descent (EMDA) \cite{BT2003} algorithm plays a critical role in enforcing constraints on the tissue fractions. Specifically, EMDA is used in the proximal step of the PRGN to impose the non-negativity and sum-to-one constraints on the fractions. This is effectively implemented through a row-wise softmax operation, ensuring that the reconstructed fractions are always physically consistent with the model assumptions.

As a second main contribution, we introduce a novel variational network, called mf-Net, which unrolls the iterations of the FR-PRGN algorithm into a trainable deep network. At the heart of this learning framework lies the proximal operator. The output of the denoiser, implemented as a Graph Neural Network (GU-Net), is projected onto the constraint space via the softmax operation. This step guarantees that the estimated fractions remain physically plausible by satisfying the intrinsic constraints of the fraction model.

The mf-Net architecture embodies a model-based learning approach that combines the interpretability and robustness of classical iterative reconstruction with the expressive power of deep learning.

\subsection{Related works}

We discuss here some related works on mfEIT that are relevant for our approach.

Malone et al. \cite{Malone2014}  presented a method for mfEIT that leverages spectral constraints to reconstruct the volume fraction distribution of the constituent tissues using a nonlinear approach. This approach was extended in \cite{malone2015} by removing the prior knowledge of the spectral profiles.
Alberti et al.\ in  \cite{alberti-etal-2016} analyzed the linearized inverse problem in mfEIT, discussed reconstruction methods for known, partially known, or unknown spectral profiles, and demonstrated the effectiveness of the multifrequency approach in handling modeling errors, such as imperfectly known boundaries.

Battistel et al. in \cite{Battistel} proposed a mfEIT reconstruction algorithm that incorporates spectral correlation via a tensor product between the spatial variation and the spectral variation. The resulting inverse problem is solved using Tikhonov regularization.

Chen et al. in \cite{MMVNET2023} introduced MMV-Net, a model-based deep learning algorithm designed for the simultaneous reconstruction of mfEIT images. MMV-Net unrolls the update steps of the Alternating Direction Method of Multipliers (ADMM) algorithm tailored to the MMV-based problem. To effectively capture both intra-frequency and inter-frequency dependencies, the nonlinear shrinkage operator in MMV-ADMM is generalized in MMV-Net through a cascade of two modules: a Spatial Self-Attention (SSA) module and a convolutional Long Short-Term Memory (ConvLSTM) module.  

Fang et al. proposed in \cite{fang2024multifrequencyelectricalimpedancetomography} a model-based unsupervised learning framework named Multi-Branch Attention Image Prior for mfEIT image reconstruction. This method employs a custom-designed Multi-Branch Attention Network to represent multiple frequency-dependent conductivity images. 
 
However, the proposals in \cite{alberti-etal-2016}, \cite{Battistel}, \cite{MMVNET2023}, and \cite{fang2024multifrequencyelectricalimpedancetomography} adopted a linearized version of the EIT inverse problem, and consider a simplified case with no overlapping among tissues. 

Our proposal instead is based on the most complete nonlinear EIT forward model. Furthermore, it overcomes the limitation of non-overlapped tissues by incorporating partially or fully overlapped tissues using the fractional model paradigm as in \cite{Malone2014}. 

\subsection{Structure of the paper}

In Section~\ref{sec:problem} we describe the model and the inverse problem under consideration. Section~\ref{sec:proximal} contains the details of FR-PRGN, the variational method that we use to solve the inverse problem. The deep learning version of this algorithm, mf-Net, is discussed in Section~\ref{sec:unfolding}. Extensive numerical experiments are presented in Section~\ref{sec:numerics}. Finally, some concluding remarks are included in Section~\ref{sec:conclusions}.

\section{Problem formulation}\label{sec:problem}

We discuss the model for the frequency-dependent conductivities and for the inverse problem under consideration.

\subsection{Discrete fraction model}
Consider a conductive body represented as a bounded, simply connected Lipschitz domain $\Omega \subset \mathbb{R}^2$ with a piecewise $C^\infty$ boundary $\partial\Omega$. We consider a piecewise constant conductivity, depending on the frequency, namely
\begin{equation}
\sigma(x,\omega) = \sum_{n=1}^N \sigma_n(\omega) \mathbf{1}_{\Omega_n}(x),
    \label{eq:sigma}
\end{equation}
where $\mathbf{1}_{\Omega_n}$, $n=1,\dots,N$, are indicator functions of (known and fixed) subdomains forming a partition $\mathcal{T}=\{\Omega_n\}_{n=1}^N$ of $\Omega$. In practice, this partition may be a triangulation, which is convenient for the numerical implementation, and represents the finest spatial resolution we are able to achieve in the reconstruction of $\sigma$.
In this way, we identify the conductivity at a frequency $\omega$ as a vector $(\sigma_n(\omega))_n \in \R^N$. 

We assume that there are $T$ types of tissues, and that each tissue $j$ has a known conductivity spectrum $\epsilon_j(\omega)$, as a function of the frequency $\omega$.

We rely on the so-called \textit{fraction model} \cite{Malone2014}: namely, we assume that in each subdomain $\Omega_n$, $ n=1,\ldots,N$, each tissue $j=1,\ldots,T$ is present with an unknown fraction $f_{nj}$ such that
\begin{equation}
0\leq f_{nj}\leq 1 \ \ \forall j,n, \quad \sum_{j=1}^T f_{nj} = 1\ \ \forall n.
    \label{eq:frac_constraint}
\end{equation} 
The $NT$ fractions $(f_{nj})$ can be arranged in a matrix
\begin{equation}
    \F = (f_{nj})_{n,j}\in \R^{N\times T}.
\end{equation}
We denote the vectorized version of $\F$ by  $F \in \R^{NT}$, composed of $T$ column blocks $f_j \in \R^N$ corresponding to the original columns, i.e., $[F]_j = f_j$.

In some applications, it is meaningful to assume that only one tissue is present in each subdomain, namely, that all the fractions are either $0$ or $1$. As anticipated in Section \ref{sec:intro}, we are also interested in applications, such as cell calcification in organic samples, in which more tissues are present at a microscopical scale (e.g., calcified and non-calcified cells). When providing a piecewise description of $\sigma$ as in \eqref{eq:sigma}, we neglect the microscopic structure of the medium and replace it with an effective conductivity, acknowledging possible different concentrations in different subdomains.
According to \cite{Malone2014}, the resulting expression of $\sigma_n(\omega)$ in each $\Omega_n$ 
is
\begin{equation}
\sigma_n(\omega) = \sum_{j=1}^T f_{nj} \epsilon_j(\omega), 
\quad n=1,\ldots, N.
\label{eq:sigman}
\end{equation}
The proposed weighted average might be too simplistic in some applications, as it only considers the relative concentrations of the different tissues, neglecting the microscopic structure: a more refined description can be obtained by means of homogenization theory  \cite{cioranescu-donato-1999}. However, it is easy to show that model \eqref{eq:sigman} is accurate in a fully discretized model, as discussed in Remark \ref{rem:averages}.

We stress that, in \eqref{eq:sigman}, the conductivity spectra $\epsilon_j(\omega)$ are assumed to be known, while the fractions $f_{nj}$ are unknown. In other words, compared to other approaches in the literature, our objective is not to uniquely reconstruct the conductivity $\sigma$, but we consider the fractions $F$ as the unknown, leveraging a priori knowledge of the conductivity spectra for each tissue.

We consider only a finite set of frequencies $\omega_i$, $i=1,\ldots,M$. 
Each of the $T$ functions $\epsilon_j(\omega)$ can be represented by a vector $\epsilon_j\in \R^M$, which can be arranged as the rows of the (known) spectra conductivity matrix $\E$ with elements $\epsilon_{ji} = \epsilon_j(\omega_i)$, 
\begin{equation}\label{eq:matE}
    \E = (\epsilon_{ji})_{j,i}\in \R^{T\times M}.
\end{equation}
By \eqref{eq:sigman}, we can summarize the \textit{discrete fraction model} as follows:
\begin{equation}
\sigma_n(\omega_i) = \sum_{j=1}^T f_{nj} \epsilon_{ji}= [\F\E]_{ni} .
    \label{eq:discreteFM}
\end{equation}
Combining this with \eqref{eq:sigma}, we obtain the following expression for the conductivity, seen as a piecewise constant function:
\begin{equation}\label{eq:Ftosigma}
    \sigma_F(\omega_i) = \sum_{n=1}^N [\F\E]_{ni} \chi_n,
\end{equation}
where we have highlighted the dependence on the unknown fractions $F$.

\subsection{The Complete Electrode Model}

The Complete Electrode Model (CEM) is a widely used mathematical framework for accurately describing the relationship between internal conductivity and boundary measurements in EIT \cite{cheney-isaacson-newell-1999,Vetal1999}. 

The key assumption behind CEM is that the current injection and boundary voltage measurements are performed through $P$ electrodes $\{E_p\}_{p=1}^P$, which are open, non-empty, and non-overlapping subsets of $\partial \Omega$. Both the injected current $I$ and the measured boundary voltages $U$ are assumed to be piecewise constant on each electrode, as well as zero-mean: i.e., they belong to $\mathcal{E}_P =\{v\in\R^P:v_1+\dots+v_P=0\}$. 

For a fixed conductivity $\sigma \in L^\infty(\Omega)$ and for a selected injected current $I \in \mathcal{E}_P$, we define the pair $(u,U) \in H^1(\Omega) \times \mathcal{E}_P$ as the solution of the following system of differential equations:
\begin{equation}
\begin{aligned}
-\nabla \cdot (\sigma \nabla u) &= 0 \qquad \text{in} \; \Omega, \\
u + z_p \sigma \frac{\partial u}{\partial n} &= U_p \qquad \text{on} \; E_p, \ p=1,\ldots,P, \\
\int_{E_p} \sigma \frac{\partial u}{\partial n} \, ds &= I_p \qquad p= 1, \ldots, P, \\
\sigma \frac{\partial u}{\partial n} &= 0 \qquad \text{on} \ \partial\Omega \setminus \cup_{p=1}^P E_p.
\end{aligned}
\label{eq:CEM}
\end{equation}
where $z_p$ are the (known) contact impedances of the electrodes and $\partial\Omega \setminus \cup_{\ell=1}^p E_p$ is the part of the boundary without electrodes. The vector $U \in \mathcal{E}_P$ corresponds to the boundary measurements of the voltage $u$ at the $P$ electrodes.

In practice, we do not solve \eqref{eq:CEM} exactly, but we rely on a finite element strategy for its numerical approximation. In doing so, we consider a triangulation $\mathcal{T}$ of $\Omega$ and introduce a space of (globally continuous) piecewise polynomials on it. When restricting the weak formulation of \eqref{eq:CEM} to such functions, the determination of the solution reduces to solving a linear system (see, e.g., \cite{lechleiter,lechleiter2008newton,felisi2024full}). 

Such a process can be iterated for different boundary currents: in particular, in the CEM formulation, the forward problem of EIT consists of determining, given a conductivity $\sigma$, the voltage measurements $U^h$ corresponding to several applied current patterns $I^h$, $h=1,\ldots,H$. Since the map from $I$ to $U$ is linear and the space $\mathcal{E}_P$ is isomorphic to $\R^{P-1}$, it would be enough to consider $H=P-1$ independent current vectors and acquire the corresponding boundary voltages. Nevertheless, it is common practice to probe the body with a larger number of current patterns ($H>P-1$), according to standardized protocols, so as to acquire overdetermined measurements. We concatenate the recorded boundary voltage vectors associated with different currents in a single vector in $\R^K$, being $K=(P-1)H$. This allows us to define the forward map of CEM-EIT, i.e., $v \colon \R^N \rightarrow \R^{K}$ such that
\begin{equation}
v(\sigma) = (U^{h})_{h=1}^H, \ \ (u^h,U^h) \text{ solves \eqref{eq:CEM}-FE with $I=I^h$}, 
    \label{eq:fwdprob}
\end{equation}
where we denote the finite element discretization of the differential problem by \eqref{eq:CEM}-FE.

When considering a frequency-dependent $\sigma$ described via the fraction model \eqref{eq:Ftosigma}, we denote the forward map as
\begin{equation}
v_F(\omega_i) = v(\sigma_F(\omega_i))
    \label{eq:Ftov}
\end{equation}
to stress its dependency (at any frequency) on the true unknown of the inverse problem, namely, the fractions $F$.

\begin{remark}
In the proposed Finite Element approximation of \eqref{eq:CEM}, the conductivity $\sigma$ only plays a role in the quadratic form
\[
\int_\Omega \sigma(x) \nabla u(x) \cdot \nabla v(x) dx = \sum_{n=1}^N [\nabla u]_n\cdot [\nabla v]_n \int_{\Omega_n} \sigma_n(x) dx,
\]
where we used the fact that the unknown and test functions are piecewise linear on the triangulation, thus their gradients are constant on each triangle.
If we assume that each element $\Omega_n$ can be partitioned into $T$ (possibly empty) subdomains $\Omega_{nj}$, each composed of a single tissue, hence associated to a constant conductivity $\epsilon_j$,  the respective integral in the previous summation reads as 
\[
 \int_{\Omega_n} \sigma_n(x) dx = \sum_{j=1}^T \epsilon_j|\Omega_{nj}| = |\Omega_n|\sigma_n, \quad \sigma_n:= \sum_{j=1}^T \epsilon_j\frac{|\Omega_{nj}|}{|\Omega_n|},
\]
which provides a convincing motivation for adopting the fraction model described in \eqref{eq:sigman}, defining $f_{nj}=\frac{|\Omega_{nj}|}{|\Omega_n|}$.
    \label{rem:averages}
\end{remark}

\subsection{Frequency difference and the inverse problem}

Instead of looking directly at the measurements $v_F(\omega_i)\in\R^K$ for several frequencies, it is common practice to use frequency-difference data,
in order to reduce the frequency-independent modeling error (such as uncertainties in the knowledge of $\partial\Omega$ or of the contact impedances), see \cite{Seo_2008,alberti-etal-2016}.  To this aim, we fix a reference frequency $\omega_0$, with associated conductivity spectrum $\epsilon_{j0},$ $j=1,\ldots,T$, for which the measurements are available, and consider the frequency-difference data
\[
\Phi(F) = (v_F(\omega_i) - v_F(\omega_0))_{i=1}^M.
\]
For simplicity of notation, we collect these $M$ vectors in $\R^K$ into a single vector of size $KM$, so that the forward map $\Phi$ is defined as
\[
\Phi \colon \R^{NT}\to \R^{KM},\quad  F\mapsto \Phi(F),
\]
where we recall that the fractions were assembled into a vector $F\in \R^{NT}$.

Let $F^\dagger$ denote the true unknown fraction vector. We consider the following perturbed measurements
\begin{equation}
    y=\Phi(F^\dagger)+\eta,
    \label{eq:IP}
\end{equation}
where $\eta\in \R^{KM}$ represents measurement noise. The inverse problem under consideration consists of the reconstruction of an approximation of $F^\dagger$ from the knowledge of $y$. From the reconstructed fractions, it is possible to obtain an approximation of the
corresponding conductivity distribution by using \eqref{eq:Ftosigma}.
A possible approach for the solution of the inverse problem is to tackle the following nonlinear least squares problem
\begin{equation}
\min_{F \in \Gamma} \frac{1}{2}\| r(F) \|^2,\qquad r(F) = \Phi(F) - y,
    \label{eq:LS}
\end{equation}
where the residual $r$
is minimized among all fractions complying with \eqref{eq:frac_constraint}, namely belonging to the set 
\begin{equation}
\label{eq:gamma}
\Gamma = 
\left\{
F: f_{nj} = [f_j]_n \geq 0 \ \forall n,j, \  \sum_{j=1}^T f_{nj} = 1 \ \forall n 
\right\}.
\end{equation}
Nevertheless, due to the well-known ill-posedness of the inverse EIT problem, the solution of \eqref{eq:LS} would be highly unstable, especially in the presence of noisy measurements. For this reason, in the next section we introduce a suitable regularized version of \eqref{eq:LS}, together with an iterative algorithm capable to handle the non-trivial constraint \eqref{eq:gamma} appearing in it.

For the implementation of gradient-based methods, it will be useful to compute the Jacobian $J_\Phi(F) \in \R^{(M K) \times (T N) } $ of $\Phi(F)$: for this reason, we provide here its explicit expression. 
Such an operator can be seen as a block matrix: for each frequency $i = 1,\ldots,M$ and tissue $j = 1, \ldots, T$ the block $[J_\Phi(F)]_{ij} \in \R^{K\times N}$ is given by
\begin{equation}
\begin{split}
    [J_\Phi(F)]_{ij} &= \frac{\partial (v_F(\omega_i) - v_F(\omega_0))}{\partial f_j} \\
    &= \frac{\partial v(\sigma_i)}{\partial\sigma} \epsilon_{ji}-\frac{\partial v(\sigma_0)}{\partial\sigma} \epsilon_{j0}, 
 \end{split}
\end{equation}
which has been obtained via the chain rule since, by \eqref{eq:discreteFM} and  \eqref{eq:Ftosigma}, we have
\(
\frac{\partial \sigma_i}{\partial f_j} = \epsilon_{ji}.
\)

\section{Proximal Regularized Gauss-Newton (PRGN) for mfEIT reconstruction} \label{sec:proximal}

In this section, we propose a variational strategy to solve the inverse problem of recovering $F^\dag$ from $y$ in \eqref{eq:IP}.

\subsection{FR-PRGN variational method}

The chosen variational formulation for fraction recovery in multi-frequency EIT is based on solving the following regularized version of the least squares problem \eqref{eq:LS}:
\begin{equation}
    \label{eq:functional2}
   \min_{F \in \Gamma}{\mathcal{J}(F)} , \quad \mathcal{J}(F):= 
   \frac{1}{2}\| r(F)\|^2 + \frac{\alpha}{2} \| F - \hat{F}\|^2 + R(F),
\end{equation}
where $R \colon \R^{NT} \rightarrow \R$ is a convex, eventually non-smooth, functional,  $\Gamma$ is defined as in \eqref{eq:gamma}, and we select a known reference $\hat{F} \in \Gamma$, as described in Section \ref{sec:FEST}.

In problem \eqref{eq:functional2}, we consider two different regularization functionals: the convex term $R$ is suitably chosen to capture desirable features of the unknown solution, possibly motivated by prior information. The quadratic term $\frac{\alpha}{2} \| F-\hat{F}\|^2$, instead, is introduced to deal with the ill-conditioning of the EIT inverse problem. The functional
\begin{equation}
\label{eq:falpha_GN}
f_{\alpha}(F) :=\frac{1}{2} \|r(F)\|^2 + \frac{\alpha}{2} \|F - \hat{F}\|^2
\end{equation} 
can be interpreted as a regularized squared residual
\[
f_\alpha(F) = \frac{1}{2}\| r_\alpha(F)\|^2, \quad r_\alpha(F) = \begin{bmatrix}
    r(F) \\ \sqrt{\alpha} (F-\hat{F})
\end{bmatrix} .
\]

The constrained optimization problem in \eqref{eq:functional2} can be reformulated by means of the characteristic function 
\[
\iota_{\Gamma}(F) =  \begin{cases} 0 & \text{if } F \in \Gamma,\\ \infty & \text{if } F \notin \Gamma, \end{cases} 
\] 
obtaining the following unconstrained optimization problem 
 \begin{equation}
    \label{eq:functional55}
   \min_{{F \in \R^{NT}}} f_\alpha(F) + g(F),
\end{equation}
being 
$g(F)=R(F) +\iota_{\Gamma}({F})$. 

We remark that the functional $f_\alpha$ is differentiable with Lipschitz continuous gradient given by $\nabla f_\alpha(F) = J_{r_\alpha}(F)^T r_{\alpha}(F) = J_\Phi(F)^T r(F) + \alpha (F-\hat{F})$, whereas $g$ is a proper, convex, non-smooth function. 
Leveraging these properties, we approximate the solution of  \eqref{eq:functional55} with a Proximal Regularized Gauss-Newton (PRGN) method. The PRGN approach extends the Proximal Gauss-Newton (PGN) method, originally proposed by \cite{salzo2012convergence,Lee2014}, to effectively manage the inherent ill-posedness of the EIT problem \cite{Col2023DeepplugandplayPG}.

It alternates an inexact Newton step with respect to the residual $f_\alpha$ with a scaled proximal step. At this aim, the Hessian matrix at $F^{(k)}$ is approximated by the positive definite matrix 
\begin{equation}
H_k = J_{r_\alpha}(F^{(k)})^T J_{r_\alpha}(F^{(k)}) = J_{\Phi}(F^{(k)})^T J_{\Phi}(F^{(k)}) + \alpha I,
    \label{eq:hessian}
\end{equation}
where $\alpha I$ overcomes the ill-conditioning of the matrix $J_{\Phi}^TJ_{\Phi}$.
Starting from an initial 
value ${F}^{(0)}$, for $k=1,2,\ldots$ the PRGN algorithm reads as 
\begin{align}
z^{(k)} &= F^{(k-1)} - \beta H_{k-1}^{-1}\nabla f_{\alpha}(F^{(k-1)}),
\label{eq:PRGN_1}\\
F^{(k)} &= \prox_{g}^{H_{k-1}} (z^{(k)}) 
\label{eq:PRGN_2}
\end{align}
where $\displaystyle \prox_g^{H}(z) := \argmin_{F \in \R^{NT}} \left\{ \frac{1}{2}\| F - z \|_H^2 + g(F) 
 \right\}$ and $\| z \|_H^2 = \langle Hz,z \rangle$.
The computation of $F^{(k)}$ via \eqref{eq:PRGN_2}  requires the solution of a constrained minimization problem of the following form:
\begin{equation}
\min_{F \in \Gamma} \tilde{\mathcal{J}}(F), \quad \tilde{\mathcal{J}}(F):= \frac{1}{2}\| F - z^{(k)} \|_{H_{k-1}}^2 + R(F).
\label{eq:md1}
\end{equation} 
For $R(F)$ in \eqref{eq:md1}, a simple Tikhonov regularization option is $R(F)=\frac{1}{2}\alpha_{\text{EMDA}}\| F\|^2$.
A closed-form solution of this kind of minimization problems over the
unit simplex is, in general, not available. 
Since $R$ is convex and differentiable, so is also the objective function $\tilde{\mathcal{J}}$; moreover, as the set $\Gamma \subset \R^{NT}$ is closed and convex, we can approximate the solution of \eqref{eq:md1} by applying the Entropic Mirror Descent Algorithm (EMDA) proposed in \cite{BT2003}.

In particular, to employ EMDA, we need first to specify a suitable entropy function $\psi\colon \Gamma \rightarrow \R$, and then, starting from $F^{(k,0)} \in \operatorname{int}(\Gamma)$, computing the following sub-iterates for $\ell=1,\ldots$:
\begin{equation}
F^{(k,\ell)} = \nabla \psi^\star\big( \nabla \psi(F^{(k,\ell-1)})-t_\ell \nabla \tilde{\mathcal{J}}(F^{(k,\ell-1)})\big),
\label{eq:EDA}
\end{equation}
where $\psi^\star$ is the Fenchel conjugate of $\psi$,
\begin{equation*}
\nabla \tilde{\mathcal{J}}(F^{(k,\ell-1)})=H_{k-1}(F^{(k,\ell-1)}-z^{(k)})+ \nabla R(F^{(k,\ell-1)}),
\end{equation*}
and $t_\ell >0$ are step sizes, chosen as $t_\ell=\frac{\sqrt{2\ln T}}{L_{\tilde{\mathcal{J}}}} \frac{1}{\sqrt{\ell}}$ (being $L_{\tilde{\mathcal{J}}}$ an estimate of the Lipschitz constant of $\tilde{\mathcal{J}}$).

A general guideline for the choice of $\psi$ is that the gradient of its conjugate $\nabla \psi^\star$ should play the role of a nonlinear projection operator on $\Gamma$. 
In \cite{BT2003}, the authors propose an entropy function for the case in which $\Gamma$ is the unitary simplex. Adapting their expression, we consider
\begin{equation}
\psi(F) =  \begin{cases}
    \sum_{n=1}^N \sum_{j=1}^T f_{nj} \ln(f_{nj})  &\text{if $F \in \Gamma$,} \\
\infty  &\text{if $F \notin \Gamma$,}
\end{cases}
    \label{eq:entropy}
\end{equation}
where $f_{nj} = [f_j]_n$ denotes the $n$-th component of the $j$-th block (or row) of $F$. It follows that, for $F \in \Gamma$, we have $\nabla\psi(F)=\ln(F) + 1$. 
Two important properties of $\psi$ are collected in the following result, whose proof can be found in the appendix.

\begin{lemma}
\label{lem:entropy}
Let $\psi$ be defined as in \eqref{eq:entropy}. Then:
    \begin{enumerate}[i)]
        \item $\psi$ is strongly convex over $\Gamma$ with respect to the 1-norm, namely, for a positive constant $C$,
\begin{equation}
\langle \nabla \psi(F)- \nabla \psi(G),F-G \rangle \geq C \| F-G\|_1^2 \quad \forall F,G \in \operatorname{int}(\Gamma).
    \label{eq:strongconv}
\end{equation}
        \item the Fenchel conjugate $\psi^\star(F)$ is given by 
\begin{equation}
    \label{eq:entropystar}
\psi^\star(F) = \sum_{n=1}^N \ln \left( \sum_{j=1}^T e^{f_{nj}} \right).
\end{equation}
    \end{enumerate}
\end{lemma}

We can now provide an explicit formulation of the EMDA algorithm \eqref{eq:EDA} for the problem at hand. Notice that, from \eqref{eq:entropystar},
\[
[\nabla \psi^\star(F)]_{nj} = \frac{e^{f_{nj}}}{\sum_{j'=1}^T e^{f_{nj'}}}, 
\]
which is equivalent to performing, for each position $n$, a soft-max operation across the blocks $F$ (analogously, recalling that $F$ is the vectorized version of a matrix $\F$, a row soft-max operation),  namely:
\[
\begin{gathered}
\nabla\psi^\star(F) = \operatorname{softmax}_r(F), \\
[\operatorname{softmax}_r(F)]_{n,:} = \operatorname{softmax}([F]_{n,:})    .
\end{gathered}
\]
Moreover, since the soft-max operation is invariant by the addition of the same constant to each component, we can replace $\nabla \psi(F)$ in \eqref{eq:EDA} by $\ln(F)$. Finally, given the expression of $\tilde{\mathcal{J}}$ in \eqref{eq:md1}, we can reformulate \eqref{eq:EDA} as follows
\begin{equation*}
F^{(k,\ell)} = \operatorname{softmax}_r\big(\ln(F^{(k,\ell-1)}) 
-t_\ell \nabla \tilde{\mathcal{J}}(F^{(k,\ell-1)})\big) .
\end{equation*}
Equivalently, for each row $n=1,\ldots,N$ and component $j=1,\ldots,T$, EMDA takes the following explicit form
\begin{equation}
\label{eq:itn}
f_{nj}^{(k,\ell)} = \frac{f_{nj}^{(k,\ell-1)} e^{-t_\ell [\nabla \tilde{\mathcal{J}}(F^{(k,\ell-1)})]_{nj}}}{\displaystyle\sum_{j'=1}^T f_{nj'}^{(k,\ell-1)} e^{-t_\ell [\nabla \tilde{\mathcal{J}}(F^{(k,\ell-1)})]_{nj'}}}.
\end{equation}

We summarize in Algorithm \ref{alg:alg1} the proposed algorithm for Fraction Reconstruction by means of PRGN, with approximate proximal computed via EMDA.
\begin{algorithm}[H]
\caption{FR-PRGN}\label{alg:alg1}
\begin{algorithmic}
\State 
\State {\textsc{Input}}\hspace{0.3cm} $y \in \R^{KM}$; $\E \in \R^{T \times M}$, $\alpha,\beta \in \R$, $\hat{F} \in \R^{NT}$, $tol$
\State  
{\textsc{Output}} \hspace{0.3cm} $F \in \R^{NT}$
\State $k=1$
\State \textbf{repeat} 
\Comment{PRGN loop}
\State \hspace{0.5cm}Update $H_{k-1}=J_\Phi(F^{(k-1)})^T J_\Phi(F^{(k-1)}) + \alpha I $
\State \hspace{0.5cm}$z^{(k)} \gets F^{(k-1)}- \beta H_{k-1}^{-1}\nabla f_{\alpha}(F^{(k-1)})$ 
\State \hspace{0.5cm}set $F^{(k,0)}=F^{(k-1)}$
\State \hspace{0.5cm}\textbf{for } $ \ell= 1,...,L$ \Comment{EMDA loop}
\State \hspace{1cm}$
F^{(k,\ell)}\!\gets\!\operatorname{softmax}_r\!\big(\!\ln(F^{(k,\ell-1)}) 
-t_\ell \nabla \tilde{\mathcal{J}}(F^{(k,\ell-1)})\!\big)$
\State \hspace{0.5cm}\textbf{end } 
\State \textbf{until } $|F^{(k)}-F^{(k-1)}| \le tol$
\end{algorithmic}
\label{alg1}
\end{algorithm}

In Algorithm \ref{alg:alg1} the PRGN loop is initialized with $ F^{(0)}$ which has $f_1 = 1_N$ for the background tissue, and 
$[f_2, \ldots, f_T ] = 0$ for all the other $T-1$ column blocks/tissues. The EMDA loop, properly initialized with $F^{(k,0)}=F^{(k-1)}$, requires only a few iterations $L$ to obtain accurate results.
Although a thorough discussion of the convergence properties of Algorithm \ref{alg:alg1} is outside the scope of this work, we remark that, due to the approximate solution of \eqref{eq:PRGN_2} via EMDA, the previous algorithm can be interpreted as an inexact Proximal Gauss-Newton method, for which a local convergence result is discussed, e.g., in \cite{porta2022inexact}.

\begin{figure*}[!t]
    \centering
\includegraphics[width=1\textwidth]{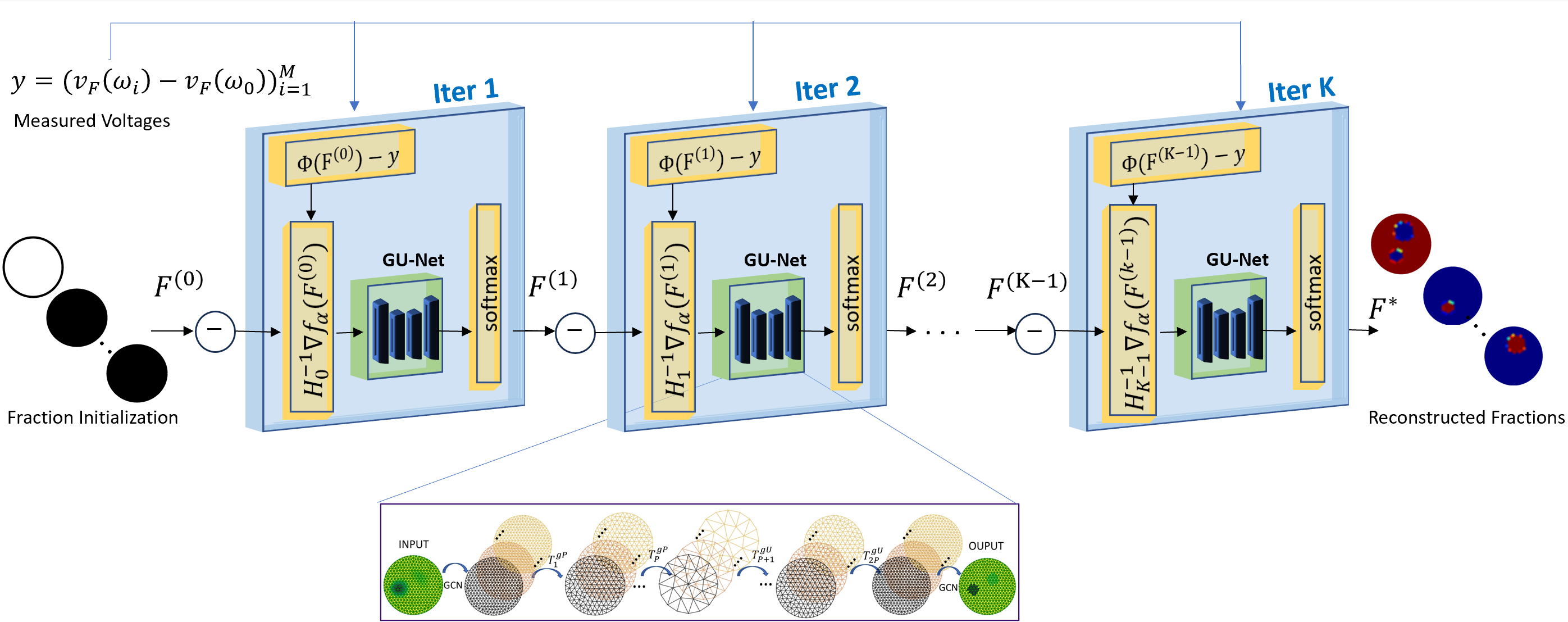} 
\caption{The overall unrolled architecture of the proposed mf-Net with $K$ iterative blocks.}
    \label{fig:unrolled}
\end{figure*}

\subsection{Fraction Estimate (F-EST)}
\label{sec:FEST}

The reference fraction matrix $\hat{F}$ is obtained by minimizing the following regularized functional for $\bar{\F}=[f_2,..,f_T]\in \R^{N\times (T-1)}$ 
\begin{equation}
    \label{eq:functional5}
\mathcal{J}_{\bar{\F}}(\bar{\F}):= \frac{1}{2}\|\bar{\F}\bar{\E}
- {\Sigma} \|_2^2 + \frac{\lambda}{2} \| \bar{\F}\|_2^2,  \end{equation}
where $\lambda >0$, and $\bar{\E}=[\epsilon_2,..,\epsilon_M]\in \R^{(T-1) \times M}$ and $\Sigma \in \R^{N \times M}$ collects the column vectors $\sigma_i - \epsilon_{i,1}$, $i=1,..,M$, where $\sigma_i$ is simply obtained by NOSER estimates \cite{noser}. Then  $\hat{\F}=
[ \mathbf{1}_N -\sum_{j=2}^T f_j,\bar{\F}]$.
Minimizing functional \eqref{eq:functional5}, which reads as
$$
\mathcal{J}_{\bar{\F}}(\bar{\F})= \frac{1}{2}( \bar{\E}^T \bar{\F}^T \bar{\F}\bar{\E} - 2 {\Sigma}^T \bar{\F} \bar{\E} 
+ {\Sigma}^T{\Sigma}) + \frac{\lambda}{2} \bar{\F}^T\bar{\F},
$$
involves calculating its gradient and then applying the first-order optimality conditions.
In particular
\begin{equation}
\label{eq:oc}
\begin{aligned}
    \nabla{\mathcal{J}_{\bar{\F}}}(\bar{\F}) &= \bar{\F}(\bar{\E}\bar{\E}^T +\lambda I )-
    {\Sigma} \bar{\E}^T =0,     
\end{aligned}
\end{equation}
leads to the computations of $N$ linear systems with shared (full rank) coefficient matrix, unknown column vectors $(\bar{\F}^T)_i$ and right-hand side $(\bar{\E}\Sigma^T)_i, i=1,..,N$.

\section{Unfolding PRGN for Fraction Reconstruction in mfEIT}\label{sec:unfolding}

We propose a variational network, named mf-Net, which unrolls the iterations of the FR-PRGN Algorithm 1 into a deep neural network that is based on a graph convolutional neural network, introduced in Section \ref{sec:GUNET}. 

\subsection{The network mf-Net}
Specifically, mf-Net alternates between two main iterative steps in Algorithm 1: the gradient descent step \eqref{eq:PRGN_1}, and the proximal step \eqref{eq:PRGN_2}, which will be the main core of the learning proposal.

The proximal operator \eqref{eq:PRGN_2} of the convex function $g$ at a point $z$ is the solution of the constrained minimization problem \eqref{eq:md1} 
which can be interpreted as a regularized denoising operator in $H$-norm. 
To ensure that the solution $F$ belongs to the set $\Gamma$, the denoising result is then projected onto $\Gamma$ via a {\tt softmax()} operator. 
The {\tt softmax()} is not a projection itself, but can be seen as an approximation of the nonlinear projection of a vector onto the probability simplex: for each mesh node, the tissues sum up to one, and are all positive. 

The graph neural network GU-Net plays the role of the learned denoiser in $\R^{NT}$, the result of which is then projected into the $\Gamma$ space.

The unrolling mfNet framework relies on the architecture illustrated in Fig. \ref{fig:unrolled}. Starting from an initial fraction guess ${F}^{(0)}$ and the collection of the voltage measurements stacked in the vector $y$, each iteration block $k$ is outlined in Algorithm \ref{alg2}.
The output $F^{(k)}$ of the iteration block $k$ is used as an input for the following one.

After stacking \texttt{K} such blocks, the resulting network is trained end-to-end using $N_s$ training samples $\{(y_i,F^\dag_i)\}_{i=1}^{N_s}$, being $y_i = \Phi(F^\dag_i)+\eta_i$, through back-propagation and ADAM optimizer. 
Let $F_{\Theta}^*(y)$ be the output of the mfNet; then, the loss function to be minimized reads as:
\[
\mathcal{L}(\Theta)= \frac{1}{N_s}\sum_{i=1}^{N_s} \| F_{\Theta}^*(y_i) -F^\dag_i \|^2. 
\]
The following part discusses the details of the GU-Net denoiser.

\subsection{GU-Net architecture}
\label{sec:GUNET}

In the following, we describe the GU-Net denoiser, proposed in \cite{Col2023DeepplugandplayPG} for the solution of the EIT inverse problem, which extends the CNN-based denoiser to non-Euclidean manifold domain and relies on the encoder-decoder Graph-U-Net architecture introduced in \cite{gao2019}.
 It is based on a convolution graph kernel and the gPool and gUnpool operators. The pool (gPool) operation samples some nodes to form a smaller graph based on their scalar projection values on a trainable projection vector. 
As an inverse operation of gPool, the unpooling (gUnpool) operation restores the graph to its original structure with the help of locations of nodes selected in the corresponding gPool \cite{gao2019}. 

The GU-Net denoiser with $D$ layers can be formalized as a composition of functions $\mathcal{G} = T_D \circ \cdots \circ T_1$, where $(T_\ell )_{1 \le \ell \le D}$ are the layers of the network. At this aim, let us define an input mesh with $N$ vertices, characterized by the adjacency and degree matrices $A,B \in \R^{N \times N}$, respectively, and a graph convolutional operator defined as
\[
GCN(X;\Theta):=B^{-1/2} A B^{-1/2} X \Theta,
\]
which incorporates the trainable weights $\Theta$. 

Each layer is characterized by the composition of the graph convolution GCN, a ReLU activation function $s$, and a gPool/gUnpool operator, here denoted by a generic $p$, which acts only on the vertices of the mesh, and is applied to the feature array $X$, namely
\begin{equation}
\label{eq:layer}
T_\ell: X \mapsto s(GCN(p(X);\Theta_{\ell})),
\end{equation}
with $X \in \R^{N \times n_c}$ for $\ell=1$, $X \in \R^{N_{\ell} \times h_f}$ for $\ell=2,\ldots,D-1$, and $X \in \R^{N \times n_c}$ for $\ell=D$, $h_f$ the number of hidden features.
The set of all the trainable weights of the GU-Net denoiser is given $\Theta=\{\Theta_\ell\}_{\ell=1}^D$, with $\Theta_{1}  \in \R^{n_c \times h_f}$, $\Theta_{\ell}  \in \R^{h_f \times h_f}$, $\ell=2,\ldots,D-1$,
$\Theta_{D} \in \R^{h_f \times n_c}$.

The resulting network is an autoencoder defined by the encoder-decoder structure as follows:
\begin{equation}
\mathcal{G} = \underbrace{T_{2P}^{gU} \circ \cdots \circ T_{P+1}^{gU}}_{decoder} \circ \underbrace{T_P^{gP} \circ \cdots \circ T_1^{gP}}_{encoder},
\label{eq:Gcomp}
\end{equation}
where in $T_\ell^{gP}$, $p(X) = gPool(X)$ for $\ell=1,\ldots,P$, and $T_\ell^{gU}$, $p(X) = gUnPool(X)$ for $\ell=P+1,\ldots,2P$, with number of layers $D=2P$. In Fig.\ref{fig:unrolled}, a representation of the GU-Net is zoomed in on the box below Iter 2. In our mf-Net, the number of input/output features is the number of tissues, i.e., $n_c = T$.

\begin{algorithm}[H]
\caption{Iter $k$: mf-Net}\label{alg:alg2}
\begin{algorithmic}
\State 
\State {\textsc{Input}}\hspace{0.3cm} $y\!\in\!\R^{KM}$; $\E\!\in\!\R^{T \times M}$, $\alpha,\!\beta\!\in\!\R$, $\hat{F}\!,\!F^{(k-1)}\!\in \!\R^{NT}$
\State  
{\textsc{Output}} \hspace{0.3cm} $F^{(k)} $
\State \hspace{0.5cm}{Update Metric $H_{k-1}=J_{r_{\alpha}}(F^{(k-1)})^T J_{r_{\alpha}}(F^{(k-1)}) +\alpha I $}
\State \hspace{0.5cm}$ z^{(k)} \gets F^{(k-1)} - \beta H_{k-1}^{-1}\nabla f_{\alpha}(F^{(k-1)})$
\State \hspace{0.5cm}$F^{(k)} = \mathrm{softmax}(\mbox{GU-Net}(z^{(k)}))$
\State
\textbf{ end } 
\end{algorithmic}
\label{alg2}
\end{algorithm}

\section{Experimental Evaluation}\label{sec:numerics}
This section details the numerical experiments conducted to evaluate our proposed approaches: the variational FR-PRGN method, and the variational network, mf-NET. 
In Section \ref{sec:ID}, the implementation details are reported; the employed datasets are described in Section \ref{sec:DS}; comparisons and ablation studies are detailed in Sections~\ref{sec:RE} and \ref{sec:AS}, respectively.

In order to evaluate the quality of the reconstructions, we consider two metrics. 
The first, Err$_{\sigma_i}$, is the $L_2$-norm of the relative error between the recovered conductivity $\sigma(\omega_i)$ and the corresponding ground truth phantom $\sigma^{GT}(\omega_i)$ for each frequency $\omega_i, i=1,\dots,M$.
The second, Err$_{f_i}$, is the fraction relative recovery error given by the $L_2$ norm between the approximated fraction solution and the reference (ground truth) fraction. 

\subsection{Implementation Details}
\label{sec:ID}

In the computation of the forward map $\Phi(F)$ we applied the KTCFwd forward solver, a two-dimensional version of the FEM described in \cite{Vetal1999}, kindly provided by the authors\footnote{\texttt{\url{https://github.com/CUQI-DTU/KTC2023-CUQI4}}}, which is based on a FEM implementation of the CEM model on triangle elements. The electric potential is discretized using second-order polynomial basis functions, while the conductivity is discretized on the nodes using linear basis functions on triangular elements.

The training of mf-Net has been performed with the ADAM optimizer \cite{Adam}, using a learning rate equal to $1.0 \times 10^{-3}$ through $1000$ epochs with a mini-batch size of $10$ instances. 
 
The free parameters of the FR-PRGN method (see Algorithm \ref{alg:alg1}) are set to $\alpha = 1.0 \times 10^{-9}$ and $\beta=0.3$, and we selected $R(F) = \frac{1}{2}\alpha_{\text{EMDA}}\|F\|^2$ with $\alpha_\text{EMDA}=1.0 \times 10^{-4}$, setting $L_{\tilde{\mathcal{J}}}=1.5$. We implemented Algorithm 1 from \cite{Malone2014}, selecting $\tau = 1.0 \times 10^{-4}$ and fixing $\beta^t = 0.3$ at each iteration.

The system was implemented leveraging PyTorch 2.5.1, Python 3.11.2, and PyTorch Geometric 2.6.1. The experiments were conducted on an Intel i9 workstation running Linux, equipped with an NVIDIA RTX 4090 GPU.
The mf-Net architecture's baseline configuration, featuring \texttt{Iter Block K}$=9$, $h_f=64$, and a GU-Net \texttt{Depth} of $3+3$, was used for the experiments in Section \ref{sec:RE}. Our choice of these parameters is justified by the ablation study presented in Section \ref{sec:AS}.

\subsection{Simulation Datasets}
\label{sec:DS}

We created two {\em training datasets}: the \texttt{No-Overlap} dataset, which contains samples with no overlapping among tissues, and the \texttt{Overlap} dataset, where all samples can present non-trivial partial or whole intersection between tissues, with randomly chosen radii and centers. 

Each dataset contains 100 samples with two/three circular inclusions, which represent different tissues contained in a saline solution ($T=3$ for \texttt{Overlap}, $T=3$ or $T=4$ for \texttt{No-Overlap}). The {\em testing datasets} for both \texttt{No-Overlap} and \texttt{Overlap} datasets consist of 50 samples each.

The spatial domain $\Omega$  is a circular tank discretized by a triangulation $\mathcal{T}$ composed of 432 vertices.
The measurements are taken with a circular 32-electrode EIT sensor at 2 different frequencies and an additional reference frequency $\omega_0$. For each sample in the dataset and for each frequency, a pair of voltage-conductivity has been synthetically generated through FEM simulation, obtaining a total of $2\times 100$ simulated data. We adopted the adjacent-adjacent measurement strategy. 

The experimental data follows the experimental design described in \cite{SpectraTissue}, by considering as circular inclusions carrot, potato and cucumber tissues in a saline background. 

For the \texttt{Overlap} dataset, the conductivity spectral data for saline ($t_1$), carrot ($t_2$) and cucumber ($t_3$) tissues are defined with reference frequency $\omega_0=1KHz$ and $M=2$  frequencies  $\omega_1=5KHz$, $\omega_2=50KHz$.
The matrix $\E\in \R^{3 \times 2 }$ in \eqref{eq:matE} is composed by two columns $\epsilon^1= [0.13, 0.043, 0.066]^T$, $\epsilon^2= [0.13, 0.150, 0.181]^T$, and, for the reference frequency, the conductivity spectra are given by $\epsilon^0= [0.13, 0.034, 0.048]^T$.

For the \texttt{No-Overlap} dataset, the conductivity spectral data for saline ($t_1$), carrot ($t_2$), cucumber ($t_3$) and potato ($t_4$) tissues are defined with reference frequency $\omega_0=1KHz$ and $M=2$  frequencies  $\omega_1=100KHz$, $\omega_2=1000KHz$.
The matrix $\E\in \R^{4 \times 2 }$ in \eqref{eq:matE} is composed column-wise by $\epsilon^1= [0.13, 0.175, 0.250, 0.130]^T$, $\epsilon^2= [0.13, 0.310, 0.405, 0.230]^T$, and $\epsilon^0= [0.13, 0.100, 0.023, 0.008]^T$ for the reference frequency $\omega_0$. 

The estimated fractions $\hat{F}$ have been computed for all the experiments by applying the Algorithm F-EST (Section~\ref{sec:FEST}), which relies on the solution of the linear systems \eqref{eq:oc}.

All the methods have been initialized with a fraction matrix $\F^{(0)}$ with columns $f_1^{(0)}= {\mathbf 1}_N$, $f_2^{(0)}= {\mathbf 0}_N$,\ldots,  $f_T^{(0)}= {\mathbf 0}_N$, perturbed by additive white Gaussian noise $\mathcal{N}(0,1)$.

\begin{figure*}[h]
    \centering
    \begin{tabular}{c|ccc|cccc}
    Sample & 3 & 23 & 46 & 47 & 0 & 17 &38  \\
    \hline
    GT & 
\includegraphics[width=0.07\textwidth]{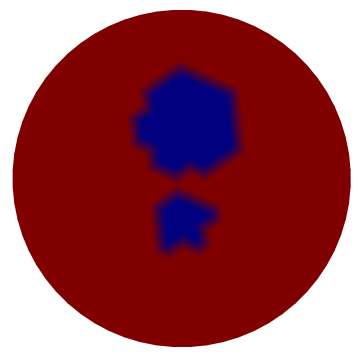} & \includegraphics[width=0.07\textwidth]{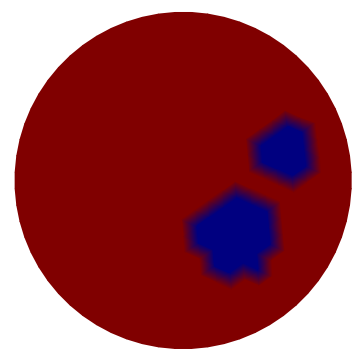} & \includegraphics[width=0.07\textwidth]{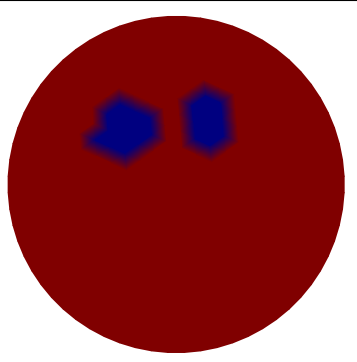} & 
\includegraphics[width=0.07\textwidth]{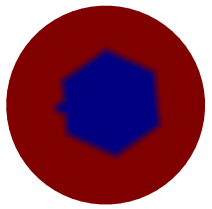} &
\includegraphics[width=0.07\textwidth]{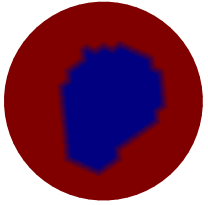} &
\includegraphics[width=0.07\textwidth]{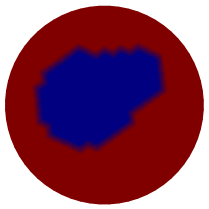} &
\includegraphics[width=0.07\textwidth]{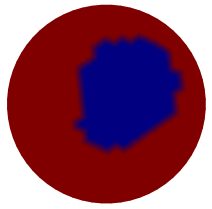} 
\\
& \includegraphics[width=0.07\textwidth]{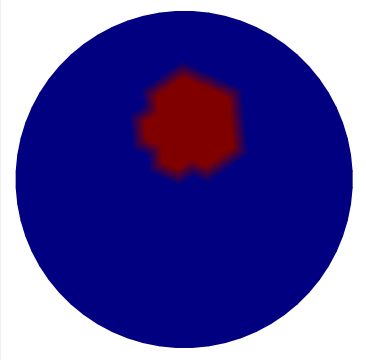} & 
\includegraphics[width=0.07\textwidth]{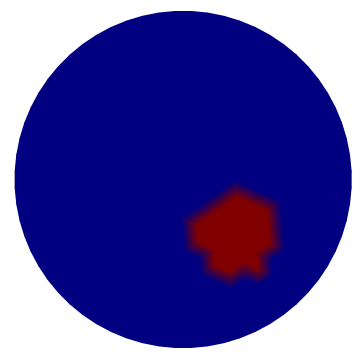} & 
\includegraphics[width=0.07\textwidth]{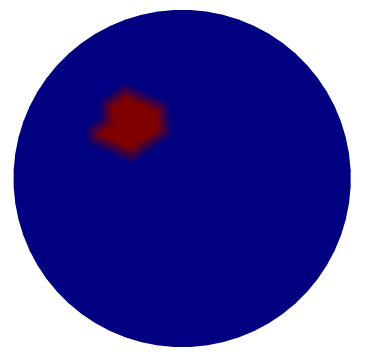} & 
\includegraphics[width=0.07\textwidth]{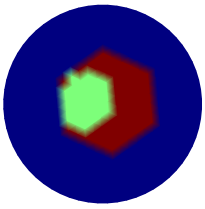}&
\includegraphics[width=0.07\textwidth]{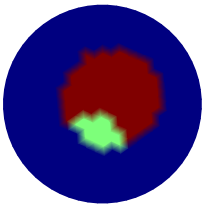} &
\includegraphics[width=0.07\textwidth]{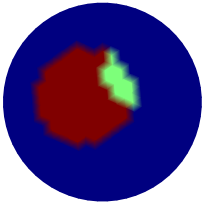} &
\includegraphics[width=0.07\textwidth]{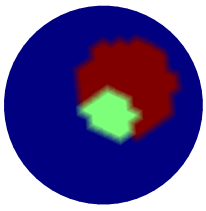} 
\\
& \includegraphics[width=0.07\textwidth]{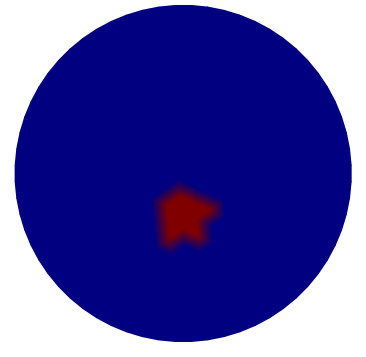} & 
\includegraphics[width=0.07\textwidth]{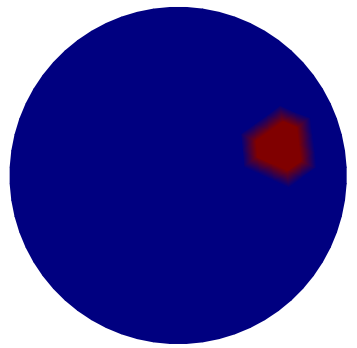} & 
\includegraphics[width=0.07\textwidth]{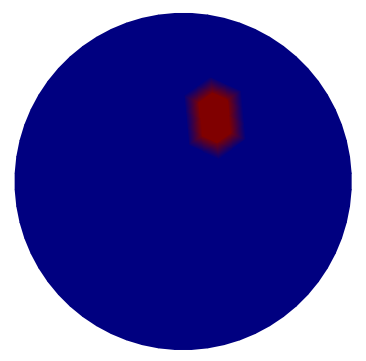} & 
\includegraphics[width=0.07\textwidth]{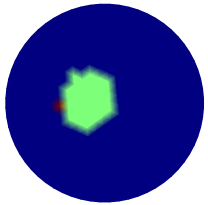}&
\includegraphics[width=0.07\textwidth]{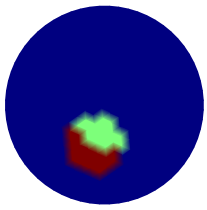} &
\includegraphics[width=0.07\textwidth]{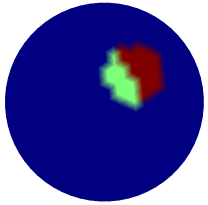} &
\includegraphics[width=0.07\textwidth]{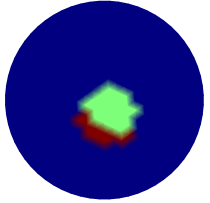} 

\\
\hline
mf-NET  &
\includegraphics[width=0.07\textwidth]{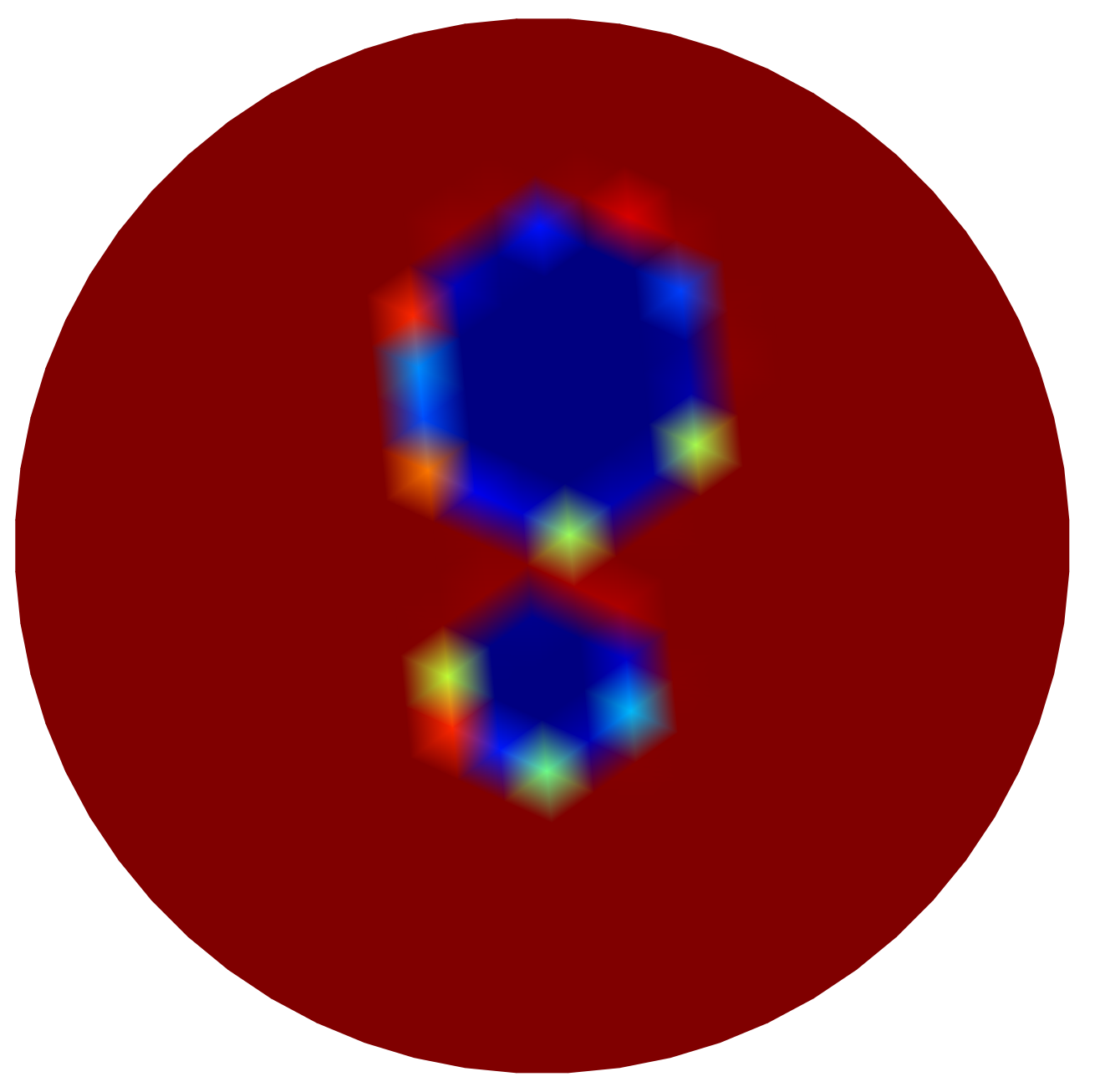} & 
\includegraphics[width=0.07\textwidth]{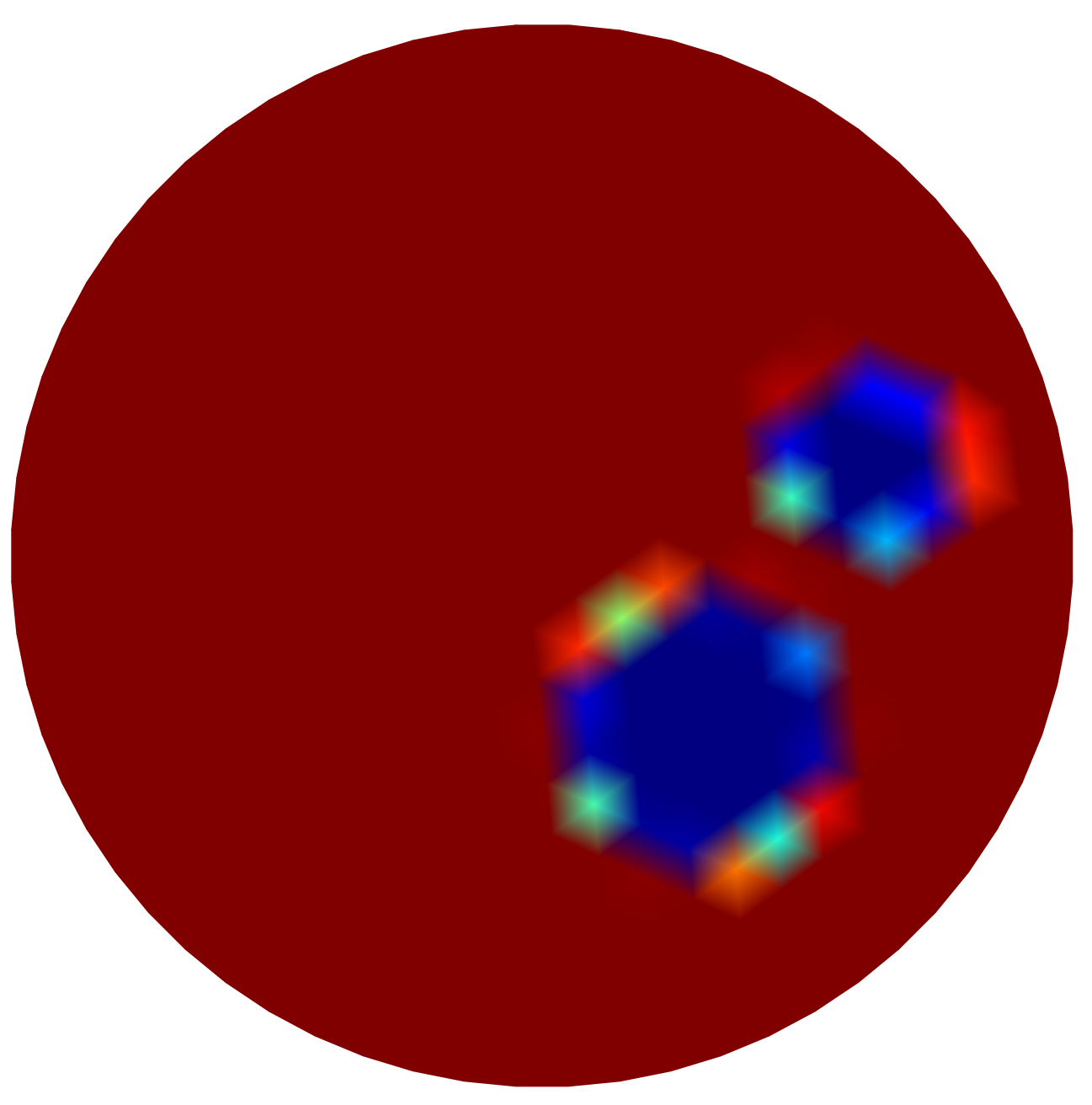} &
\includegraphics[width=0.07\textwidth]{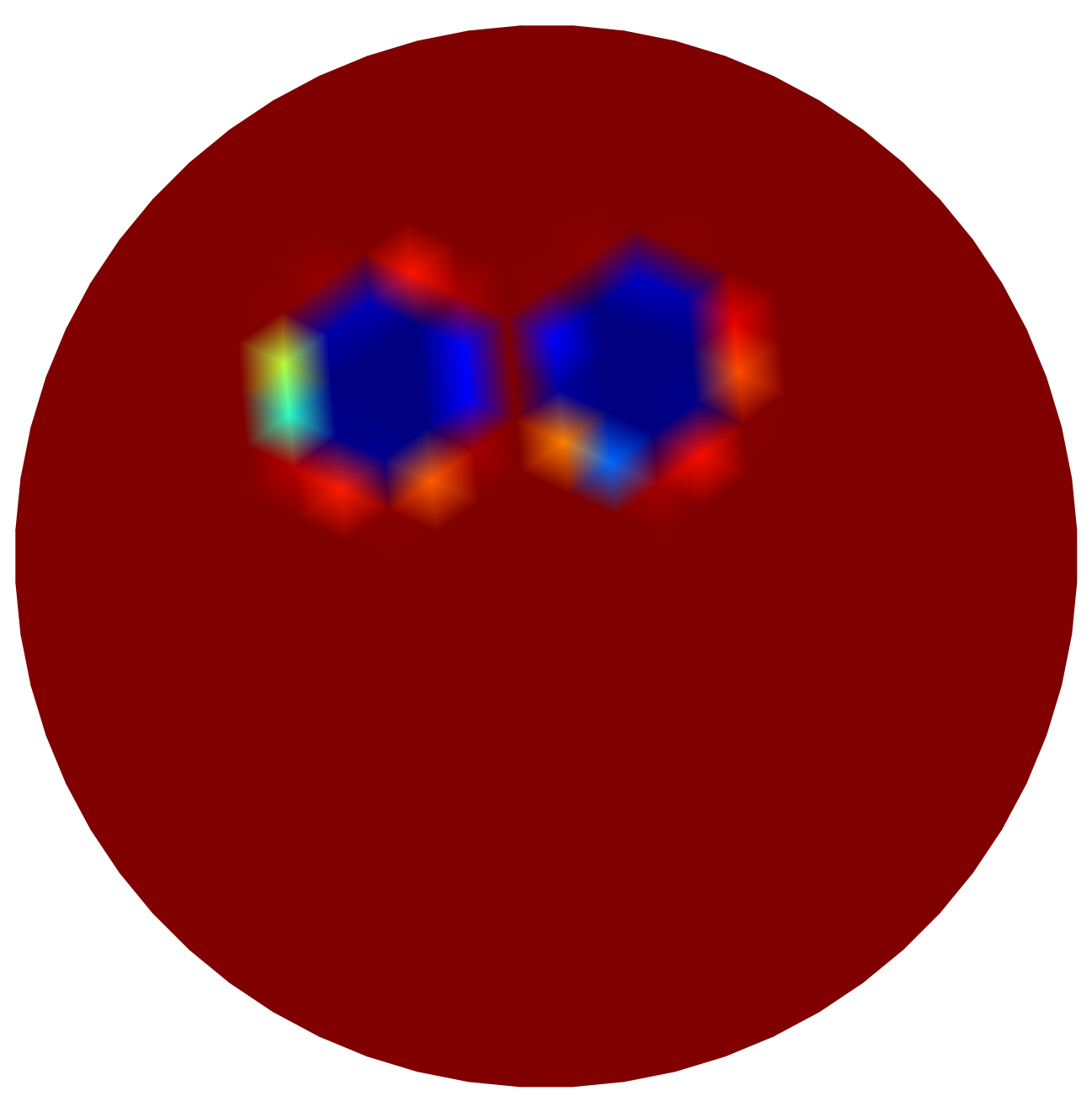} &
\includegraphics[width=0.07\textwidth]{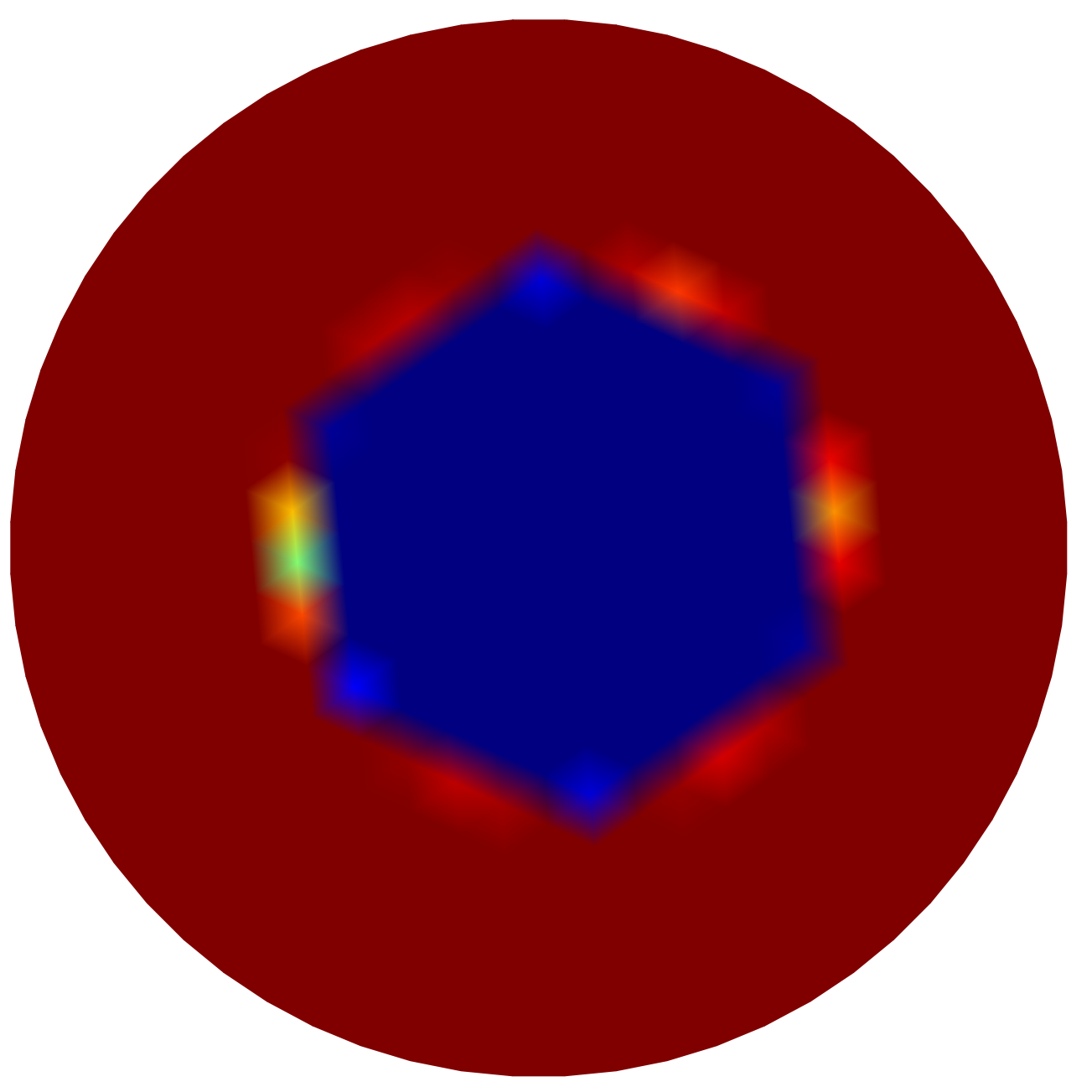} &
\includegraphics[width=0.07\textwidth]{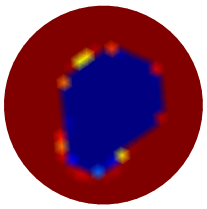} &
\includegraphics[width=0.07\textwidth]{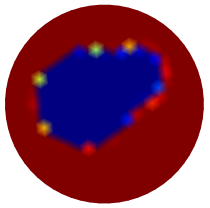} &
\includegraphics[width=0.07\textwidth]{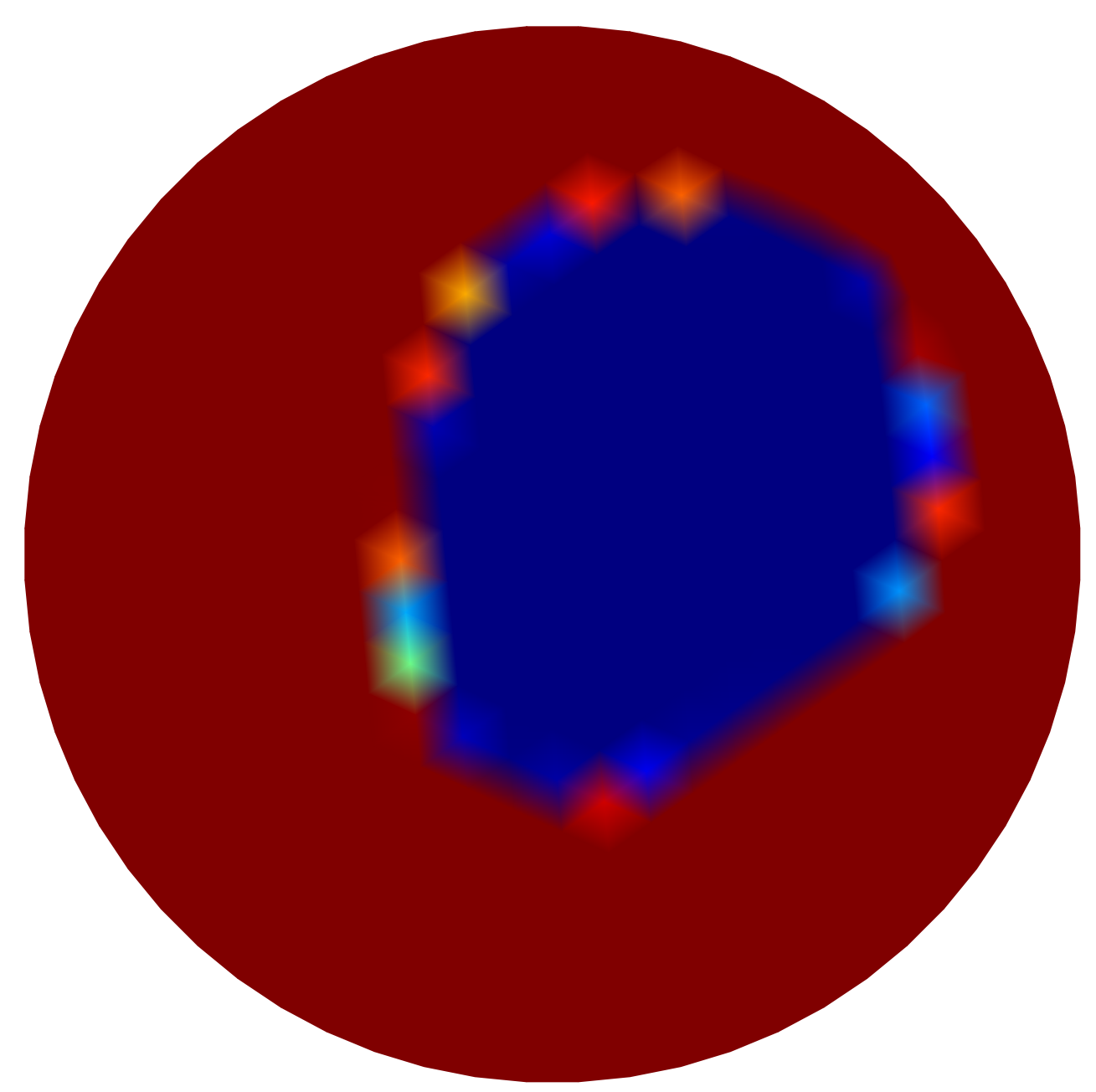} 
\\
& \scriptsize{Err$_{f_1}$=0.068} & \scriptsize{Err$_{f_1}$=0.056} & \scriptsize{Err$_{f_1}$=0.085} &
\scriptsize{Err$_{f_1}$=0.039}&\scriptsize{Err$_{f_1}$=0.054}&\scriptsize{Err$_{f_1}$=0.047}&\scriptsize{Err$_{f_1}$=0.076} 
\\
& \includegraphics[width=0.07\textwidth]{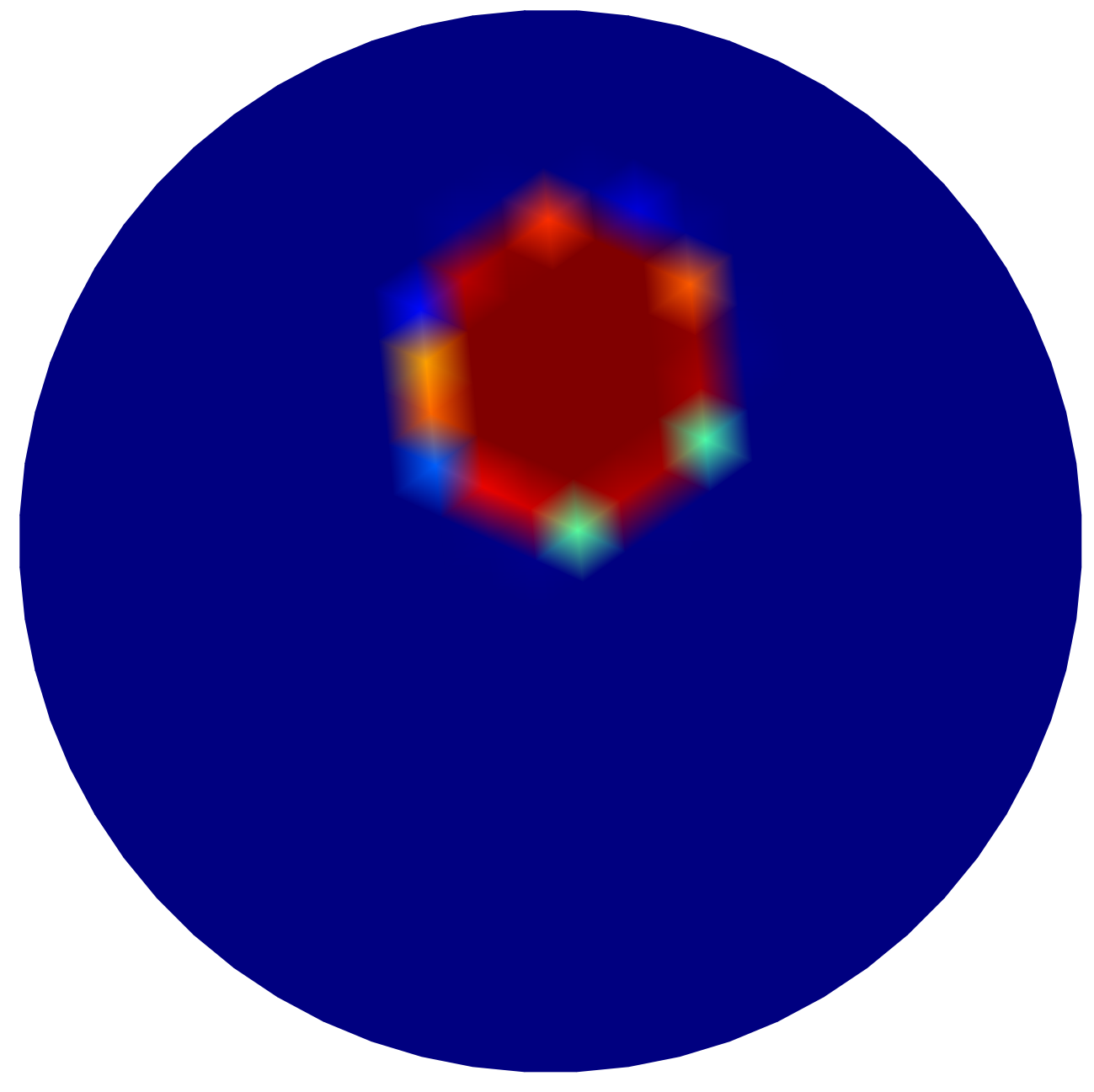} & 
\includegraphics[width=0.07\textwidth]{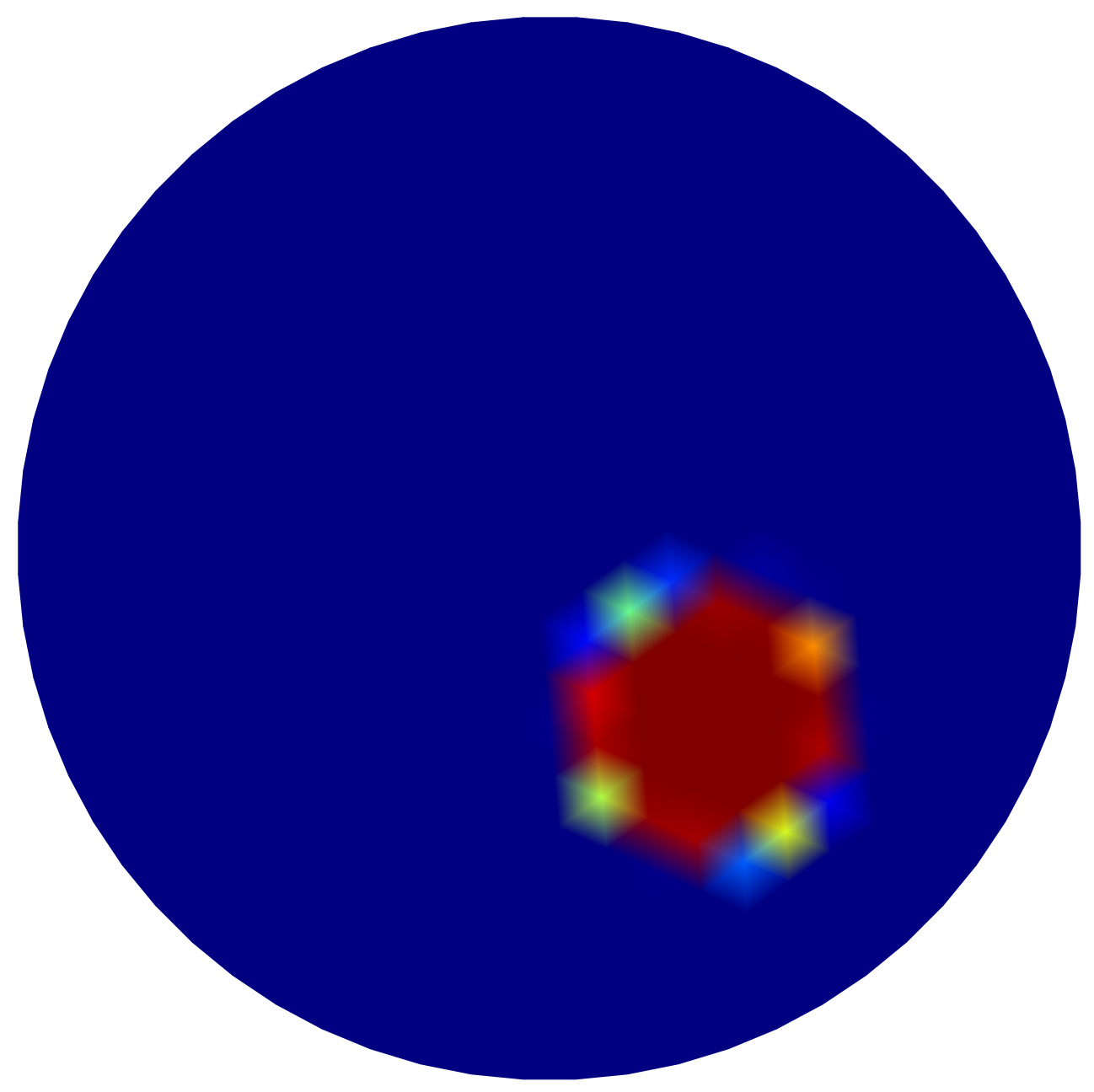} &
\includegraphics[width=0.07\textwidth]{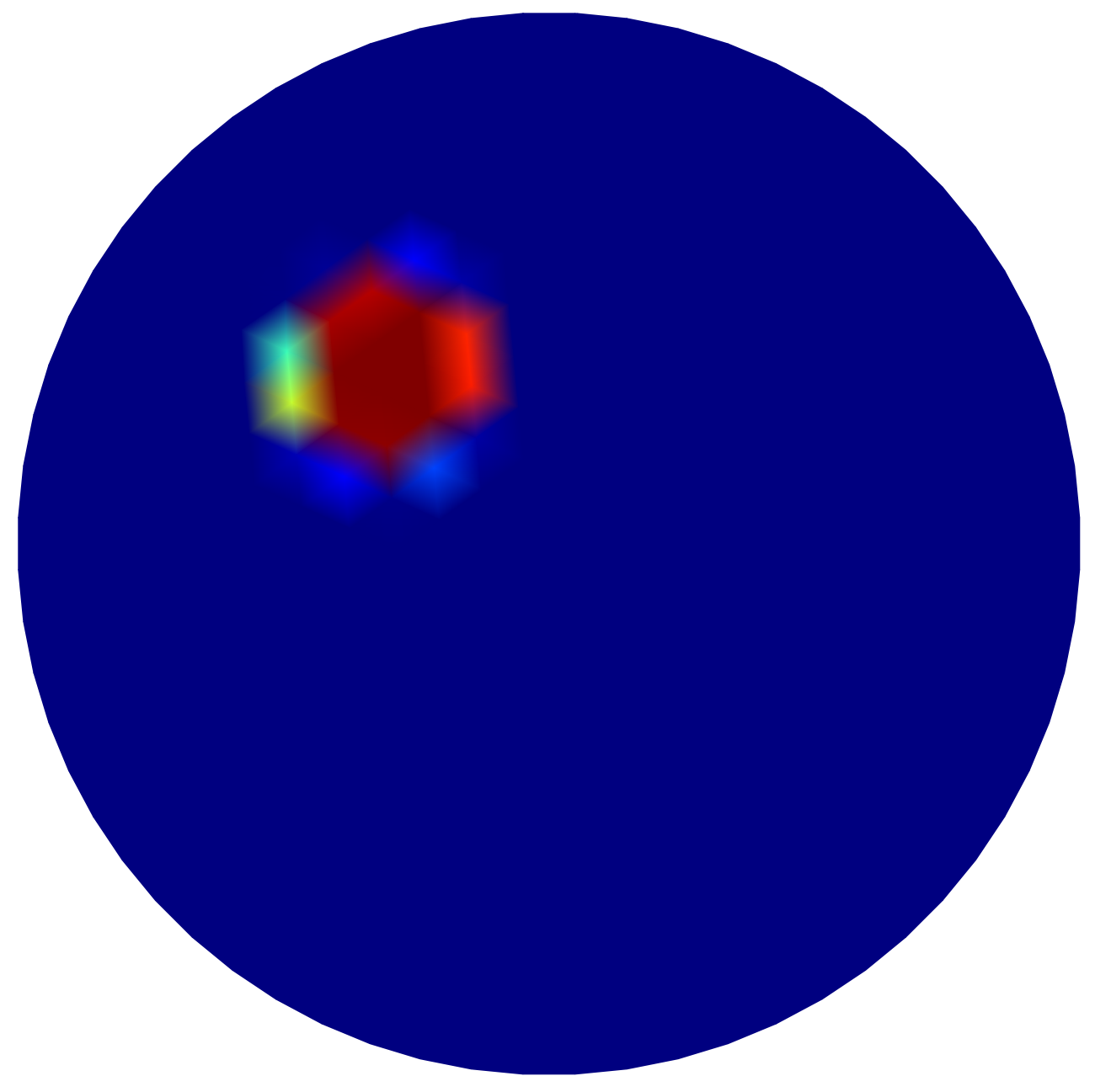} & \includegraphics[width=0.07\textwidth]{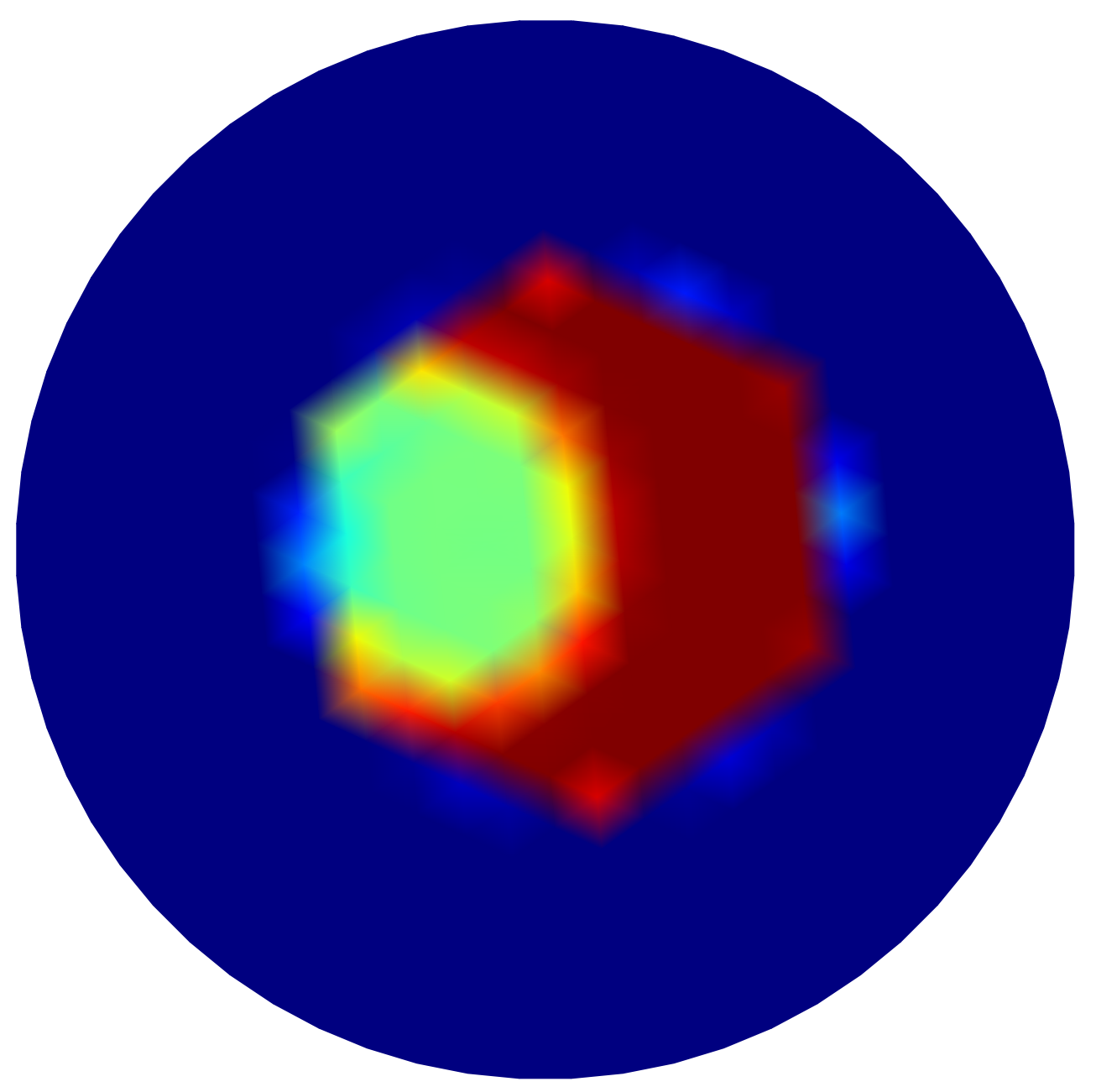}&
\includegraphics[width=0.07\textwidth]{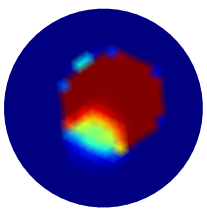} &
\includegraphics[width=0.07\textwidth]{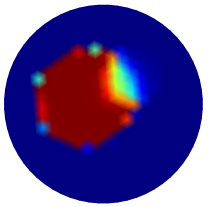} &
\includegraphics[width=0.07\textwidth]{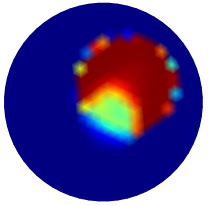} 
\\
& \scriptsize{Err$_{f_2}$=0.152} & \scriptsize{Err$_{f_2}$=0.191} & \scriptsize{Err$_{f_2}$=0.187} &
\scriptsize{Err$_{f_2}$=0.120} &
\scriptsize{Err$_{f_2}$=0.130} &
\scriptsize{Err$_{f_2}$=0.128} &
\scriptsize{Err$_{f_2}$=0.132} 
\\
& \includegraphics[width=0.07\textwidth]{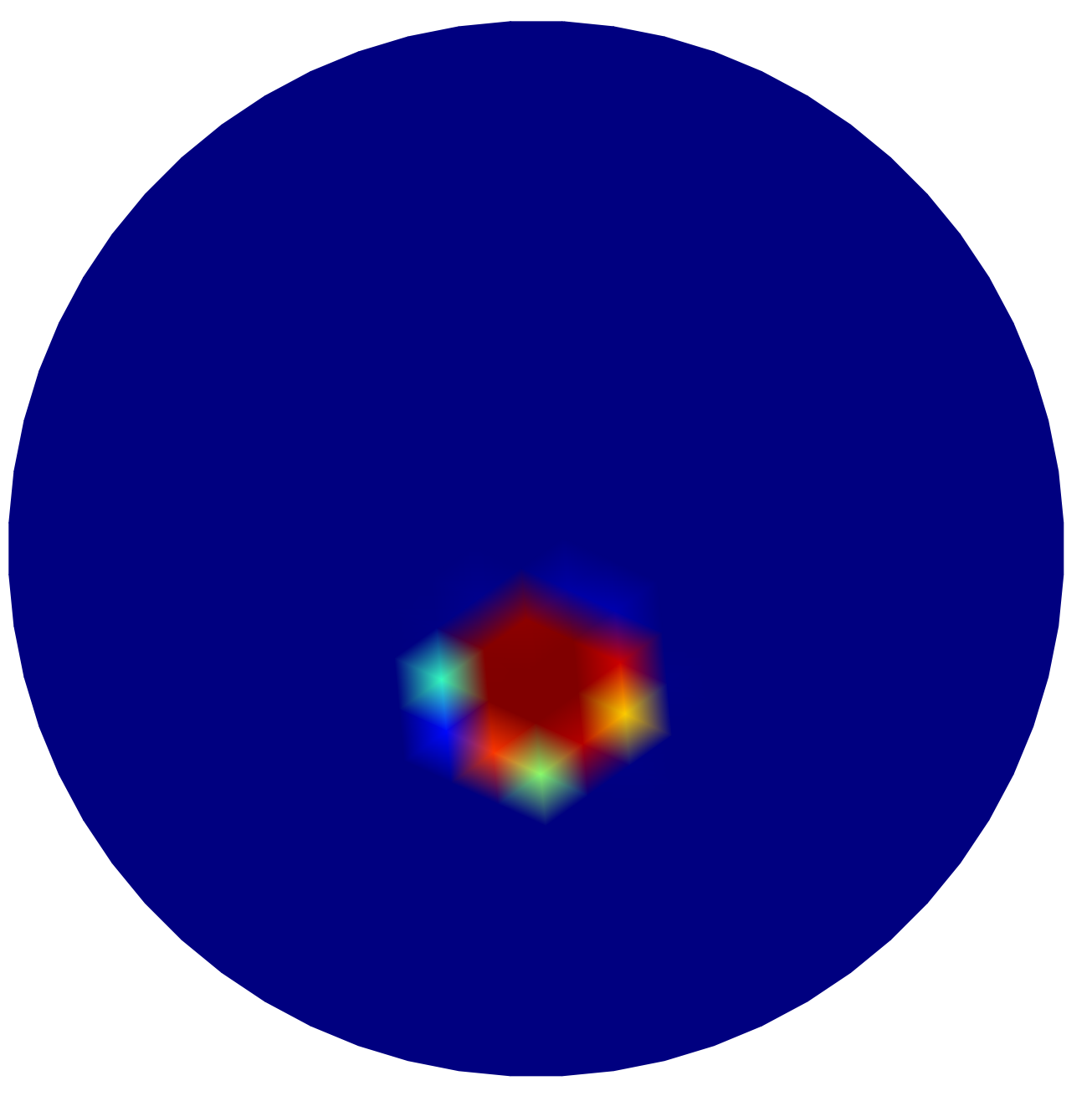} & 
\includegraphics[width=0.07\textwidth]{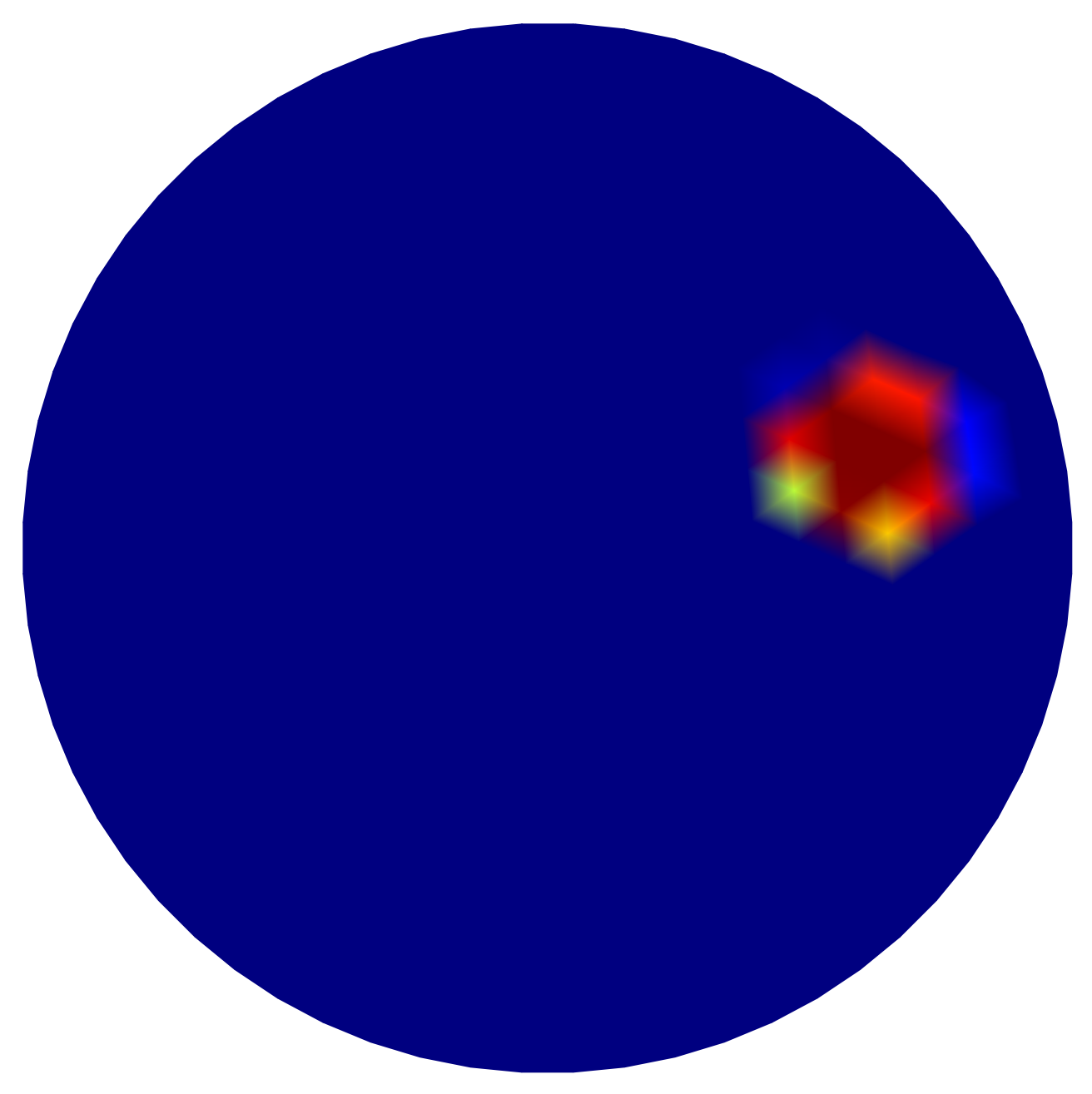} &
\includegraphics[width=0.07\textwidth]{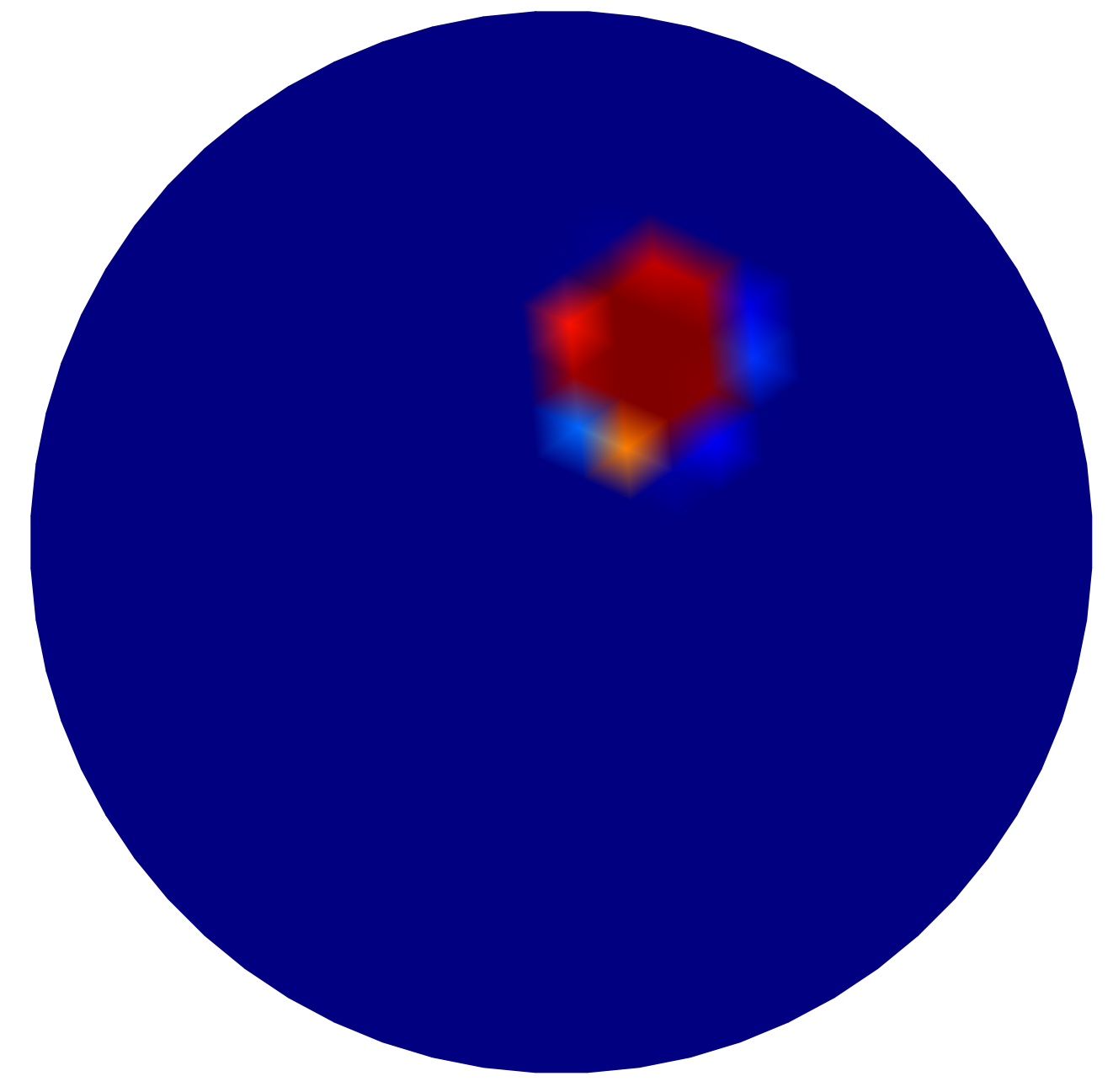} &
\includegraphics[width=0.07\textwidth]{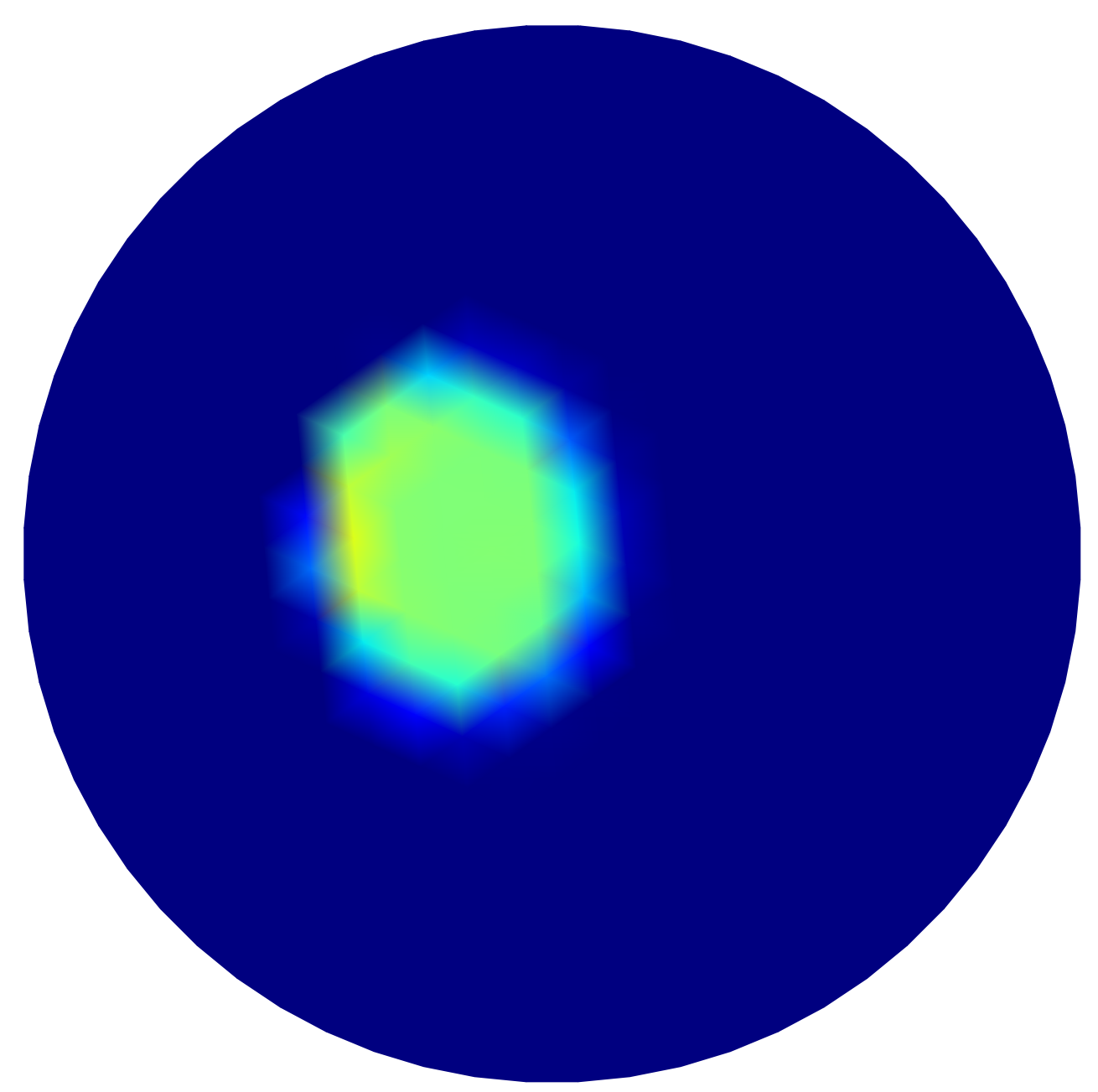} &
\includegraphics[width=0.07\textwidth]{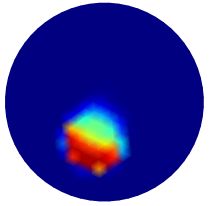} &
\includegraphics[width=0.07\textwidth]{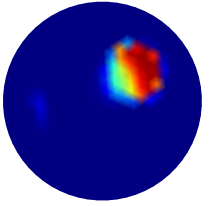} &
\includegraphics[width=0.07\textwidth]{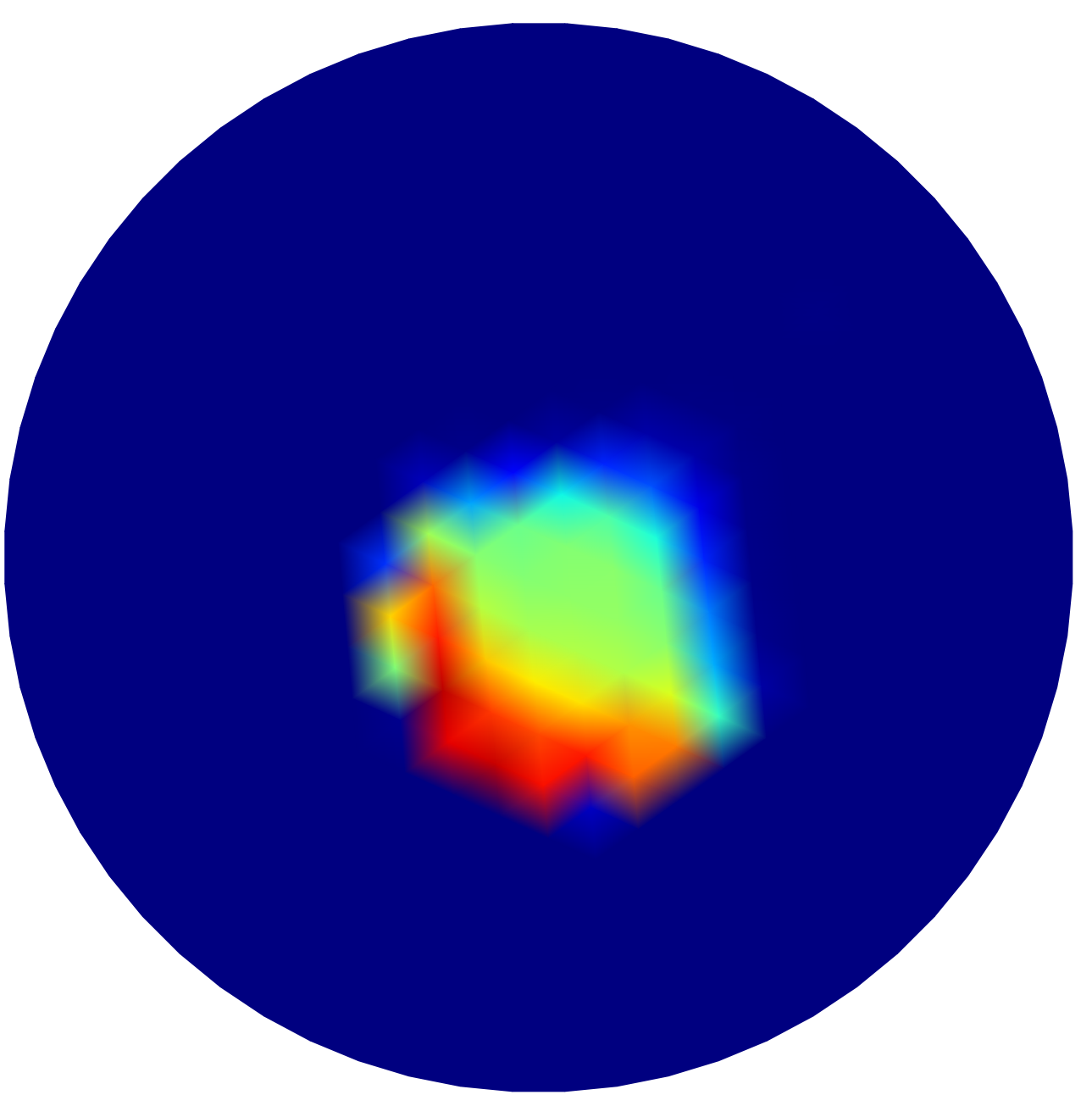} 
\\
& \scriptsize{Err$_{f_3}$=0.311} & \scriptsize{Err$_{f_3}$=0.170} & \scriptsize{Err$_{f_3}$=0.493} &
\scriptsize{Err$_{f_3}$=0.382} &
\scriptsize{Err$_{f_3}$=0.213} &
\scriptsize{Err$_{f_3}$=0.218} &
\scriptsize{Err$_{f_3}$=0.352} 
\\
\hline
    \cite{Malone2014} & 
    \includegraphics[width=0.07\textwidth]{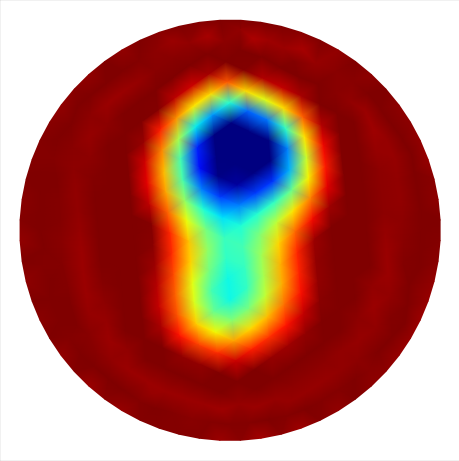} &
    \includegraphics[width=0.07\textwidth]{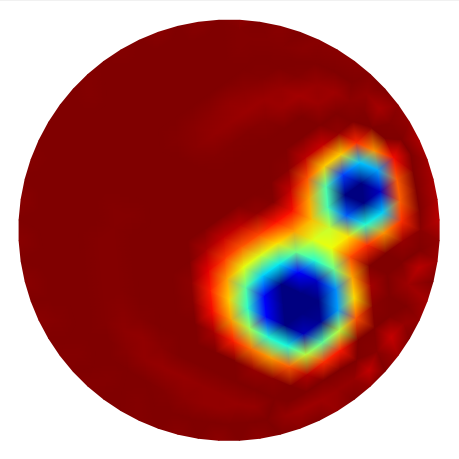} &
    \includegraphics[width=0.07\textwidth]{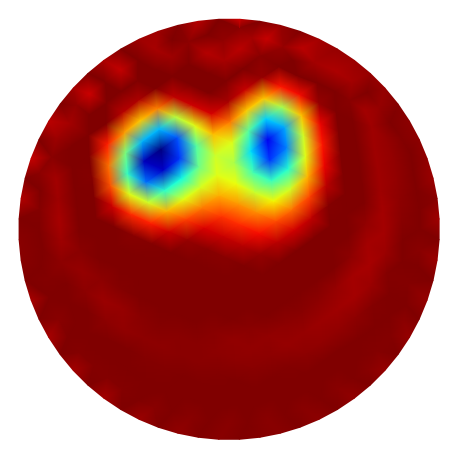} &
    \includegraphics[width=0.07\textwidth]{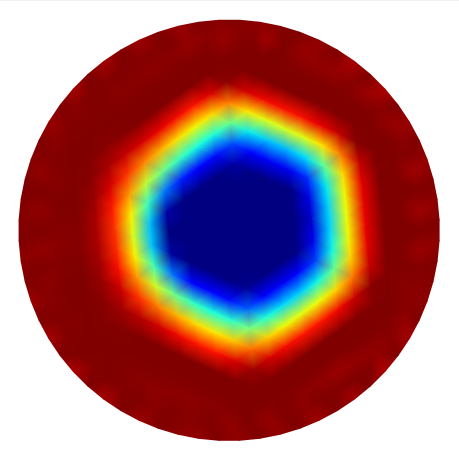} &
    \includegraphics[width=0.07\textwidth]{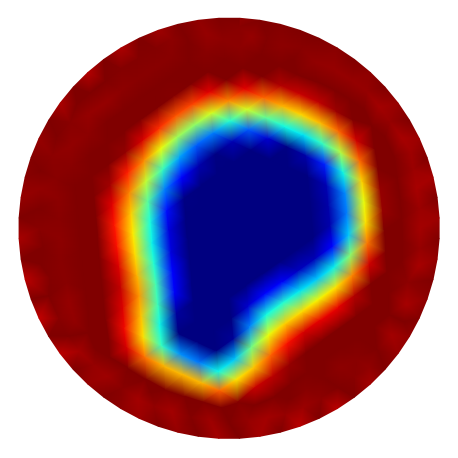} & 
    \includegraphics[width=0.07\textwidth]{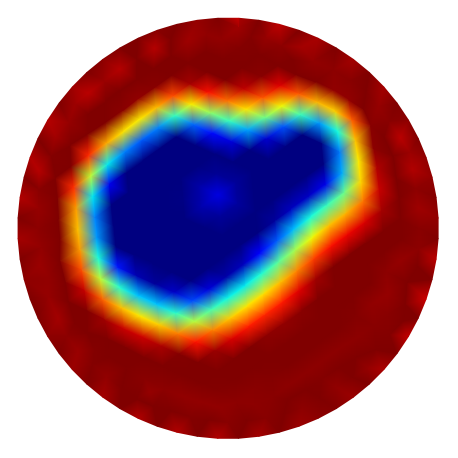} & 
    \includegraphics[width=0.07\textwidth]{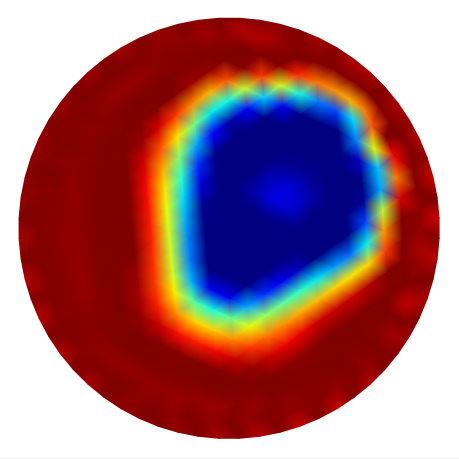} \\&
\scriptsize{Err$_{f_1}$=0.167} &
\scriptsize{Err$_{f_1}$=0.124} &
\scriptsize{Err$_{f_1}$=0.130} &
\scriptsize{Err$_{f_1}$=0.169} &
\scriptsize{Err$_{f_1}$=0.171} &
\scriptsize{Err$_{f_1}$=0.163} &
\scriptsize{Err$_{f_1}$=0.158} \\
& \includegraphics[width=0.07\textwidth]{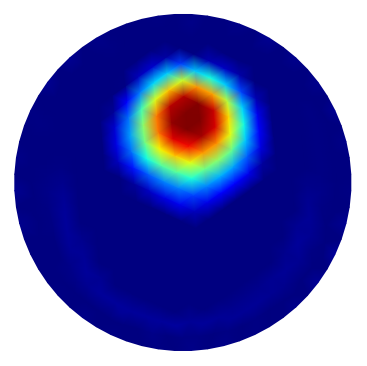} &
\includegraphics[width=0.07\textwidth]{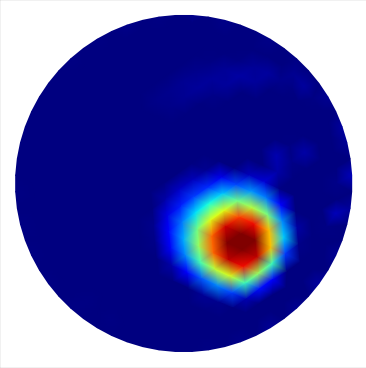} &
\includegraphics[width=0.07\textwidth]{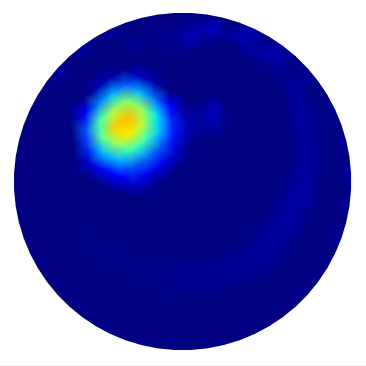} &
    \includegraphics[width=0.07\textwidth]{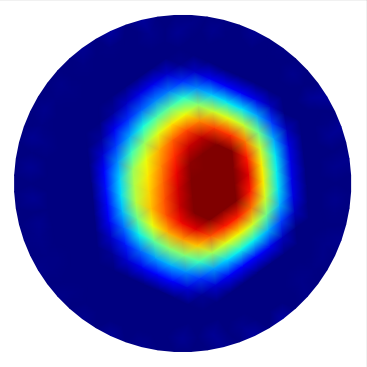} & 
    \includegraphics[width=0.07\textwidth]{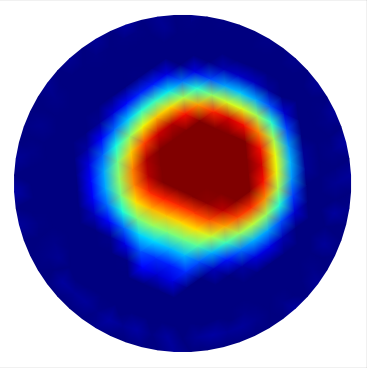} & 
    \includegraphics[width=0.07\textwidth]{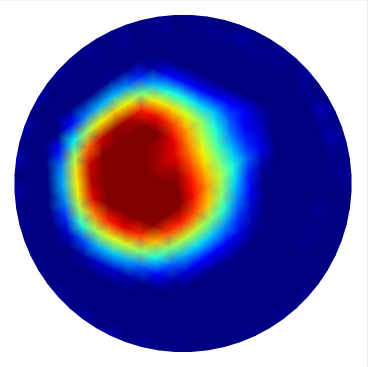} & 
    \includegraphics[width=0.07\textwidth]{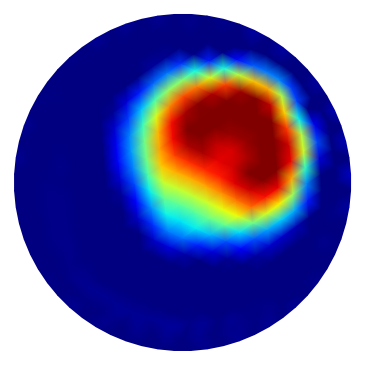} \\ &
\scriptsize{Err$_{f_2}$=0.414} &
\scriptsize{Err$_{f_2}$=0.419} &
\scriptsize{Err$_{f_2}$=0.576} &
\scriptsize{Err$_{f_2}$=0.427} &
\scriptsize{Err$_{f_2}$=0.323} &
\scriptsize{Err$_{f_2}$=0.310} &
\scriptsize{Err$_{f_2}$=0.321} \\
& 
\includegraphics[width=0.07\textwidth]{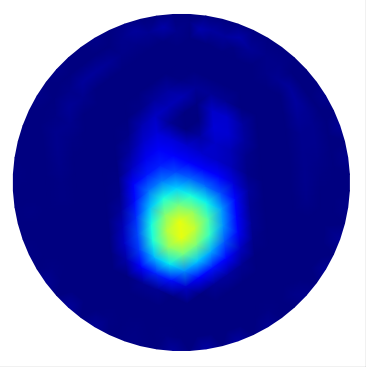} &
\includegraphics[width=0.07\textwidth]{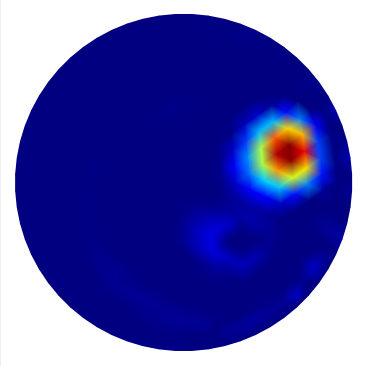} &
\includegraphics[width=0.07\textwidth]{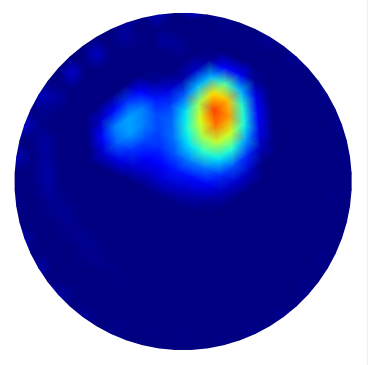} &
    \includegraphics[width=0.07\textwidth]{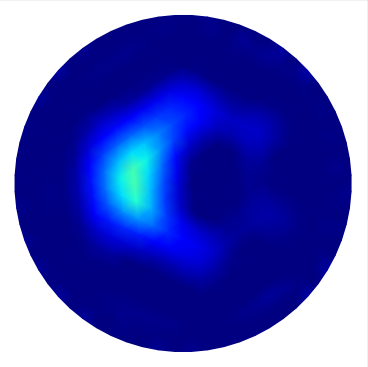} & 
    \includegraphics[width=0.07\textwidth]{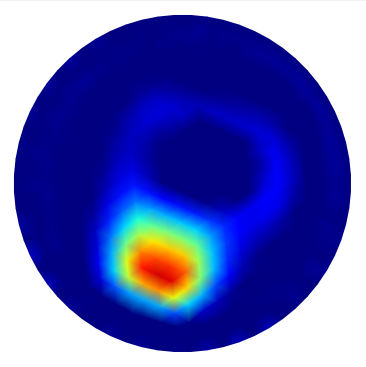} & 
    \includegraphics[width=0.07\textwidth]{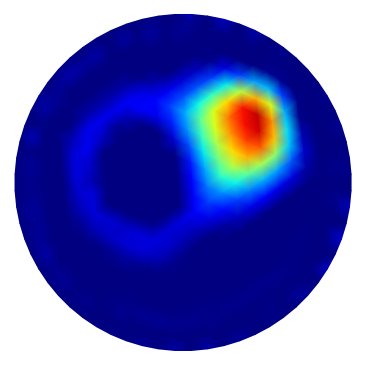} & 
    \includegraphics[width=0.07\textwidth]{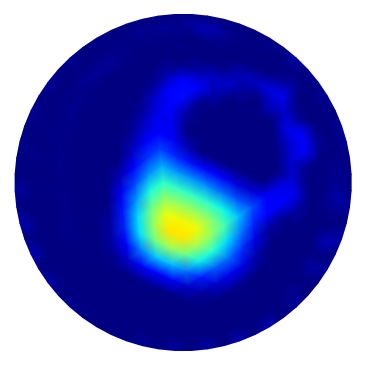} \\&
\scriptsize{Err$_{f_3}$=0.705} &
\scriptsize{Err$_{f_3}$=0.398} &
\scriptsize{Err$_{f_3}$=0.653} &
\scriptsize{Err$_{f_3}$=0.688} &
\scriptsize{Err$_{f_3}$=0.447} &
\scriptsize{Err$_{f_3}$=0.429} &
\scriptsize{Err$_{f_3}$=0.578} \\
\hline
FR-PRGN 
& \includegraphics[width=0.07\textwidth]{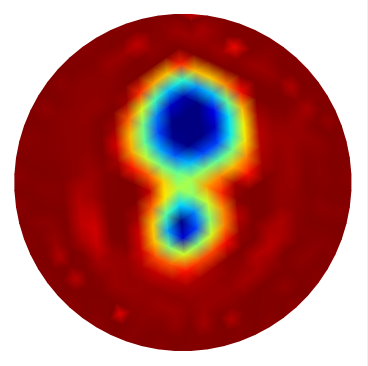} & 
\includegraphics[width=0.07\textwidth]{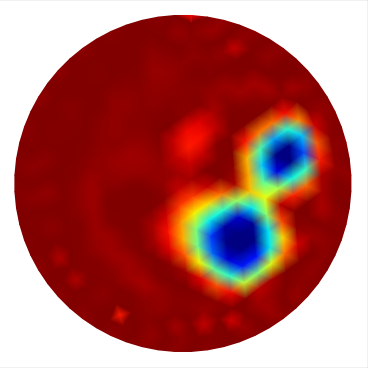} &
\includegraphics[width=0.07\textwidth]{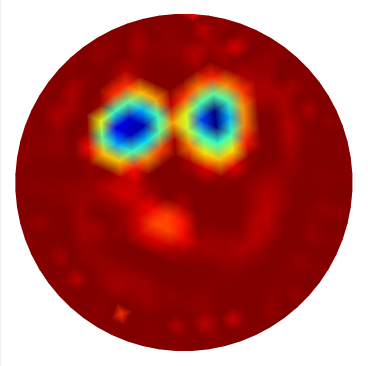} &
\includegraphics[width=0.07\textwidth]{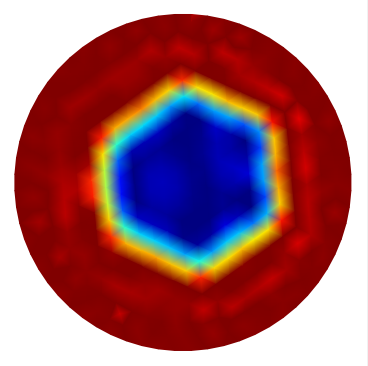} & 
\includegraphics[width=0.07\textwidth]{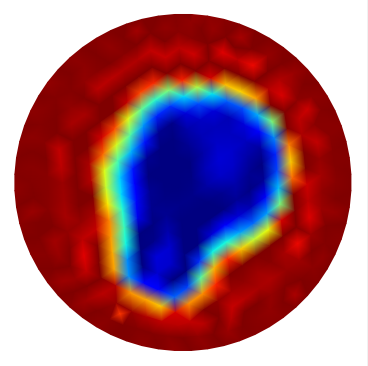} & 
\includegraphics[width=0.07\textwidth]{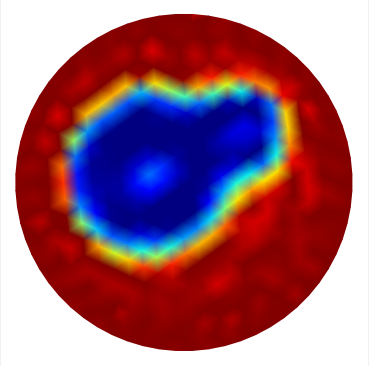} & 
\includegraphics[width=0.07\textwidth]{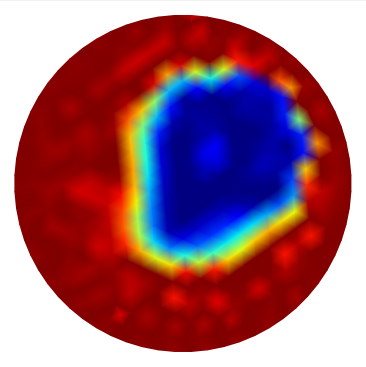} \\
& \scriptsize{Err$_{f_1}$=0.132} & \scriptsize{Err$_{f_1}$=0.122} & \scriptsize{Err$_{f_1}$=0.106} &
\scriptsize{Err$_{f_1}$=0.133}&
\scriptsize{Err$_{f_1}$=0.152}&
\scriptsize{Err$_{f_1}$=0.142}&
\scriptsize{Err$_{f_1}$=0.143} 
\\
& \includegraphics[width=0.07\textwidth]{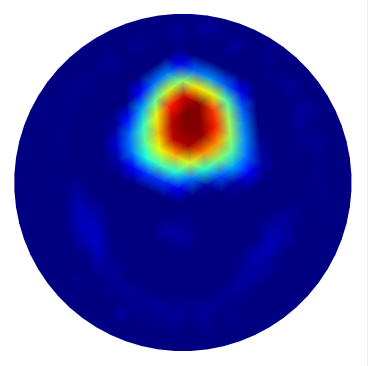} & 
\includegraphics[width=0.07\textwidth]{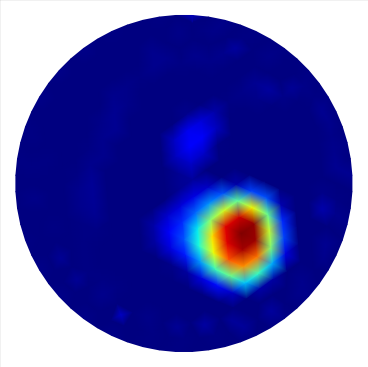} &
\includegraphics[width=0.07\textwidth]{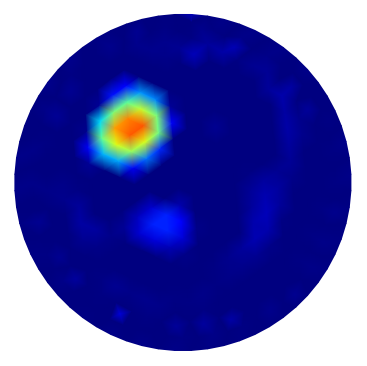} & 
\includegraphics[width=0.07\textwidth]{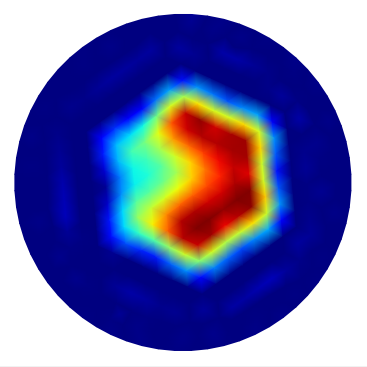}& 
\includegraphics[width=0.07\textwidth]{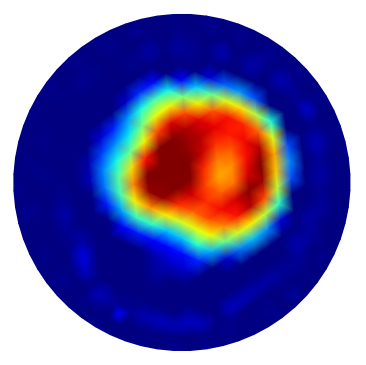}& 
\includegraphics[width=0.07\textwidth]{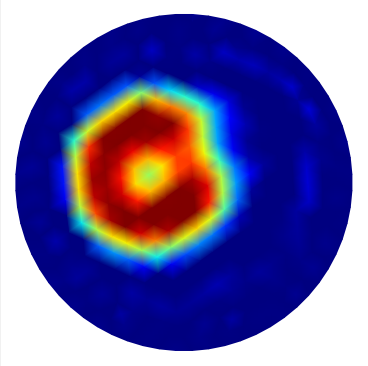}& 
\includegraphics[width=0.07\textwidth]{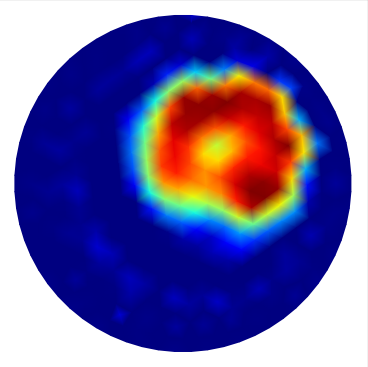}\\
& \scriptsize{Err$_{f_2}$=0.404} & \scriptsize{Err$_{f_2}$=0.430} & \scriptsize{Err$_{f_2}$=0.443} &
\scriptsize{Err$_{f_2}$=0.338} &
\scriptsize{Err$_{f_2}$=0.332}&
\scriptsize{Err$_{f_2}$=0.288}&
\scriptsize{Err$_{f_2}$=0.317}\\
& \includegraphics[width=0.07\textwidth]{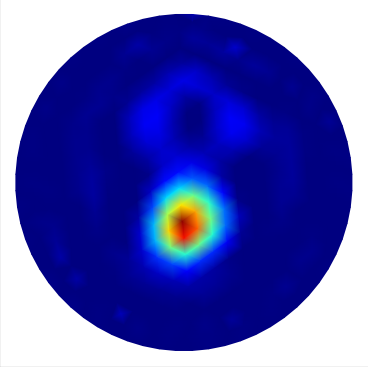} & 
\includegraphics[width=0.07\textwidth]{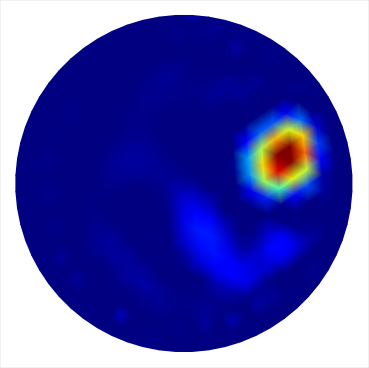} &
\includegraphics[width=0.07\textwidth]{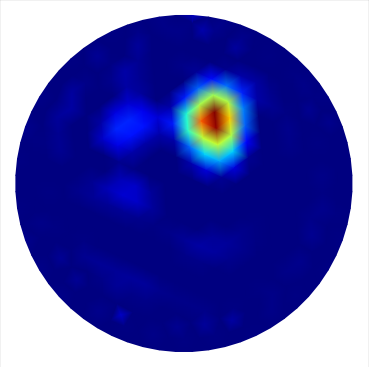} & 
\includegraphics[width=0.07\textwidth]{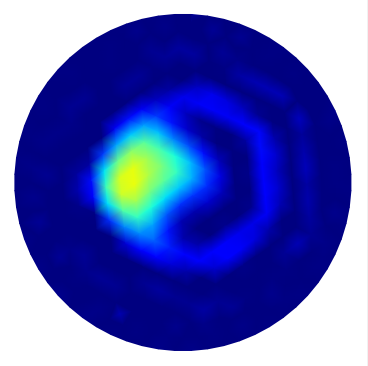}& 
\includegraphics[width=0.07\textwidth]{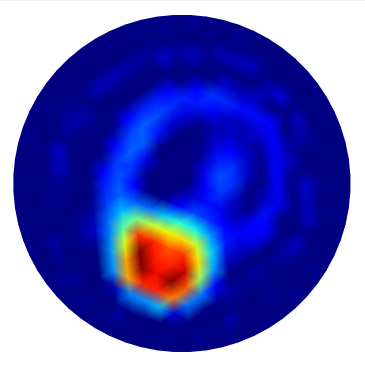}& 
\includegraphics[width=0.07\textwidth]{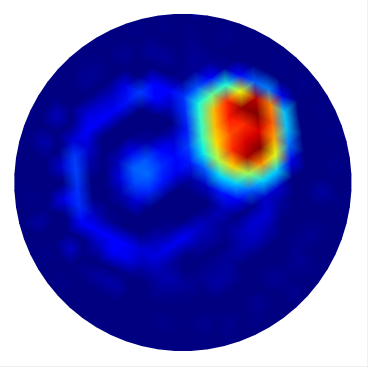}& 
\includegraphics[width=0.07\textwidth]{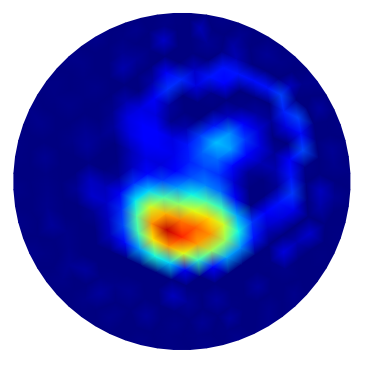}\\
& \scriptsize{Err$_{f_3}$=0.547} & \scriptsize{Err$_{f_3}$=0.449} & \scriptsize{Err$_{f_3}$=0.469} &
\scriptsize{Err$_{f_3}$=0.532} &
\scriptsize{Err$_{f_3}$=0.437}&
\scriptsize{Err$_{f_3}$=0.395}&
\scriptsize{Err$_{f_3}$=0.566}\\
\hline    
\end{tabular}
    \caption{Example 1: Fraction reconstructions comparison among FR-PRGN, mf-Net, and \cite{Malone2014} on a few test samples (column-wise).}
    \label{fig:Comp}
\end{figure*}

\subsection{Reconstruction examples}
\label{sec:RE}

In the following examples, we aim to validate the two proposed approaches – unrolled mf-Net and FR-PRGN – on synthetic data for both fraction reconstruction and conductivity reconstruction.

\paragraph{Example 1. Fraction Reconstruction.}

This example aims to validate the fraction reconstruction using MMV-based procedures.
We compared the performance of Algorithm 1 FR-PRGN, with the unrolled mf-Net - see $k$th iterative block in Algorithm \ref{alg:alg2} - and with the iterative method proposed in \cite{Malone2014}.  
Among the existing MMV-based mfEIT reconstruction methods, only the iterative algorithm proposed in \cite{Malone2014} addresses the direct reconstruction of fraction distributions, which then easily allows reconstructing the associated conductivities for each frequency. 
Other mfEIT approaches instead reconstruct only the different conductivities for each frequency, see \cite{MMVNET2023}. Moreover, no other methods but \cite{Malone2014} consider the co-existence of more than one tissue at a given region on the domain,  a necessary requirement for the applications to tissue engineering.

Fig. \ref{fig:Comp} provides both a visual and quantitative comparison of the recovered fractions for the three evaluated methods.
In each horizontal block, the first, second, and third rows show the first, second, and third reconstructed fraction, respectively, which corresponds to a different tissue. In the first three columns, we report a few results from the \texttt{No-Overlap} test dataset ($T=3,M=2$), while in the remaining four columns the results come from the \texttt{Overlap} test dataset ($T=3,M=2$).

All three methods demonstrate a reasonable capability to recover the underlying fractional distribution, broadly capturing the general shape and presence of the active regions compared to the Ground Truth (GT in the first row block).
Both mf-NET and FR-PRGN are highly effective in recovering the fractional distributions, often outperforming or matching the performance of the established variational method \cite{Malone2014}.
This is supported by the Err$_{f_i}$ values reported below each reconstruction, reinforcing their strong performance.

The variational method \cite{Malone2014} also shows good recovery, though in some instances, the recovered fractions might appear slightly more diffused or less sharply defined than those from the GT or mf-Net. However, it consistently identifies the correct locations of the tissue fractions.
 
FR-PRGN exhibits strong performance, often yielding recovered fractions that are distinct and accurately located. In several samples, its output is highly comparable to both mf-Net and GT, demonstrating its effectiveness in capturing the underlying distribution.

We achieve the best results applying mf-Net, obtaining accurately shaped recovered fractions. The distribution and intensity of the recovered fractions generally align well with the GT.

In Table \ref{tab:table1}, we aim to demonstrate the value added by our proposals, going beyond the basic estimate $\hat{F}$ generated by the algorithm in Section \ref{sec:FEST}, and shown in the first column.
In Table \ref{tab:table1} we report the average reconstruction errors both for the fraction values $\overline{Err}_{f_i}$ and for the conductivities $\overline{Err}_{\sigma_i}$, on the whole \texttt{Overlap} test dataset.
All three methods improve the estimated fractions obtained by the F-EST algorithm.  mf-Net outperforms the other methods, halving the reconstruction errors for the fractions.

\begin{table}
\caption{Average Error on \texttt{Overlap} test dataset}
\begin{center}
\begin{tabular}{c|cccc}
 & F-EST  & mf-Net & FR-PRGN & \cite{Malone2014}\\
\hline\\
$\overline{\rm Err}_{f_1}$  &0.2850 & \textbf{0.0732}& 0.1579& 0.1642\\
$\overline{\rm Err}_{f_2}$  &0.4482 & \textbf{0.1684}& 0.3523& 0.3425 \\
$\overline{\rm Err}_{f_3}$  & 0.8046& \textbf{0.3186}& 0.5128& 0.5124\\
$\overline{\rm Err}_{\sigma_1}$ & - &\textbf{0.0444} & 0.1003& 0.3621\\
$\overline{\rm Err}_{\sigma_2}$ & - & \textbf{0.0205}& 0.0345 & 0.1067
\end{tabular}
\end{center}
\label{tab:table1}
\end{table}

\begin{table}
\caption{Average Error on \texttt{No-Overlap} test dataset}
\begin{center}
\begin{tabular}{c|cc|cc}
 & PnP &  mf-Net & PnP \scriptsize{(Noise)} &  mf-Net \scriptsize{(Noise)}\\
\hline \\
$\overline{\rm Err}_{\sigma_1}$ & 
0.0818 & \textbf{0.0441} & 
 0.2526  & \textbf{0.0921} \\
$\overline{\rm Err}_{\sigma_2}$ & 0.1764 &\textbf{0.0973} & 
0.2996 & \textbf{0.1978} 
\end{tabular}
\end{center}
\label{tab:table2}
\end{table}

\paragraph{Example 2. Conductivity reconstruction.}

If our only goal was the reconstruction of the conductivity, SMV methods could be applied for each frequency individually. However, to highlight the advantage of using MMV approaches—which account for the correlation between various tissue fractions—we are comparing our mf-Net against the Plug-and-Play (PnP) strategy for EIT conductivity reconstruction presented in \cite{Col2023DeepplugandplayPG}, as our chosen SMV benchmark.

In this example, we can only consider the data from the \texttt{No-Overlap} dataset ($T=4,M=2$), which can naturally be processed by those MMV methods that consider only 0/1 fractions and SMV methods.

Figure \ref{fig:PnP} presents a visual comparison of the recovered conductivities $\sigma_1$ and $\sigma_2$ with their respective GT, for three selected samples.
We notice that at frequency $\omega_1$ the tissue potato is not visible, as it presents a conductivity spectrum similar to the saline background.
The performances are evaluated against the GT, reporting the Err$_{\sigma_i}$ values below each reconstructed result.
Based on both visual inspection and the provided error metrics, the mf-Net method consistently outperforms the PnP method in reconstructing the shown conductivities. mf-NET yields reconstructions that are much closer to the GT, exhibiting clearer boundaries and more accurate representations of the original structures, while PnP produces significantly blurred and less accurate results.
Table \ref{tab:table2} clearly shows that mf-Net consistently outperforms PnP across both evaluated average error metrics on the \texttt{No-Overlap} test dataset. In particular, mf-Net achieves almost half the error of PnP, indicating superior accuracy and reconstruction quality in this specific testing scenario. 
 
\begin{figure*}[h]
    \centering
    \begin{tabular}{cc|cc|cc}
    Sample & GT & mf-Net & PnP  &  mf-Net \scriptsize{(Noise)} & PnP \scriptsize{(Noise)} \\
     \hline   
& \includegraphics[width=0.07\textwidth]{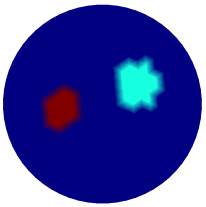} & 
\includegraphics[width=0.07\textwidth]{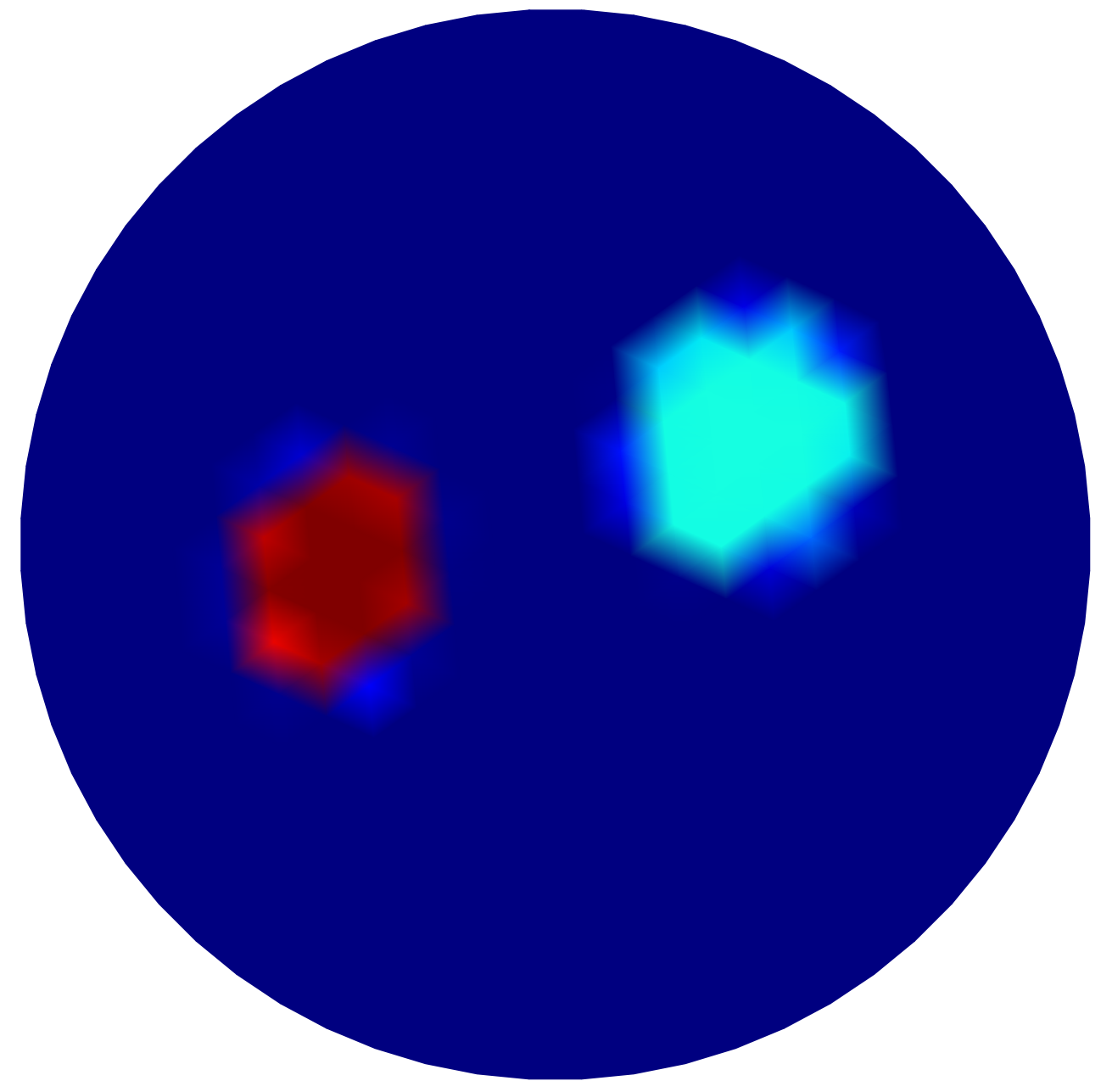} &
\includegraphics[width=0.07\textwidth]{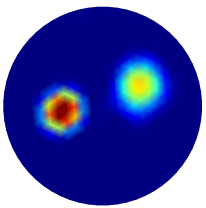} &
\includegraphics[width=0.07\textwidth]{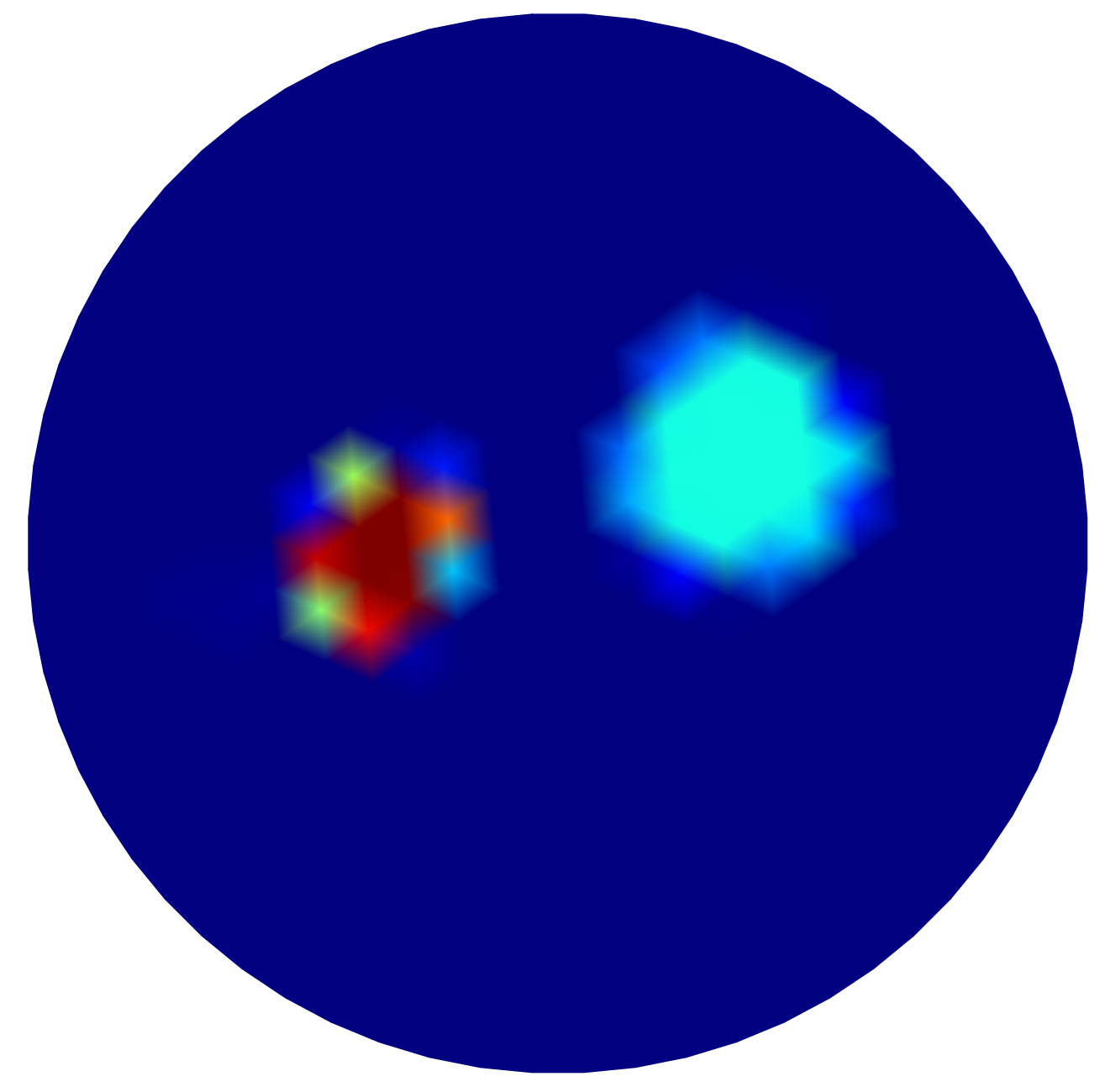} &
\includegraphics[width=0.07\textwidth]{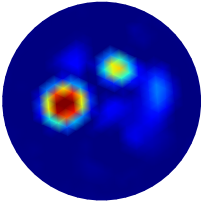} 
\\
& &
\scriptsize{\scriptsize{Err$_{\sigma_1}$=0.015}} &
\scriptsize{\scriptsize{Err$_{\sigma_1}$=0.070}} & 
\scriptsize{\scriptsize{Err$_{\sigma_1}$=0.107}} & 
\scriptsize{\scriptsize{Err$_{\sigma_1}$=0.116}}  
\\
& \includegraphics[width=0.07\textwidth]{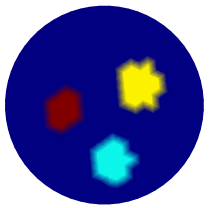} &
\includegraphics[width=0.07\textwidth]{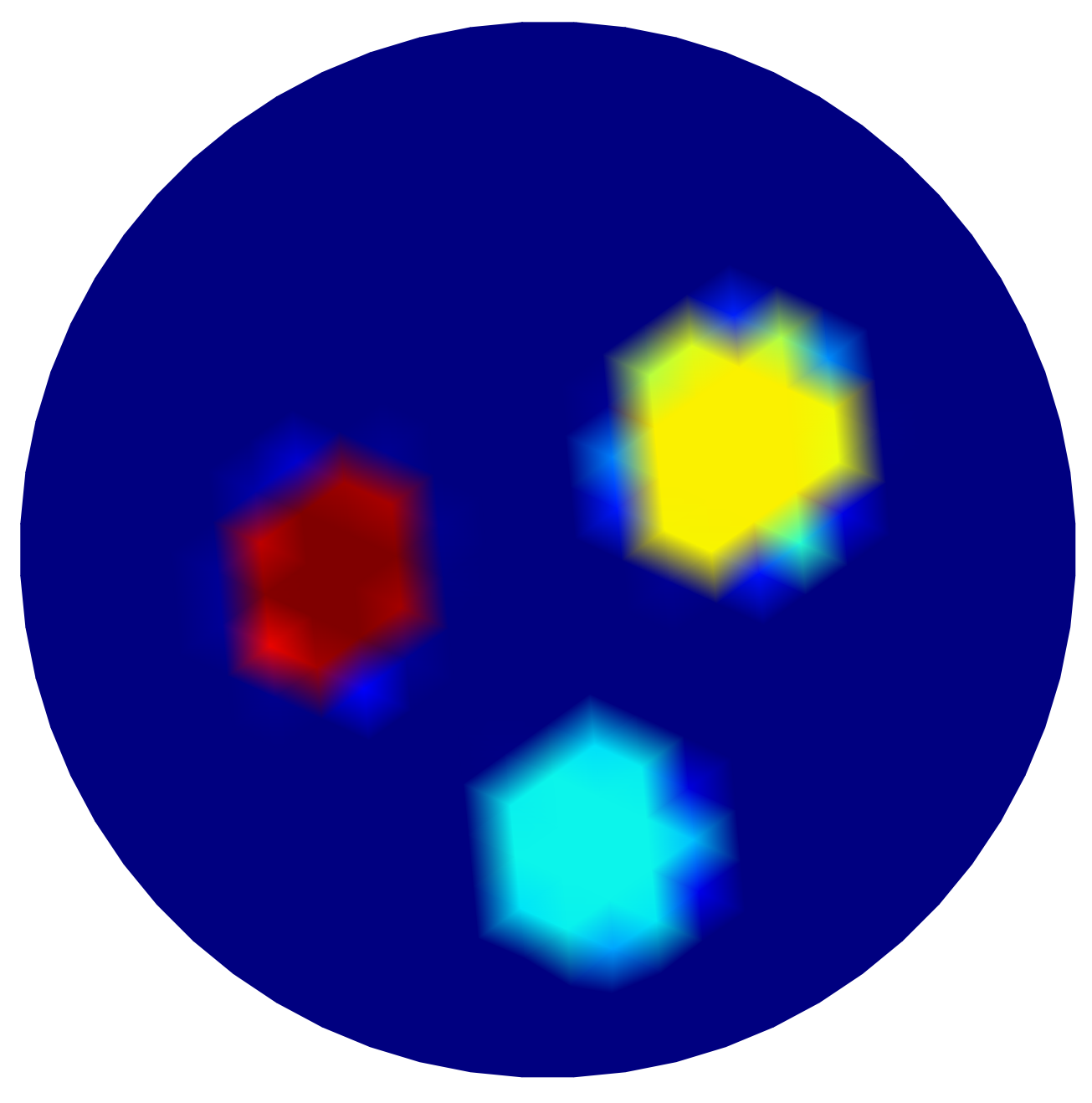} & 
\includegraphics[width=0.07\textwidth]{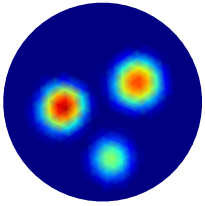} &
\includegraphics[width=0.07\textwidth]{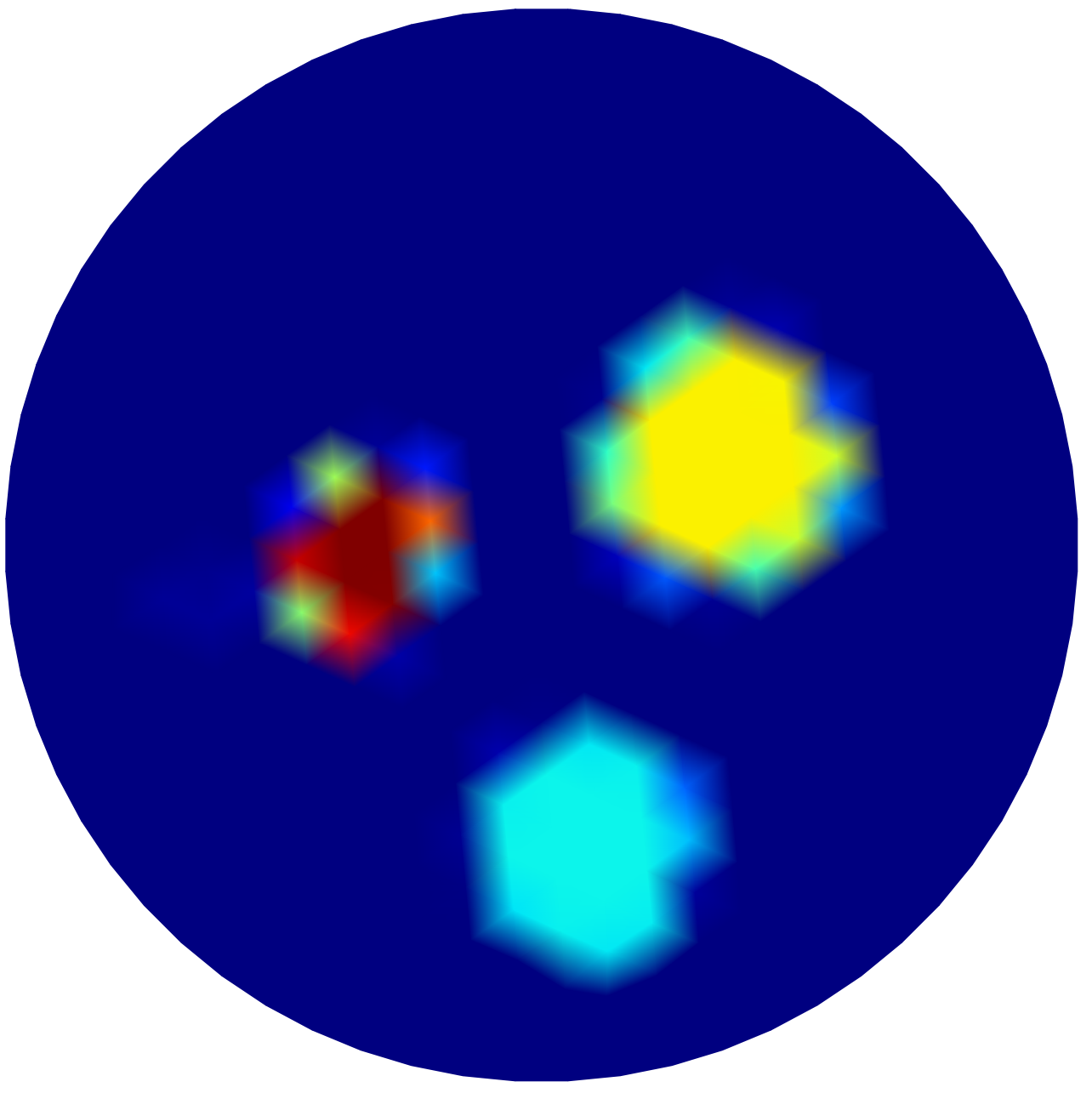} & 
\includegraphics[width=0.07\textwidth]{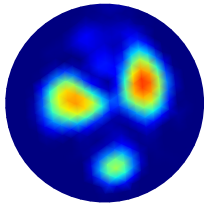} 
\\
\scriptsize{23}& &
\scriptsize{\scriptsize{Err$_{\sigma_2}$=0.048}} &
\scriptsize{\scriptsize{Err$_{\sigma_2}$=0.168}} & 
\scriptsize{\scriptsize{Err$_{\sigma_2}$=0.222}} & 
\scriptsize{\scriptsize{Err$_{\sigma_2}$=0.266}}  
\\
\hline
& \includegraphics[width=0.07\textwidth]{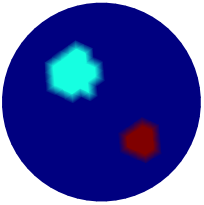} & 
\includegraphics[width=0.07\textwidth]{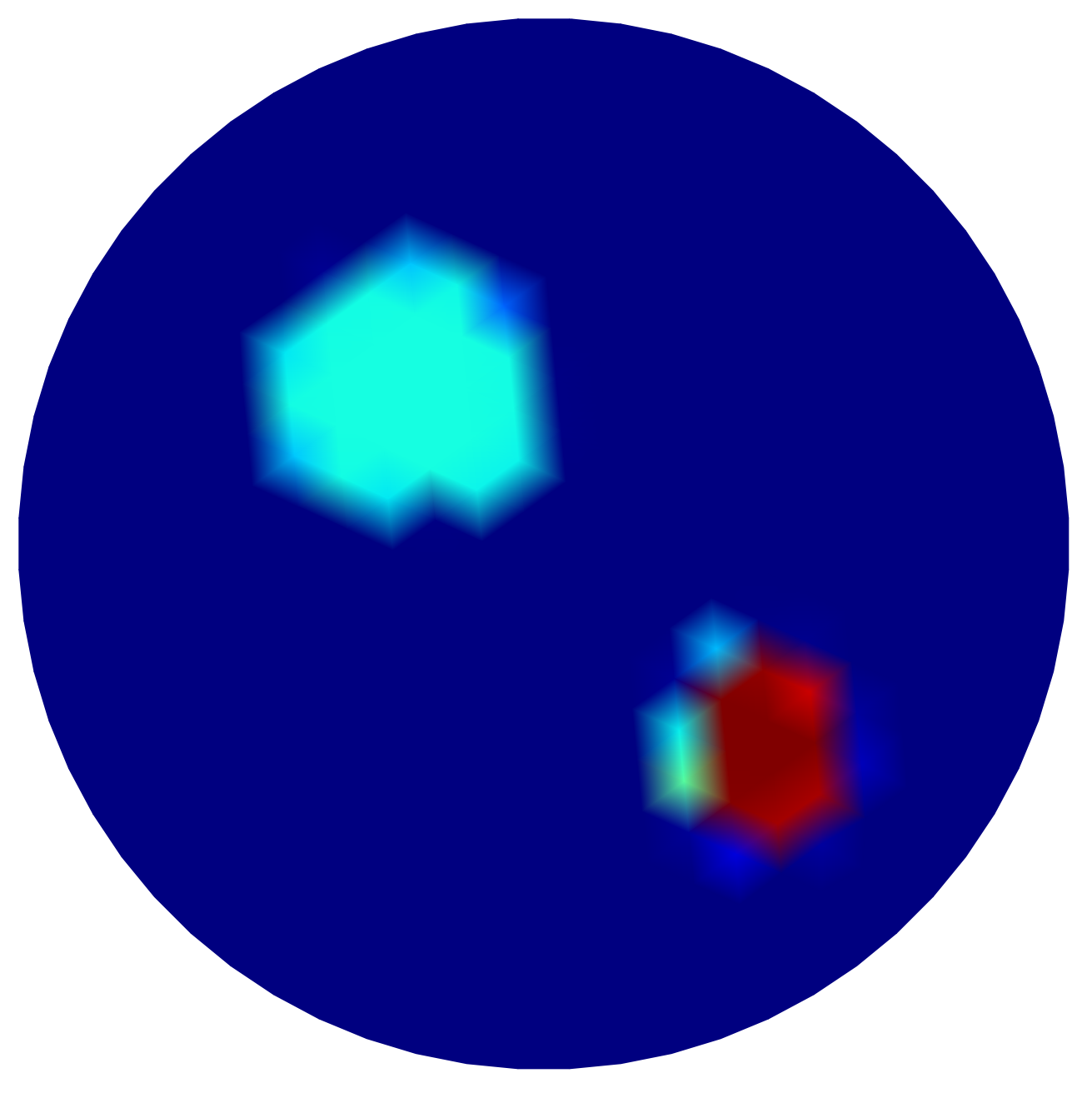} &
\includegraphics[width=0.07\textwidth]{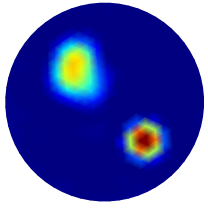} &
\includegraphics[width=0.07\textwidth]{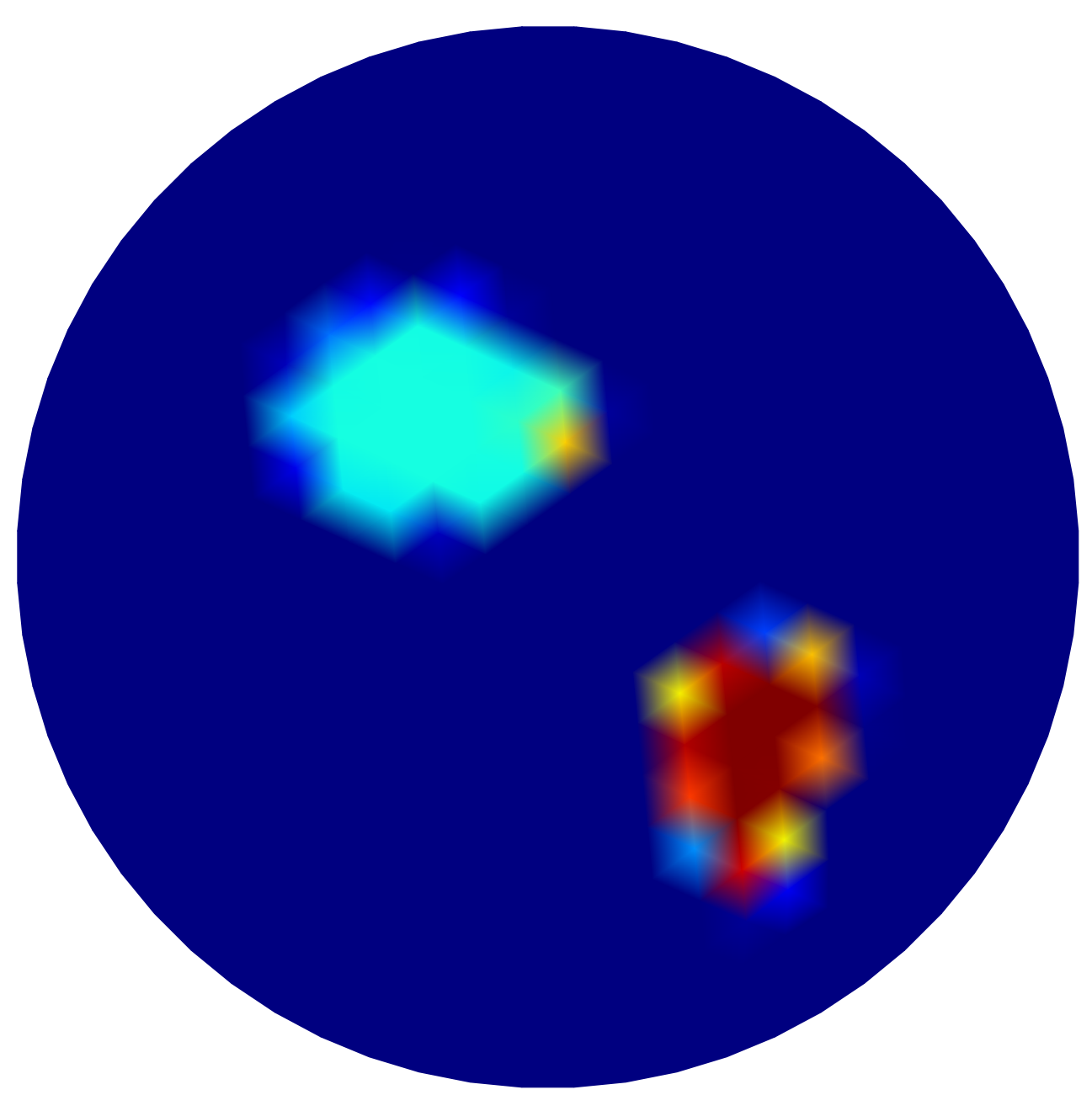} &
\includegraphics[width=0.07\textwidth]{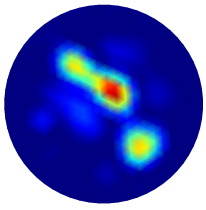} 
\\
& & \scriptsize{\scriptsize{Err$_{\sigma_1}$=0.042}} &
\scriptsize{\scriptsize{Err$_{\sigma_1}$=0.076}} & 
\scriptsize{\scriptsize{Err$_{\sigma_1}$=0.097}} & 
\scriptsize{\scriptsize{Err$_{\sigma_1}$=0.144}}

\\
& \includegraphics[width=0.07\textwidth]{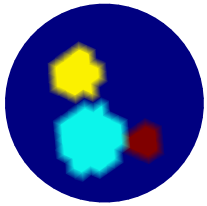} &
\includegraphics[width=0.07\textwidth]{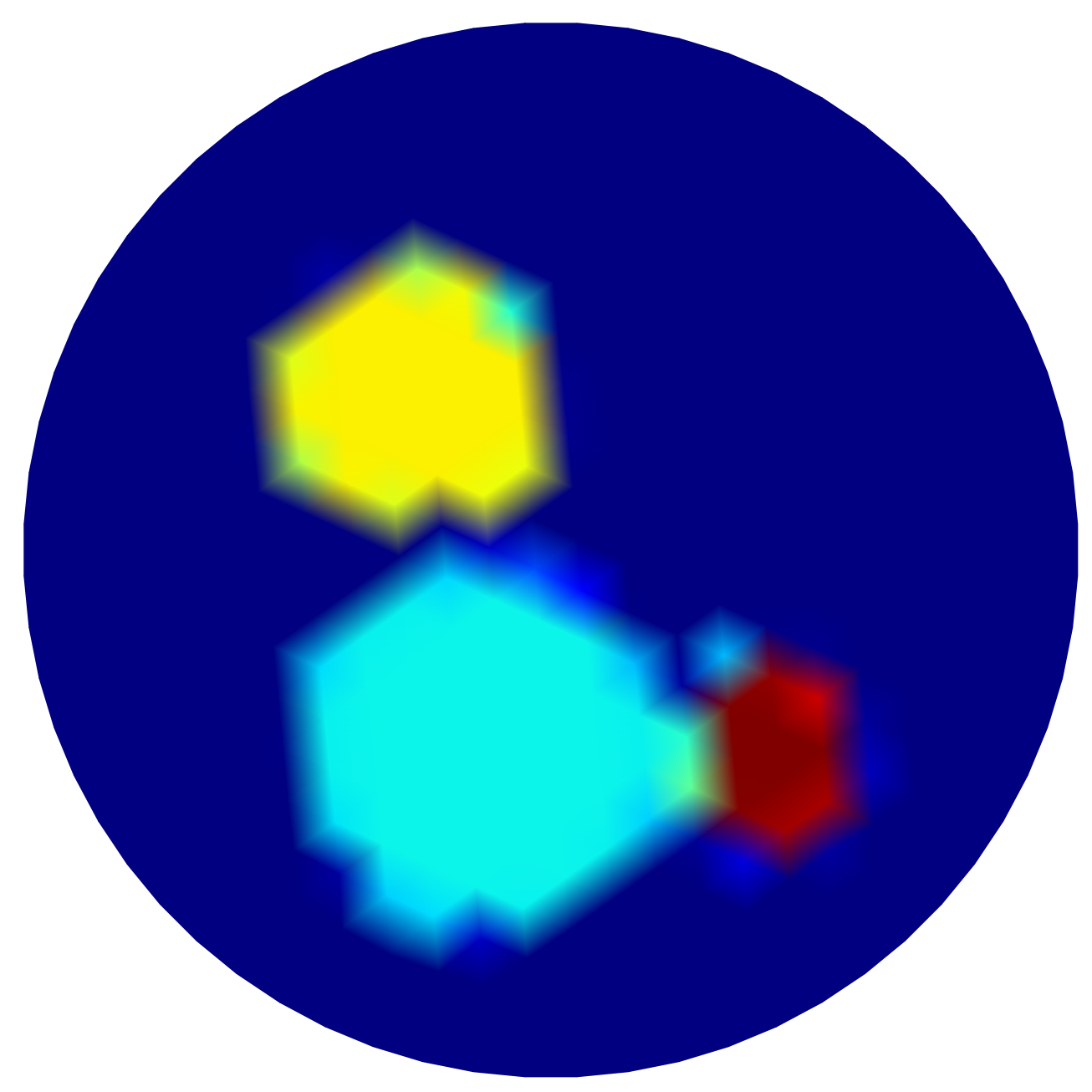} & 
\includegraphics[width=0.07\textwidth]{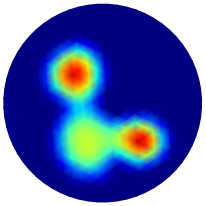} &
\includegraphics[width=0.07\textwidth]{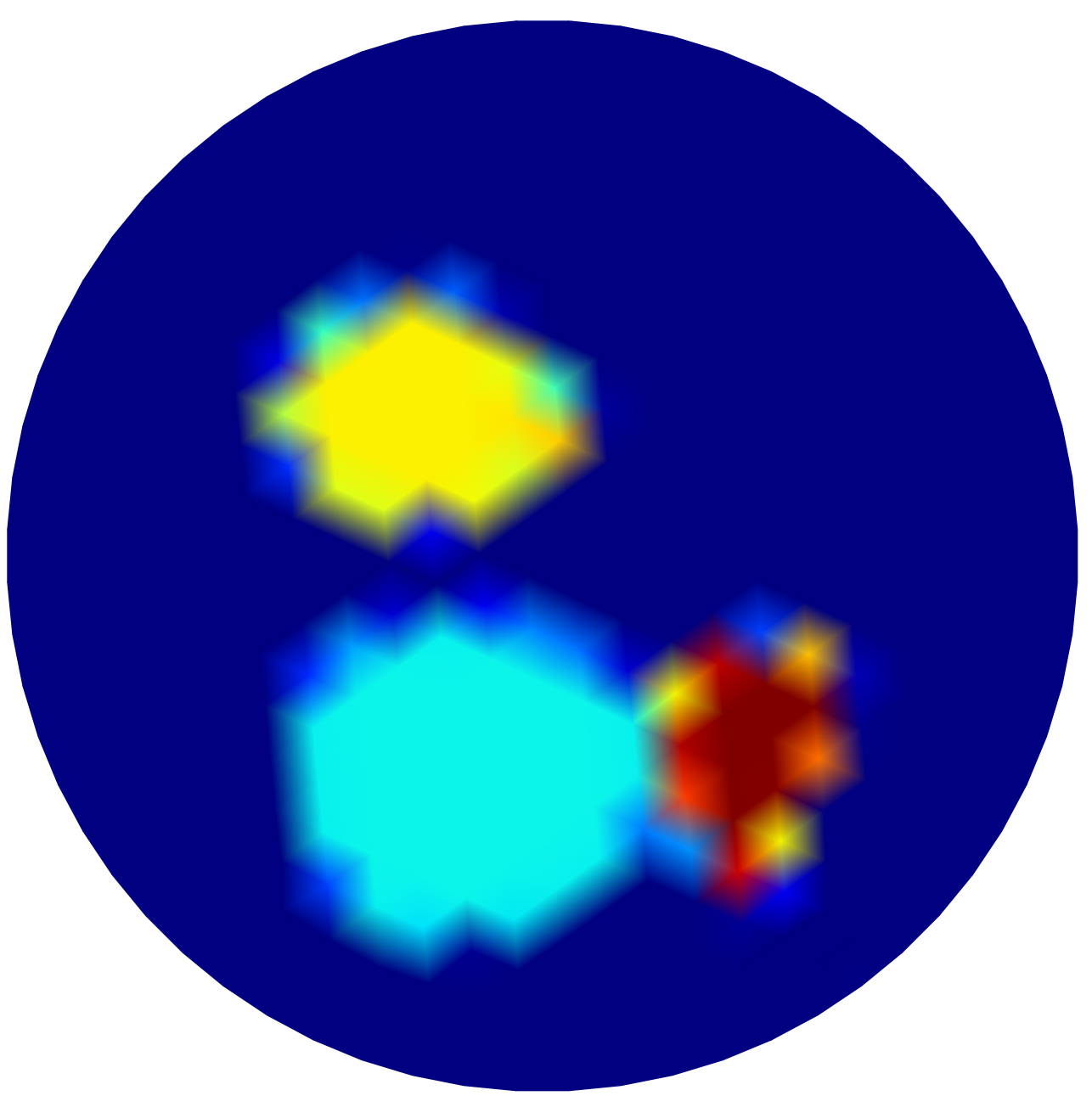} & 
\includegraphics[width=0.07\textwidth]{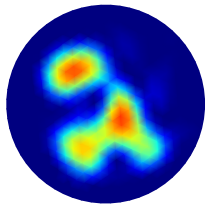} 
\\
\scriptsize{39}& & \scriptsize{\scriptsize{Err$_{\sigma_2}$=0.095}} &
\scriptsize{\scriptsize{Err$_{\sigma_2}$=0.178}} & 
\scriptsize{\scriptsize{Err$_{\sigma_2}$=0.207}} & 
\scriptsize{\scriptsize{Err$_{\sigma_2}$=0.296}}
\\
\hline
& \includegraphics[width=0.07\textwidth]{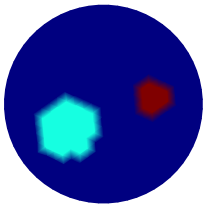} & 
\includegraphics[width=0.07\textwidth]{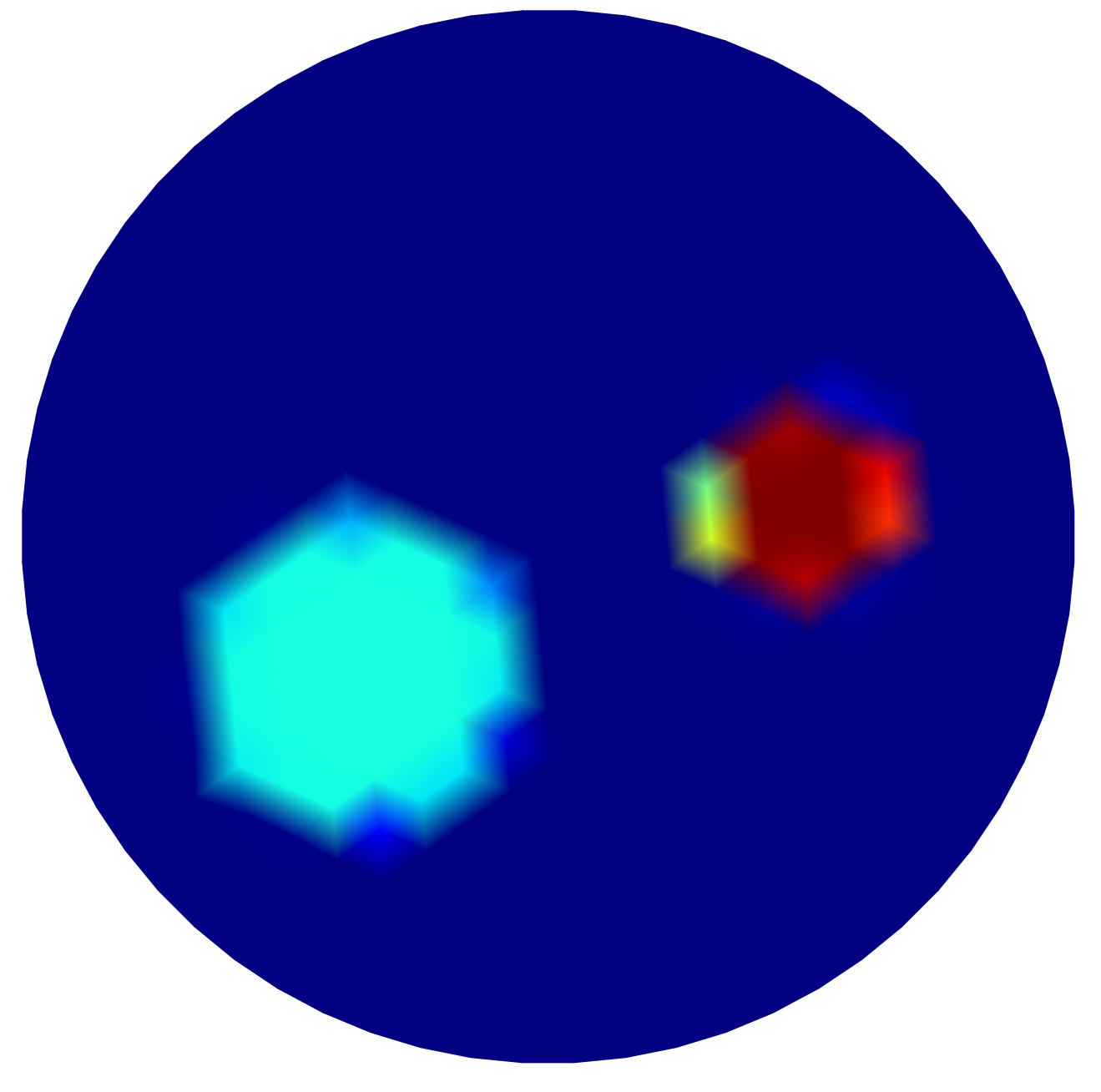} &
\includegraphics[width=0.07\textwidth]{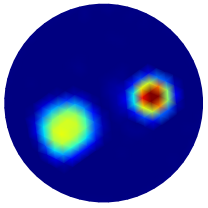} &
\includegraphics[width=0.07\textwidth]{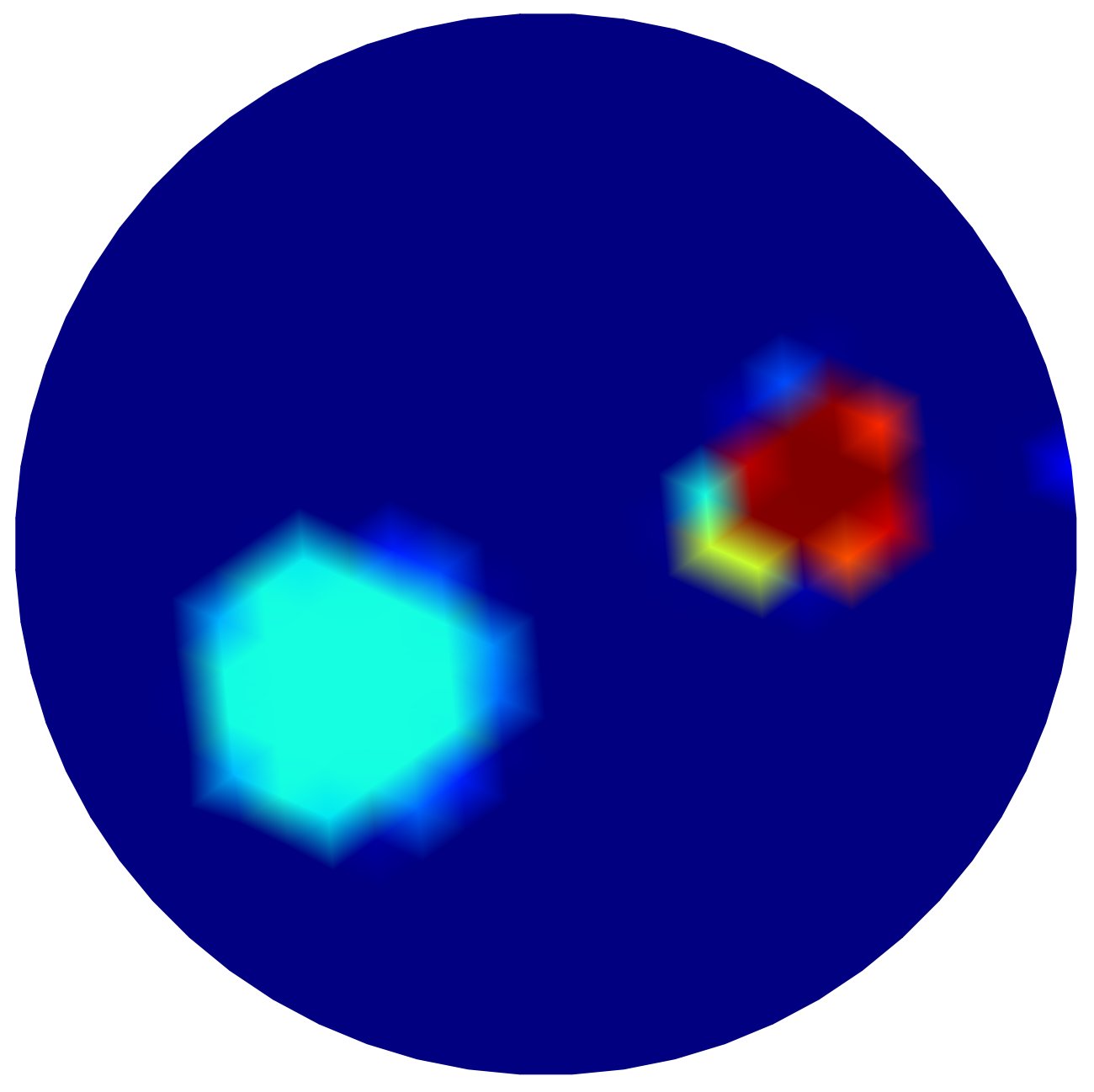} &
\includegraphics[width=0.07\textwidth]{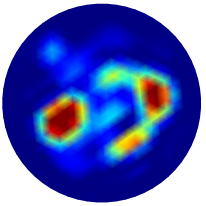} 
\\
& & \scriptsize{\scriptsize{Err$_{\sigma_1}$=0.034}} &
\scriptsize{\scriptsize{Err$_{\sigma_1}$=0.071}} & 
\scriptsize{\scriptsize{Err$_{\sigma_1}$=0.082}} & 
\scriptsize{\scriptsize{Err$_{\sigma_1}$=0.182}} \\
& \includegraphics[width=0.07\textwidth]{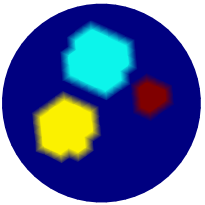} &
\includegraphics[width=0.07\textwidth]{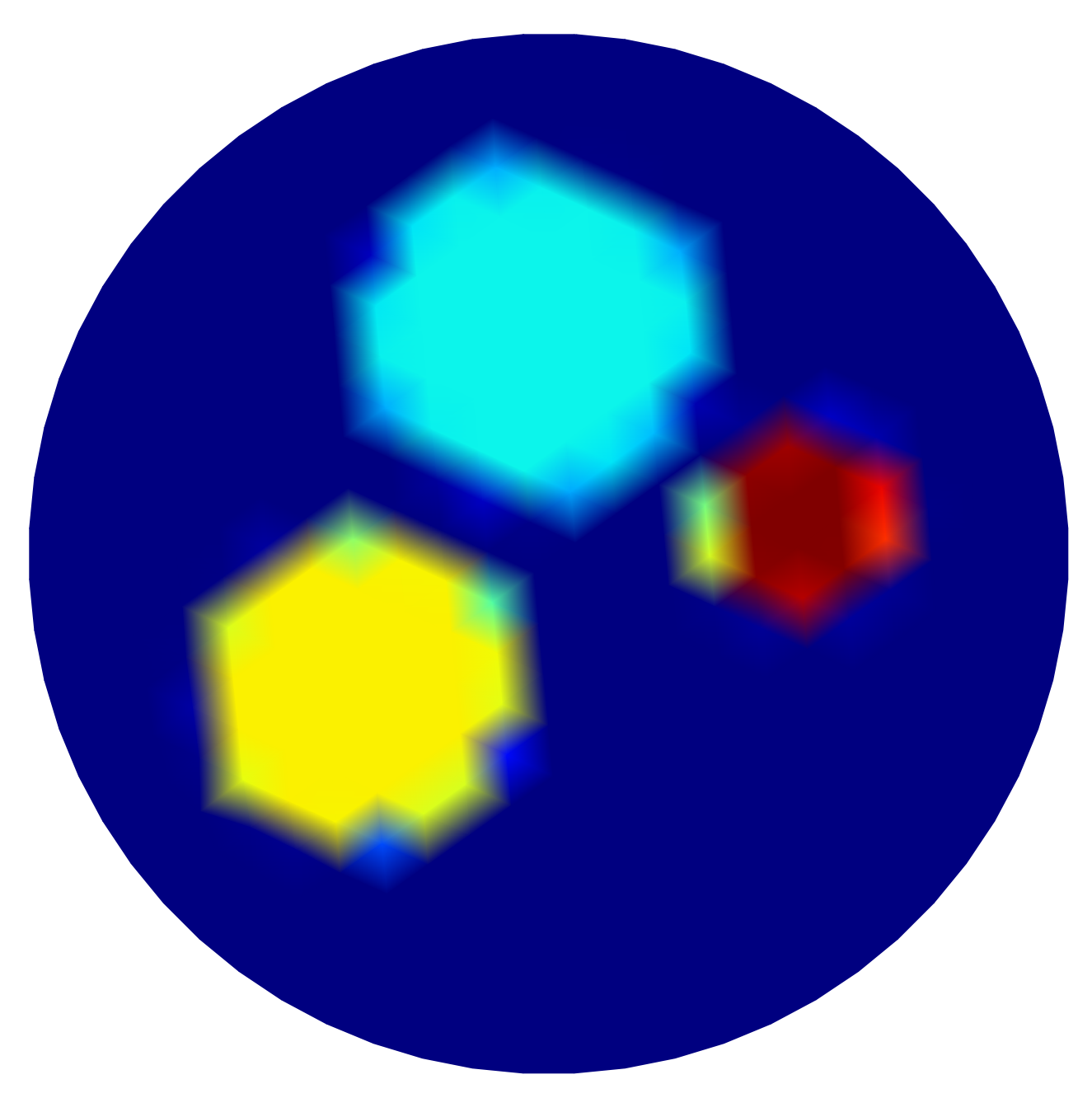} & 
\includegraphics[width=0.07\textwidth]{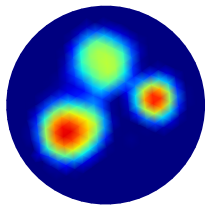} &
\includegraphics[width=0.07\textwidth]{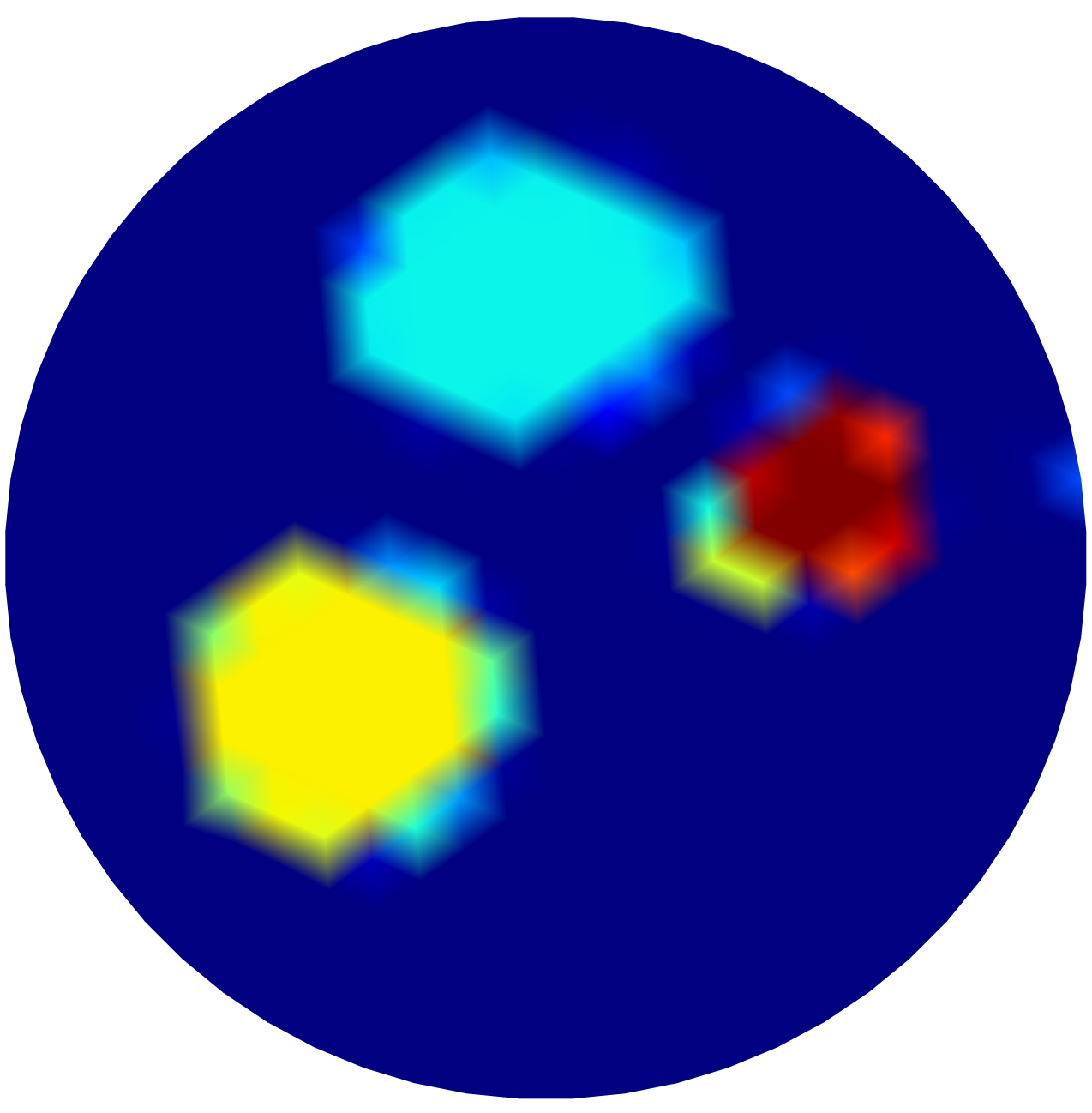} & 
\includegraphics[width=0.07\textwidth]{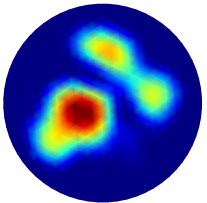} 
\\
\scriptsize{42}& & \scriptsize{\scriptsize{Err$_{\sigma_2}$=0.068}} &
\scriptsize{\scriptsize{Err$_{\sigma_2}$=0.176}} & 
\scriptsize{\scriptsize{Err$_{\sigma_2}$=0.186}} & 
\scriptsize{\scriptsize{Err$_{\sigma_2}$=0.304}} \\
\hline
\end{tabular}
\caption{Example 2: Conductivity reconstructions comparison among mf-Net and PnP on three test samples from the \texttt{No-Overlap} dataset ($T=4,M=2$).}
    \label{fig:PnP}
 \end{figure*}

Practical EIT measurements are inherently susceptible to noise, which can easily contaminate the acquired voltage data. To assess the robustness of our proposed method against such a perturbation, we introduced additive white Gaussian noise to the voltage measurements generated by our forward model.
In particular, according to the degradation model \eqref{eq:IP}, we generated the  noisy measurement vector $y$ 
	by adding to the noise-free  measurements $\widetilde{y}$  a vector	$ \eta\sim\mathcal{N}(0,\delta \overline{|\widetilde{y}|})$
	of Gaussian noise, where $\overline{|\widetilde{y}|}$ represents the average of the absolute measured voltage values, and $\delta \in \R_{++}$ the noise level. 
The measurements have been corrupted by a noise level $\delta=5 \times 10^{-3}$ which corresponds to SNR=54dB.

Table \ref{tab:table2} presents the average conductivity reconstruction values for our entire test dataset. We show results for noise-free test cases (first and second columns) and for degraded measurements (third and fourth columns), using both PnP and mf-Net trained exclusively on noise-free samples.
Figure \ref{fig:PnP} illustrates how degraded measurements affect the reconstruction of conductivities $\sigma_1$ and $\sigma_2$ for three selected test samples.
The results in Table \ref{tab:table2} and visual inspection of Fig.\ref{fig:PnP} both indicate that the mf-Net is significantly more robust to noise, partly due to its exploitation of the correlation in frequency-dependent measurements.

\begin{table}[!htbp]
    \centering
    \caption{Ablation Study}
    \label{tab:MSE_layer_actfun_hf}
    \begin{tabular}{c|c|c}
    \texttt{Iter Block K} & \texttt{$h_f$} & $\overline{Err}_F^K$\\
    \hline
    3 & 32 & $  0.3233 $ \\
    - & 64 & 0.3190 \\
     - & 128 & 0.3031 \\
    \hline
    5 & 32 & 0.2648 \\
    - & 64 & 0.2255\\
     - & 128 & 0.2776\\
    \hline
    7 & 32 & 0.2692 \\
    - & 64 & 0.2242\\
     - & 128 & 0.2521\\
    \hline
    \bf{9} & 32 & 0.2160 \\
    - & 64 &  0.2094  \\
     - & \bf{128} &\bf{0.1989} \\ 
      \end{tabular}
\end{table}

\subsection{Ablation studies}
\label{sec:AS}

In order to determine the optimal network architecture, three ablation studies were conducted. The first study focused on the number of iterative blocks \texttt{Iter Blocks K}, while the second and third studies were related to the GU-Net architecture. In particular, we analyzed the number of features in each hidden layer \texttt{$h_f$} and the depth \texttt{Depth}.
To assess the performance of mf-Net, we focused on scenarios with $M=2$ frequencies and $T=3$ tissues. 
In Table \ref{tab:MSE_layer_actfun_hf} we report the ablation results.

We created distinct mf-Net models, specifically mf-Net$_K$, where $K$ indicates the fixed number of blocks used during the deep iterative processing. We tested models with $K=3,5,7$, and $9$ blocks. A test dataset containing 50 samples was used for quantitative evaluations of these mf-Net$_K$ variants.
For each model, we compute
$$\overline{Err}_F^K= \frac{1}{T} \sum_{i=1}^T\overline{\rm Err}_{f_i}^K$$
where $\overline{\rm Err}_{f_i}^K$ represents the mean of the fraction recovery error of ${\rm Err}_{f_i}$ over all samples in the dataset. 

Our ablation study clearly shows how the number of layers $K$ and hidden features $h_f$
influence the error metric. As detailed in Table \ref{tab:MSE_layer_actfun_hf}, for a fixed GU-Net depth (3+3), increasing $K$ generally improves model performance by reducing the error. Similarly, increasing $h_f$  from 32 to 128 typically leads to better performance.

Optimal performance seems to occur  with $K = 9$ and $h_f$
values around 64 or 128, with an absolute lowest error of 0.1989  achieved with $h_f =128$. However, the modest improvement comes at the cost of significantly greater computational complexity (about $10\%$). 
Therefore, to prioritize computational efficiency in our experiments, we consider as sub-optimal baseline the one with $K=9$ and $h_f = 64$.

Established this baseline, we then investigated how increasing the encoder depth in the GU-Net autoencoder affects the error metric. We found that increasing the depth from 3+3 to 5+5 layers did not significantly alter the error, which shifted only slightly from 0.2094 to 0.2021. This allows us to confirm the baseline with $K=9$, $h_f = 64$ and \texttt{Depth}$=3+3$ (in the notation of Section~\ref{sec:GUNET}, this corresponds to $P=3$), which have been used for the experiments described in Section~\ref{sec:RE}. 

\section{Conclusion}\label{sec:conclusions}
We presented a model-based supervised learning method for mfEIT image reconstruction, named mf-Net. This approach leverages a fraction model for capturing complex intra-frequency correlations and overlapped inclusions. The proposed mf-Net variational network offers high efficiency, although it only approximately minimizes the underlying FR-PRGN variational model. Our numerical experiments, focusing on both fraction and conductivity mfEIT reconstructions, reveal that this approximate minimization consistently improves performance, and this enhanced data exploitation enables exceptionally accurate results. Looking ahead, we plan to validate our method on real-world experiments and to extend this approach to 3D multifrequency EIT imaging.

\section*{Appendix}

\begin{proof}[Proof of Lemma \ref{lem:entropy}]
We follow the same technique as in \cite[Proposition 5.1]{B2008}. By the definition of $\psi$ in \eqref{eq:entropy}, we have
\[
\langle \nabla \psi(F)- \nabla \psi(G),F-G \rangle = \sum_{n,j}(f_{nj}-g_{nj})\ln \frac{f_{nj}}{g_{nj}}.
\]
Since the function $\varphi\colon \R \rightarrow \R$ given by $\psi(t) = (t-1)\ln(t)-2\frac{(t-1)^2}{t+1}$ is non-negative, considering $t = \frac{f_{nj}}{g_{nj}}$ we get
\[
 \sum_{n,j}(f_{nj}-g_{nj})\ln \frac{f_{nj}}{g_{nj}}
 \geq \sum_{n,j} \frac{2(f_{nj}-g_{nj})^2}{f_{nj}+g_{nj}}.
\]
For any $n=1,\ldots,N$, consider now that
\[
\begin{aligned}
    \left(\sum_{j}|f_{nj}-g_{nj}|\right)^{\!\!2}
    &= \!\left(\sum_{j} \left(\frac{f_{nj}+g_{nj}}{2}\right)^{\!\!\frac{1}{2}} \frac{|f_{nj}-g_{nj}|}{ \left(\frac{f_{nj}+g_{nj}}{2} \right)^{\!\!\frac{1}{2}} } \right)^{\!\!2} \\
    &\leq \sum_j \frac{f_{nj}+g_{nj}}{2} \sum_j \frac{2(f_{nj}-g_{nj})^2}{f_{nj}+g_{nj}}
\end{aligned}
\]
and since $\Gamma$ is convex we have $\frac{F+G}{2} \in \Gamma$, thus $\sum_{j}\frac{f_{nj}+g_{nj}}{2}=1$ for all $n$. In conclusion,
\[
\begin{aligned}
\| F-G\|_1^2 & \leq N \sum_{n} \left(\sum_{j}|f_{nj}-g_{nj}|\right)^2 \\
& \leq N \langle \nabla \psi(F)- \nabla \psi(G),F-G \rangle.    
\end{aligned}
\]
To obtain the expression \eqref{eq:entropystar}, we start from the definition of the Fenchel conjugate, namely:
\[
\begin{aligned}
\psi^\star(F) &= \max_{G} \left\{ \langle G,F \rangle -\psi(G)\right\}\\
&= \max_{G \in \Gamma} \left\{ \sum_{n,j} g_{nj}f_{nj} - g_{nj}\ln(g_{nj}) \right\}.
\end{aligned}
\]
We introduce $N$ Lagrange multipliers $\mu_n$ associated with the constraints $\sum_{j}g_{nj} = 1$ and tackle the unconstrained problem
\[
\max_{G,\mu_1,\ldots,\mu_N}\!\! \left\{ \!\sum_{n,j} g_{nj}f_{nj} - g_{nj}\ln(g_{nj}) + \!\sum_n \mu_n \!\!\left(\!\sum_j g_{nj}-1\!\right)\!\!\right\}.
\]
Imposing first-order optimality conditions, we get
\[
\begin{aligned}
    f_{nj} - \ln(\hat{g}_{nj}) - 1 + \hat{\mu}_n &= 0 \quad \forall n,j \\
\sum_j \hat{g}_{nj} &= 1 \quad \forall n
\end{aligned}
\]
from which we deduce the expression of the maximum point
\[
\begin{aligned}
    \hat{g}_{nj} = e^{f_{nj}+\hat{\mu}_n-1}, \quad \hat{\mu}_n = 1-\ln\left(\sum_{j} e^{f_{nj}}\right)
\end{aligned}
\]
and, finally,
\[
\begin{aligned}
    \psi^\star(F) &= \sum_{n,j} \hat{g}_{nj}(f_{nj}-\ln(\hat{g}_{nj})) = \sum_{n}(1-\hat{\mu}_n)
\end{aligned}
\]
from which we get the desired expression.
\end{proof}

\section[*]{Acknowledgements}
The research of SM has been funded by PNRR CN-HPC, under the NextGeneration EU program CUP J33C22001170001. 
The work of SM, DL, and LR was in part supported by INDAM-GNCS 2025 projects, PRIN2022\_MORIGI, titled "Inverse Problems in the Imaging Sciences (IPIS)" 2022 ANC8HL - CUP J53D23003670006, and PRIN2022\_PNRR\_CRESCENTINI CUP J53D23014080001.
The research of LR has been funded by PNRR - M4C2 - Investimento 1.3. Partenariato Esteso PE00000013 - ``FAIR - Future Artificial Intelligence Research'' - Spoke 8 ``Pervasive AI'', which is funded by the European Commission under the NextGeneration EU programme. 
Co-funded by the European Union (ERC, SAMPDE, 101041040 – Next Generation EU, Missione 4 Componente 1 CUP D53D23005770006 and CUP D53D23016180001) and by the MIUR Excellence Department Project awarded to Dipartimento di Matematica, Università di Genova, CUP D33C23001110001. This material is based upon work supported by the Air Force Office of Scientific Research under award number FA8655-23-1-7083.

\bibliographystyle{abbrv}
\bibliography{refs}

\begin{thebibliography}{10}

\bibitem{adler-boyle-2017}
A.~Adler and A.~Boyle.
\newblock Electrical impedance tomography: Tissue properties to image measures.
\newblock {\em IEEE Transactions on Biomedical Engineering}, 64(11):2494--2504,
  2017.

\bibitem{alberti-etal-2016}
G.~S. Alberti, H.~Ammari, B.~Jin, J.-K. Seo, and W.~Zhang.
\newblock The linearized inverse problem in multifrequency electrical impedance
  tomography.
\newblock {\em SIAM J. Imaging Sci.}, 9(4):1525--1551, 2016.

\bibitem{Battistel}
A.~Battistel, J.~Wilkie, R.~Chen, and K.~Möller.
\newblock Multifrequency image reconstruction for electrical impedance
  tomography.
\newblock {\em Current Directions in Biomedical Engineering}, 10(4):61--65,
  2024.

\bibitem{BT2003}
A.~Beck and M.~Teboulle.
\newblock Mirror descent and nonlinear projected subgradient methods for convex
  optimization.
\newblock {\em Operations Research Letters}, 31(3):167--175, 2003.

\bibitem{B2008}
T.~Blumensath and M.~E. Davies.
\newblock Gradient pursuit for non-linear sparse signal modelling.
\newblock In {\em 2008 16th European Signal Processing Conference}, pages 1--5,
  2008.

\bibitem{borcea-2002}
L.~Borcea.
\newblock Electrical impedance tomography.
\newblock {\em Inverse Problems}, 18(6):R99--R136, 2002.

\bibitem{MMVNET2023}
Z.~Chen, J.~Xiang, P.-O. Bagnaninchi, and Y.~Yang.
\newblock {MMV-Net: A Multiple Measurement Vector Network for Multifrequency
  Electrical Impedance Tomography}.
\newblock {\em IEEE Transactions on Neural Networks and Learning Systems},
  34(11):8938--8949, 2023.

\bibitem{noser}
M.~Cheney, D.~Isaacson, J.~Newell, S.~Simske, and J.~Goble.
\newblock Noser: An algorithm for solving the inverse conductivity problem.
\newblock {\em Int J Imaging Syst Technol.}, 2(2):66--75, 1990.

\bibitem{cheney-isaacson-newell-1999}
M.~Cheney, D.~Isaacson, and J.~C. Newell.
\newblock Electrical impedance tomography.
\newblock {\em SIAM Rev.}, 41(1):85--101, 1999.

\bibitem{VCherepenin_2001}
V.~Cherepenin, A.~Karpov, A.~Korjenevsky, V.~Kornienko, A.~Mazaletskaya,
  D.~Mazourov, and D.~Meister.
\newblock A 3d electrical impedance tomography (eit) system for breast cancer
  detection.
\newblock {\em Physiological Measurement}, 22(1):9, feb 2001.

\bibitem{cioranescu-donato-1999}
D.~Cioranescu and P.~Donato.
\newblock {\em An introduction to homogenization}, volume~17 of {\em Oxford
  Lecture Series in Mathematics and its Applications}.
\newblock The Clarendon Press, Oxford University Press, New York, 1999.

\bibitem{Col2023DeepplugandplayPG}
F.~Colibazzi, D.~Lazzaro, S.~Morigi, and A.~Samor{\'e}.
\newblock Deep-plug-and-play proximal gauss-newton method with applications to
  nonlinear, ill-posed inverse problems.
\newblock {\em Inverse Problems and Imaging}, 17(6):1226--1248, 2023.

\bibitem{DEAN2008165}
D.~Dean, T.~Ramanathan, D.~Machado, and R.~Sundararajan.
\newblock Electrical impedance spectroscopy study of biological tissues.
\newblock {\em Journal of Electrostatics}, 66(3):165--177, 2008.

\bibitem{fang2024multifrequencyelectricalimpedancetomography}
H.~Fang, Z.~Liu, Y.~Feng, Z.~Qiu, P.~Bagnaninchi, and Y.~Yang.
\newblock Multi-frequency electrical impedance tomography reconstruction with
  multi-branch attention image prior, 2024.

\bibitem{felisi2024full}
A.~Felisi and L.~Rondi.
\newblock Full discretization and regularization for the calder{\'o}n problem.
\newblock {\em Journal of Differential Equations}, 410:513--577, 2024.

\bibitem{frerichs2000}
I.~Frerichs.
\newblock Electrical impedance tomography {(EIT)} in applications related to
  lung and ventilation: a review of experimental and clinical activities.
\newblock {\em Physiological measurement}, 21(2):R1, 2000.

\bibitem{gao2019}
H.~Gao and S.~Ji.
\newblock {Graph U-Nets}.
\newblock In {\em Proceedings of the 36th International Conference on Machine
  Learning}, 2019.

\bibitem{SpectraTissue}
S.~C. Jun, J.~Kuen, J.~Lee, E.~J. Woo, D.~Holder, and J.~K. Seo.
\newblock Frequency-difference eit (fdeit) using weighted difference and
  equivalent homogeneous admittivity: validation by simulation and tank
  experiment.
\newblock {\em Physiol Meas.}, 30(10):1087--1099, 2009.

\bibitem{Adam}
D.~P. Kingma and J.~Ba.
\newblock Adam: A method for stochastic optimization.
\newblock In {\em Proceedings of 3rd International Conference on Learning
  Representations}, 2015.

\bibitem{Lazzaro2024OracleNet}
D.~Lazzaro, S.~Morigi, and L.~Ratti.
\newblock Oracle-net for nonlinear compressed sensing in electrical impedance
  tomography reconstruction problems.
\newblock {\em J Sci Comput}, 101:49, 2024.

\bibitem{lechleiter}
A.~{Lechleiter} and A.~{Rieder}.
\newblock {Newton regularizations for impedance tomography: a numerical study}.
\newblock {\em Inverse Problems}, 22(6):1967--1987, Dec. 2006.

\bibitem{lechleiter2008newton}
A.~Lechleiter and A.~Rieder.
\newblock Newton regularizations for impedance tomography: convergence by local
  injectivity.
\newblock {\em Inverse problems}, 24(6):065009, 2008.

\bibitem{Lee2014}
J.~D. Lee, Y.~Sun, and M.~A. Saunders.
\newblock Proximal newton-type methods for minimizing composite functions.
\newblock {\em SIAM Journal on Optimization}, 24(3):1420--1443, 2014.

\bibitem{Malone2014}
E.~Malone, G.~S. dos Santos, D.~Holder, and S.~Arridge.
\newblock Multifrequency electrical impedance tomography using spectral
  constraints.
\newblock {\em IEEE Transactions on Medical Imaging}, 33(2):340--350, 2014.

\bibitem{malone2015}
E.~Malone, G.~S. dos Santos, D.~Holder, and S.~Arridge.
\newblock A reconstruction-classification method for multifrequency electrical
  impedance tomography.
\newblock {\em IEEE Transactions on Medical Imaging}, 34(7):1486--1497, 2015.

\bibitem{porta2022inexact}
F.~Porta, S.~Villa, M.~Viola, and M.~Zach.
\newblock On the inexact proximal gauss--newton methods for regularized
  nonlinear least squares problems.
\newblock In {\em INdAM Workshop: Advanced Techniques in Optimization for
  Machine learning and Imaging}, pages 151--165. Springer, 2022.

\bibitem{salzo2012convergence}
S.~Salzo and S.~Villa.
\newblock {Convergence analysis of a proximal Gauss-Newton method}.
\newblock {\em Computational Optimization and Applications}, 53:557--589, 2012.

\bibitem{Seo_2008}
J.~K. Seo, J.~Lee, S.~W. Kim, H.~Zribi, and E.~J. Woo.
\newblock Frequency-difference electrical impedance tomography (fdeit):
  algorithm development and feasibility study.
\newblock {\em Physiological Measurement}, 29(8):929, jul 2008.

\bibitem{TanyuNingHauptmannJinMaass+2025+437+470}
D.~N. Tanyu, J.~Ning, A.~Hauptmann, B.~Jin, and P.~Maass.
\newblock {\em Electrical impedance tomography: a fair comparative study on
  deep learning and analytic-based approaches}, pages 437--470.
\newblock De Gruyter, Berlin, Boston, 2025.

\bibitem{TIDSWELL2001283}
T.~Tidswell, A.~Gibson, R.~H. Bayford, and D.~S. Holder.
\newblock Three-dimensional electrical impedance tomography of human brain
  activity.
\newblock {\em NeuroImage}, 13(2):283--294, 2001.

\bibitem{Vetal1999}
P.~Vauhkonen, M.~Vauhkonen, T.~Savolainen, and J.~Kaipio.
\newblock Three-dimensional electrical impedance tomography based on the
  complete electrode model.
\newblock {\em IEEE Transactions on Biomedical Engineering}, 46(9):1150--1160,
  1999.

\bibitem{widlak-scherzer-2012}
T.~Widlak and O.~Scherzer.
\newblock Hybrid tomography for conductivity imaging.
\newblock {\em Inverse Problems}, 28(8):084008, 28, 2012.

\end{thebibliography}

\end{document}